\documentclass{amsart}
\usepackage[utf8]{inputenc}

\title[The growth of the Green function and Poincar\'e series]{The growth of the Green function for random walks and Poincar\'e series}

\date{}
\usepackage{amsmath}
\usepackage{amssymb}

\newcommand\N{\mathbb{N}}
\newcommand\Z{\mathbb{Z}}

\newcommand\R{\mathbb{R}}

\newtheorem{thm}{Theorem}[section]
\newtheorem{prop}[thm]{Proposition}
\newtheorem{cor}[thm]{Corollary}
\newtheorem{lem}[thm]{Lemma}
\newtheorem{conj}[thm]{Conjecture}

\theoremstyle{definition}
\newtheorem{defn}[thm]{Definition}

\theoremstyle{remark}
\newtheorem{quest}[thm]{Question}
\newtheorem{rem}[thm]{Remark}

\newcommand{\showcomments}{yes}

\newsavebox{\commentbox}
{\ifthenelse{\equal{\showcomments}{yes}}%
{\footnotemark
    \begin{lrbox}{\commentbox}
    \begin{minipage}[t]{1.25in}\raggedright\sffamily\upshape\tiny
    \footnotemark[\arabic{footnote}]}%
{\begin{lrbox}{\commentbox}}}%
{\ifthenelse{\equal{\showcomments}{yes}}%
{\end{minipage}\end{lrbox}\marginpar{\usebox{\commentbox}}}%
{\end{lrbox}}}

\newcommand{\pG}{\partial{\Gamma}}
\newcommand{\dirac}[1]{{\mbox{Dirac}}{(#1)}}
\newcommand{\supp}{\mbox{supp}}

\usepackage{xcolor}

\usepackage[colorlinks=true,pdfstartview=FitH,bookmarks=false]{hyperref}
\hypersetup{urlcolor=black,linkcolor=black,citecolor=black,colorlinks=true}

\usepackage{etoolbox}
\apptocmd{\sloppy}{\hbadness 10000\relax}{}{}
\apptocmd{\sloppy}{\vbadness 10000\relax}{}{}
\usepackage{enumerate}

\author[M. Dussaule]{Matthieu Dussaule}
\address{Matthieu Dussaule : Cogitamus Laboratory \\ France}
\email{matthieu.dussaule@hotmail.fr}

\author{Longmin Wang}
\address{Longmin Wang : School of Statistics and Data Science, LPMC, Nankai University, 94 Weijin Road, Nankai District, Tianjin, China }
\email{wanglm@nankai.edu.cn}

\author{Wenyuan Yang}
\address{Wenyuan Yang : Beijing International Center for Mathematical Research (BICMR), Beijing University, No. 5 Yiheyuan Road, Haidian District, Beijing, China}
\email{yabziz@gmail.com}

\numberwithin{equation}{section}

\thanks{L.W. is supported by National Natural Science Foundation of China (12171252).  W.Y. is supported by National Key R \& D Program of China (SQ2020YFA070059).}

\begin{document}

\begin{abstract}
Given a probability measure $\mu$ on a finitely generated group $\Gamma$, the Green function $G(x,y|r)$ encodes many properties of the random walk associated with $\mu$.
%\ywy{The Green function $G(e,x|r)$ encodes many properties of the $\mu$-random walk on a finitely generated group $\Gamma$.}
Finding asymptotics of $G(x,y|r)$ as $y$ goes to infinity is a common thread in probability theory and is usually referred as renewal theory in literature. %\ywy{Finding the asymptotics of $G(e,x|r)$ is a common thread in probability theory and is sometimes referred as renewal theory in literature.}
%Endowing $\Gamma$ with a word distance, one can compute the sum $H_r(n)$ of the Green function along the sphere of radius $n$.
%Computing the asymptotics in $n$ of $H_r(n)$ goes a bit further than renewal theory and has not been investigated much, although this quantity appears naturally when studying branching random walks driven by $\mu$ on $\Gamma$.
Endowing $\Gamma$ with a word distance, we denote by $H_r(n)$ the sum of the Green function $G(e,x|r)$ along the sphere of radius $n$. This quantity appears naturally when studying asymptotic properties of branching random walks driven by $\mu$ on $\Gamma$ and the behavior of $H_r(n)$ as $n$ goes to infinity is intimately related to renewal theory.
Our motivation in this paper is to construct various examples of particular behaviors for $H_r(n)$.
%Besides giving several examples of particular behavior for $H_r(n)$ for various classes of finitely generated groups, we answer some questions raised in \cite{DWY} about branching random walks on relatively hyperbolic groups.
First, our main result exhibits a class of relatively hyperbolic groups with convergent Poincar\'e series generated by $H_r(n)$, which  answers some questions raised in \cite{DWY}.  Along the way, we  investigate    the behavior of $H_r(n)$ for several classes of finitely generated groups, including abelian groups, certain nilpotent groups, lamplighter groups, and Cartesian products of free groups.
%\ywy{I know little about classical renewal theory. It appears to me that the first few sentences should be mentioned in the introduction, and we should give more concrete description of the results proved in this paper. }{\color{blue}This is done at the beginning of the second paragraph on page 2. Please check if you agree with the formulation and feel free to edit.}\ywy{Following your idea, I suggest to rewrite a few sentences as above. Please feel free to edit it.   }
\end{abstract}

\maketitle
\section{Introduction}

%{\color{blue}I took into account Wenyuan's comments and rewrote parts of the introduction.
%I'm satisfied now, but if you want to change things feel free to edit the text.
%I hesitated to give more details in the last part of the introduction about the example on the bi-tree.
%In particular, I first wanted to explain a bit how the competition between $T_1$ and $T_2$ raises a phase transition when $l_1\neq l_2$, but I don't know how to write this.
%If someone wants to do it, this is cool, but otherwise, I don't think that it is truly necessary.
%Also, I started the abstract, but didn't finish. I can do more later, but don't hesitate to take over.
%}

Let $\Gamma$ be a finitely generated group endowed with a finite generating set.
Denote by $|\!\cdot\!|$ the induced word distance and by $S_n = \left\{ x \in \Gamma \colon |x| =n \right\}$ the sphere of radius $n$ centered at the neutral element $e$ of $\Gamma$.
We set
$$v=\limsup_{n\to\infty} \frac{1}{n}\log \sharp S_n.$$
The number $v$ is called the volume growth of $\Gamma$ with respect to the chosen generating set.
Note that $v$ is also the growth rate of balls, i.e.\ $v=\limsup_{n\to\infty} \frac{1}{n}\log \sharp B_n$, where $B_n$ is the ball of radius $n$ centered at $e$, see \cite[Section~2.2]{Mann} for more details.
The group $\Gamma$ has exponential growth (respectively sub-exponential growth) if $v>0$ (respectively $v=0$) and this does not depend on the choice of the generating set, although the exact value of $v$ does, when non-zero.

Consider an admissible probability measure $\mu$ on $\Gamma$. % with finite support.
Let $p_n(x, y)$ be the $n$-step transition probability for the random walk $(X_k)$ with step distribution $\mu$.  Define for $0 < r \leq R$ the Green function
\[
  G(x, y | r) = \sum_{n = 0}^{\infty} r^n p_n(x, y), 
\]
where $R$ is the radius of convergence of the series, which is also the inverse of the spectral radius of the random walk.
When $r = 1$ we write $G(x, y) = G(x, y | 1)$ for simplicity. 
Note that we have $G(x, y | r) = G(e, e | r) F(x, y | r)$, where $F(x, y | r)$ is the first return Green function \cite[Lemma~1.13]{Woessbook}. 
Define the \textit{growth function} associated with the Green function $H_r(n)$ as
\[
  H_r(n) = \sum_{x \in S_n} G(e, x | r)
\]
and  the \textit{growth rate of the Green function} $\omega_\Gamma(r)$ as
\[
  \omega_{\Gamma}(r) = \limsup_{n \to \infty} \frac{1}{n} \log H_r(n). 
\]

We will always consider random walks that are transient at the spectral radius, i.e.\ $G(x,y|r)$ is finite for every $x,y\in \Gamma$ and for every $r\leq R$.
By Varapoulos Theorem \cite[Theorem~7.8]{Woessbook}, only groups with at most quadratic growth can carry a random walk which is not transient at the spectral radius.
In other words, we will always assume that $\Gamma$ is not virtually $\Z^d$, $d\leq 2$.
In particular, this ensures that $H_r(n)$ and $\omega_\Gamma(r)$ are well defined for every $r\leq R$.

\medskip
\textbf{Notations.} In all the paper, given two functions $f$ and $g$, we write $f\simeq g$ if the difference between $f$ and $g$ is bounded from above and below, that is $|f-g|\leq C$ for some constant $C$.
We write $f\simeq_L g$ if the constant $C$ depends on a parameter $L$.
Assuming further that $f$ and $g$ are positive, we write $f\asymp g$ if the ratio of $f$ and $g$ is bounded from above and below, that is
$\frac{1}{C}f\leq g\leq Cf$
for some positive constant $C$.
Similarly, if $C$ depends on $L$, we write $f\asymp_L g$.
Also, if $f\leq Cg$, we write $f\prec g$ and $f\prec_L g$ if $C$ depends on $L$.
If the dependency is not clear from the context, we will avoid these notations.

\medskip
This paper mostly deals with the asymptotics of $H_r(n)$ as $n$ goes to infinity.
A fruitful line of research in the theory of random walks is to compute asymptotics in space of the Green function, that is asymptotics of
$G(e,x|r)$ as $x$ goes to infinity.
This is referred as renewal theory and goes back to Blackwell's renewal theorem
for drifted random walks on $\mathbb R$, see \cite{Blackwell} and earlier references therein.
The terminology renewal comes from an interpretation of the Green function $G(e,x)$ as the probability that a renewal event takes place at $x$ for a suited process, see \cite[Chapter~II.9]{Spitzer}.

In fact, the terminology renewal theory is used in a much broader setting and we refer to \cite[Chapter~XI]{Feller} for a more complete exposition within the scope of probability theory.
Let us also mention that Lalley \cite{Lalleyrenewal} generalized classical renewal theory, with a new approach to deal with asymptotics of certain counting functions arising in geometric group theory.
This led to significant research in dynamical systems, where renewal theory is now a common thread.

Finally, note that Ledrappier interpreted the computation of the limit of $H_1(n)$ on the free group as a renewal theorem for the distance \cite{Ledrappier01}, see also \cite{Ledrappier93} for related results concerning the Brownian motion on the universal cover of compact negatively curved manifolds.
However, besides such specific examples, the behaviors of $H_r(n)$ and $\omega_\Gamma(r)$ have not been investigated much in literature, so let us first explain in which context these quantities occur.

\medskip
Consider a probability measure $\nu$ on $\Z_{\geq 0}=\{0,1,2,...\}$.
The branching random walk driven by $\nu$ and $\mu$ on $\Gamma$, denoted by $\mathrm{BRW}(\Gamma,\nu,\mu)$ is described as follows.
One starts with a single particle at $e$.
For every $n$,  every alive particle at time $n$ dies after giving birth to an independent number of children, according to the distribution of $\nu$, each of which independently moves on $\Gamma$ according to the distribution of $\mu$.
The measure $\nu$ is called the offspring distribution.

In recent years, there has been a large body of work dedicated to understanding the asymptotic behavior of $\mathrm{BRW}(\Gamma,\nu,\mu)$ in terms of geometric features of $\Gamma$.
For such study, one usually condition on non-extinction of the system, which boils down to considering an offspring measure distributed on $\N=\{1,2,...\}$, see \cite[Chapter~1]{AN04}.
It is thus natural to assume that $\mathbf E[\nu]>1$, otherwise, conditioned on non-extinction, $\nu$ is the Dirac measure at 1 and the branching random walk is nothing but the usual random walk whose step distribution is given by $\mu$.
One of the cornerstone result is that letting $\mathbf E[\nu]=r$, then $1<r\leq R$ if and only if the branching random walk is transient, i.e.\ almost surely it does not visit every vertex infinitely many times, see \cite{BP94} and \cite{GM06}.
%{\color{blue}If $r=1$, then $\nu=\delta_1$ and so we are considering a classical random walk, which is transient by assumption (see above).
%We do not need to assume that $1<r$ here.
%But I added a sentence explaining why we assume that $r>1$ and a sentence below explaining it still matters to consider the case $r\leq 1$. Is it better ?}\ywy{Yes, this is great!}
Furthermore, if $1<r\leq R$, then the branching random walk has exponential volume growth by \cite{BM12}.
Precisely, letting $M_n$ be the cardinality of the number of elements of $S_n$ that are ever visited by the branching random walk, we have that almost surely,
$$\limsup_{n\to\infty} \frac{1}{n}\log M_n>0.$$
In some cases, such as hyperbolic groups \cite[Theorem~1.1]{SWX} and relatively hyperbolic groups \cite[Theorem~1.1]{DWY}, it was shown that this growth rate coincides almost surely with the growth rate of the Green function at $1<r\le R$.
That is,
$$\limsup_{n\to\infty} \frac{1}{n}\log M_n=\omega_\Gamma(r).$$
Moreover, for hyperbolic groups, we have $H_r(n)\asymp \mathrm{e}^{n\omega_\Gamma(r)}$, i.e.\ the Green function has purely exponential growth, see \cite[Theorem~3.1]{SWX}.

%In the context of relatively hyperbolic groups, letting $A$ be a maximal parabolic subgroup, the behavior of $\omega_A(r)$ is of particular importance in understanding the behavior of $\omega_\Gamma(r)$.
%We will provide more details in the last two sections of the paper.
% \ywy{We should be more careful with the distinction $r<1$ and $r>1$ in the above paragraph, as we wish to bring this to reader.  }

\medskip
In fact, our original motivation for this paper was  to answer some questions raised in \cite{DWY} about branching random walks on relatively hyperbolic groups.
We introduce the following Poincar\'e series associated with $\mu$ on $\Gamma$.
For $r\leq R$ and $s\in \R$, we set
\begin{align}\label{PoincareSeriesEQ}
\Theta_\Gamma(r,s)=\sum_{x\in \Gamma}G(e,x|r)\mathrm{e}^{-sd(e,x)}=\sum_{n\geq 0}H_r(n)\mathrm{e}^{-sn}.    
\end{align}
The growth rate $\omega_\Gamma(r)$ is thus the critical exponent of this Poincar\'e series, i.e.\ for $s<\omega_\Gamma(r)$, $\Theta_\Gamma(r,s)$ diverges and for $s>\omega_\Gamma(r)$, $\Theta_\Gamma(r,s)$ converges.

Several criteria were found in \cite{DWY} to ensure that this Poincar\'e series diverges at $s=\omega_\Gamma(r)$.
One of the missing pieces was whether  there exists an example with convergent Poincar\'e series.
Our main result exhibits such an example by actually constructing a relatively hyperbolic group for which divergence of the Poincar\'e series depends on $r$.
Precisely, we prove the following.
%\begin{thm}\label{mainthm}
%There exists a relatively hyperbolic group $\Gamma$, endowed with a finitely supported symmetric admissible probability measure $\mu$ such that the Poincar\'e series $\Theta_\Gamma(r,s)$ is convergent at $s=\omega_\Gamma(r)$ for some $r\leq R$.
%\end{thm}

%\ywy{I would prefer the following complete formulation instead of mentioning it as a remark.}
\begin{thm}[Theorem~\ref{transitionphasePoincare}]\label{mainthm}
There exists a relatively hyperbolic group $\Gamma$, endowed with a finitely supported symmetric admissible probability measure $\mu$, and there exist $1<r_*<r_\sharp<R$ such that  
\begin{enumerate}
    \item 
    for any $r\leq r_*$, $\Theta_\Gamma(r,\omega_\Gamma(r))$ diverges,
    \item
    at $r=r_\sharp$, $\Theta_\Gamma(r,\omega_\Gamma(r))$ converges.
\end{enumerate}
\end{thm}
\begin{rem}
We are unable to determine whether the second assertion is true for any $r>r_*$. If this were true, we would obtain a phase transition for the divergence of the Poincare series and also for all the   assertions in Corollary \ref{transitionphaseCor} below.
However, up to changing the measure $\mu$, we have a weak form of a phase transition at the level of the parabolic subgroups, see Remark~\ref{remarkchoiceofr1} for more details.
\end{rem}

We are now moving on explaining several applications of Theorem~\ref{mainthm}. Before that, let us put into context the series  $\Theta_\Gamma(r,s)$ under consideration.
The question whether $\Theta_\Gamma(r,s)$ converges or diverges at $s=\omega_\Gamma(r)$ is of particular importance for many properties such as  the construction of boundary measures and growth problems.  It is also related to the parabolic gap property coined in \cite{DWY} that we discuss next.

Consider a relatively hyperbolic group $\Gamma$ and a maximal parabolic subgroup $P$.
The \textit{growth rate of the Green function} induced on $P$ is defined as
$$\omega_P(r)=\limsup_{n\to\infty}\frac{1}{n}\log \sum_{\underset{|x|=n}{x\in P}}G(e,x|r). $$
Following \cite{DWY}, we say  that $\Gamma$ has a {parabolic gap along $P$ for the Green function} at $r\in[1,R]$ if $\omega_P(r)<\omega_\Gamma(r)$. If this holds along every parabolic subgroups, we say $\Gamma$ has a {parabolic gap for the Green function} at $r$ and if this in turn holds for every $r\in [1,R]$ we say that $\Gamma$ has a \textit{parabolic gap for the Green function}.

This notion is analogous to the parabolic gap condition introduced in \cite{DOP}, where first examples of convergent (standard) Poincar\'e series were constructed without this parabolic gap condition. As an interesting consequence, the authors of \cite{DOP} produced Patterson-Sullivan measures having atoms at parabolic points in the visual boundary of Hadamard manifolds.

In \cite{DWY}, we proved that if there exists a relatively hyperbolic group with convergent Poincar\'e series, then the parabolic gap condition fails.
On the other hand, having a parabolic gap has various applications related to asymptotic properties of branching random walks, see in particular \cite[Theorem~1.8, Remark~5.4]{DWY}.

Motivated by this discussion and by the work of \cite{DOP}, we make use of the above theorem to further develop various properties with a similar behavior.
We introduce a family of Patterson-Sullivan type measures $\nu_e(r)$ associated with the Poincar\'e series defined in~(\ref{PoincareSeriesEQ}) and a family of proper distances $\mathfrak d_r$ which are quasi-isometric to the word distance.
We summarize here the different results we obtain, see Theorem~\ref{transitionphasePoincare}, Corollary~\ref{transitionphasePattersonSullivan} and Theorem~\ref{transitionphasegrowthtightness}.
\begin{cor}\label{transitionphaseCor}
The   pair $(\Gamma,\mu)$   in Theorem~\ref{mainthm} has the following properties: 
\begin{enumerate}
\item  the parabolic gap for the Green functions holds for $r\le r_*$ but fails at $r=r_\sharp$,
\item  the Patterson-Sullivan measure $\nu_e(r)$ is supported on conical limit points for $r\le r_*$ and is purely atomic and supported on parabolic limit points at $r=r_\sharp$,
\item
the growth tightness for the proper distance $\mathfrak d_r$ holds for $r\le r_*$  but fails at $r=r_\sharp$.     
\end{enumerate}    
\end{cor}
\begin{rem}
In \cite{Peigne11}, Peign\'e constructed a divergent Schottky group without the parabolic gap condition. In view of our examples, it is relevant to ask  here  whether there exist examples of divergent Poincar\'e series (\ref{PoincareSeriesEQ})  without a parabolic gap for the Green function. 
\end{rem}
\begin{rem}
The proper metric in (3) is defined by
$$\mathfrak d_r(x,y):=\omega_\Gamma(r)|x^{-1}y|+|x^{-1}y|_r$$
so it is a linear combination of the word metric $|x^{-1}y|$ and the $r$-Green metric $|x^{-1}y|_r:=-\log \frac{G(x,y|r)}{G(e,e|r)}$. Cashen-Tao \cite{CaTao16} have shown examples   of product groups with growth tightness for one generating set but not for another generating set. The above examples within the class of relatively hyperbolic groups are new. 
\end{rem}

%We will actually exhibit a phase transition for the couple $(\Gamma,\mu)$ we construct and show that there exists $r_*$ such that for $r\leq r_*$, $\Theta_\Gamma(r,\omega_\Gamma(r))$ diverges and for $r>r_*$, $\Theta_\Gamma(r,\omega_\Gamma(r)$ converges, see Section~\ref{SectionRHG} for more details and in particular Theorem~\ref{transitionphasePoincare}.

These theorems are the conclusion of several results that we prove along the paper.
If we restrict our attention to branching random walks, the study of $H_r(n)$ and $\omega_\Gamma(r)$ seems to be relevant only for $r>1$.
However, as we now observe, even for $r\leq 1$, both these quantities appear naturally and are worth being studied.

For any subset $A$  of $\Gamma$, the growth rate $\omega_A(r)$ of the Green function restricted to $A$ at $r$ can be defined similarly to $\omega_\Gamma(r)$.
Namely, we set
\[
  H_{A,r}(n) = \sum_{x \in S_n\cap A} G(e, x | r)
\]
and
\[
  \omega_{A}(r) = \limsup_{n \to \infty} \frac{1}{n} \log H_{A,r}(n). 
\]
% \ywy{For me, the following paragraph could be removed, as the induced return kernel is not necessary to introduce our main results.}
%{\color{blue}I kind of agree, but on the other hand, this is a good place to introduce properly in a general setting the first return kernel and it explains why $\omega_\Gamma(r)$ can arise in nature, for $r<1$.
%I prefer to keep it there, unless this is really a problem.
%Feel free to edit nonetheless !}\ywy{Sorry, I realize now its relation with  $\omega_A(r)$. Let us keep it.}
We can  interpret  $\omega_A(r)$ via      the first return kernel to   $A$ associated with $r\mu$, which is defined by
$$p_{r,A}(x,y)=\sum_{n\geq 1}\sum_{z_1,...,z_{n-1}\notin A}r^np(x,z_1)p(z_1,z_2)...p(z_{n-1},y), \ x,y\in A.$$
Letting $G_{r,A}$ denote the corresponding Green function, it holds that for every $x,y$ in $A$,
$$G_{r,A}(x,y|1)=G(x,y|r)$$ by \cite[Lemma~4.4]{DG21}.
Consequently, the growth rate $\omega_A(r)$  coincides with the growth rate at 1 of this new Green function $G_{r,A}$.
In general, there is no reason for $p_{r,A}$ to be a Markov transition kernel, that is, $\sum_{y\in A}p_{r,A}(x,y)$ might  not be a constant equal to 1.
When $A$ is a subgroup of $\Gamma$, then $p_{r,A}$ is $A$-invariant, so $\sum_{y\in A}p_{r,A}(x,y)$ is independent of $x$, but there is still no reason for this transition kernel to be Markov.
As a matter of fact, in various interesting cases, $p_{r,A}$ is a sub-Markov transition kernel.
Letting $t$ be its total mass and setting $\tilde{p}_{r,A}=t^{-1}p_{r,A}$ and $\widetilde{G}_{r,A}$ for the corresponding Green function, we see that
$G_{r,A}(\cdot,\cdot|1)=\widetilde{G}_{r,A}(\cdot,\cdot|t)$.
Therefore, the growth rate $\omega_A(r)$ restricted to $A$ at $r$ coincides in this case with the growth rate at some $t<1$ of a Markov transition kernel on $A$.
As explained above, in the context of relatively hyperbolic groups, letting $A$ be a maximal parabolic subgroup, the relation between $\omega_A(r)$ and $\omega_\Gamma(r)$ is of particular importance in understanding the asymptotic behavior of $H_r(n)$. 
In particular, this discussion  motivates the study of $H_r(n)$ and $\omega_\Gamma(r)$ also for $r\leq 1$.

\medskip
In fact, a large part of our work is devoted to understanding the behavior of $H_1(n)$ for symmetric admissible and finitely supported random walks on amenable groups and our study goes beyond applications to relatively hyperbolic groups.
Here is a summary of the different results we obtain.

%\ywy{How about we put the first return kernel here? The above discussion seems just justify the study of $H_1(n)$, and its continuity at $r=1$, etc..}
\begin{thm}[Theorem~\ref{thmnilpotent}, Theorem~\ref{thmlamplighter}]\label{thmamenable}
The function $H_1(n)$ is asymptotically linear in $n$ for the following classes of groups:
\begin{enumerate}
    \item if $\Gamma$ is a free abelian group of finite rank endowed  with a finitely supported symmetric admissible probability measure, then $H_1(n)\sim Cn$ for any finite generating set,
    \item if $\Gamma$ is a finitely generated nilpotent group of nilpotency class $N_\Gamma\leq 2$ endowed with a finitely supported symmetric admissible probability measure, then $H_1(n)\asymp n$ for any finite generating set,
    \item
    if $\Gamma$ is a lamplighter group endowed with the set of generators considered in \cite{BrofferioWoess}, then for the simple random walk, $H_1(n)\sim Cn$.
\end{enumerate}
\end{thm}

Finally, elaborating on the work of Picardello and Woess \cite{PicardelloWoessbitree}, we prove the following phase transition result for random walks on bi-trees.
Consider the Cartesian product of two regular trees $T_1$ and $T_2$ of respective degree $l_1$ and $l_2$.
Let $\mu_i$ be finitely supported admissible symmetric probability measures on $T_i$ and set $\mu=\alpha \mu_1 +(1-\alpha)\mu_2$ for $0<\alpha<1$.
\begin{thm}[Theorem~\ref{thmbitree}]\label{mainthmbitree}
The probability measure $\mu$ on $T_1\times T_2$ satisfies the following.
If $l_1=l_2$, then for every fixed $r<R$,
$H_r(n)\asymp \mathrm{e}^{n\omega_\Gamma(r)}$ for all $n\ge 1$.
If $l_1>l_2$, then there exists a phase transition at some $r_0\in (1,R)$ such that the following holds.
\begin{enumerate}
    \item For every $r<r_0$, we have $H_r(n)\asymp \mathrm{e}^{n\omega_\Gamma(r)}$ for all $n\ge 1$.
    \item At $r=r_0$, we have $H_r(n)\asymp n^{-1}\mathrm{e}^{n\omega_\Gamma(r)}$ for all $n\ge 1$.
    \item For every $r_0<r<R$, we have $H_r(n)\asymp n^{-3/2}\mathrm{e}^{n\omega_\Gamma(r)}$  for all $n\ge 1$.
\end{enumerate}
\end{thm}
This theorem is a pivotal result in our paper. It is the conclusion of our general study of $H_r(n)$ on finitely generated group and is also one of the main pieces in proving Theorem~\ref{mainthm}.
%This is kind of thing in my mind to summarize those estimates. Feel free  to modify it or suggest other ways.
%}
%{\color{blue}Yes this is great. I gave a slightly more precise result. And I would call this a theorem.}

\medskip
Let us now give more details on how the paper is organized.
It has seven sections including the introduction.
In the first half which consists of Sections~\ref{Sectiongrowthrateat1},~\ref{Sectionabeliannilpotent},~\ref{Sectionlamplighter} and~\ref{Sectionbitree}, we give general statements on $H_r(n)$ and $\omega_\Gamma(r)$, as well as various examples of different situations.
The second half focuses on applications to relatively hyperbolic groups and is divided into Sections~\ref{Sectionpattersonsullivan} and~\ref{SectionRHG}.
It is based both on the first half and on previous results of \cite{DWY}.

\medskip
More specifically,
in Section~\ref{Sectiongrowthrateat1}, we consider the growth rate $\omega_\Gamma(r)$ at the special value $r=1$.
We prove that $\omega_\Gamma(1)=0$ and study further continuity of the function $r\mapsto \omega_\Gamma(r)$ at $1$.

We then study the asymptotic behavior of $H_1(n)$ on amenable groups $\Gamma$ and we prove Theorem~\ref{thmamenable} in Sections~\ref{Sectionabeliannilpotent} and~\ref{Sectionlamplighter}.
We show in particular that $H_1(n)$ is asymptotically linear in $n$ for abelian groups and provide partial results for nilpotent groups, see Theorem~\ref{thmnilpotent} and Conjecture~\ref{conjecturenilpotent}.
This raises the question whether $H_1(n)$ always behaves linearly in amenable groups.
We prove that it is the case for the simple random walk on the lamplighter group in Theorem~\ref{thmlamplighter}.
The main strategy in these two sections is basically to combine previously known renewal results, i.e.\ about the asymptotics in space of the Green function and known results on the large scale properties of the spheres $S_n$.

Then, in Section~\ref{Sectionbitree}, we prove Theorem~\ref{mainthmbitree}.
The behavior of $H_r(n)$ that we exhibit is of an again different form than the one described above, for (relatively) hyperbolic groups and for amenable groups respectively.
As a particular consequence, we see that the quantity $H_r(n)$ may fail to be sub-multiplicative.
This disproves an argument due to Candellero, Gilch and M\"uller \cite{CandelleroGilchMuller} as well as its consequences, see Remark~\ref{remarkCGM} for more details.

\medskip
In Section~\ref{Sectionpattersonsullivan}, we focus on relatively hyperbolic groups.
We introduce a proper distance $\mathfrak d_r$ on $\Gamma$, $1\leq r\leq R$, and its associated Poincar\'e series $\mathcal{P}(s)$.
For $r<R$, $\mathfrak d_r$ is quasi-isometric to the word distance.
The critical exponent of $\mathcal{P}(s)$ is 1 and
we prove that divergence, respectively convergence of $\mathcal{P}(s)$ at 1 is equivalent to divergence, respectively convergence of the
Poincar\'e series $\Theta_\Gamma(r,s)$ at $s=\omega_\Gamma(r)$.
Then, we use $\mathcal{P}(s)$ to construct Patterson-Sullivan type measures $\nu_x$ on the Bowditch boundary of the group.
This allows us to find a characterization of divergence of $\Theta_\Gamma(r,s)$ at the growth rate $\omega_\Gamma(r)$ in terms of the support of the measures $\nu_x$, see in particular Theorem~\ref{divergencetypeandatoms}.

We then prove Theorem~\ref{mainthm} in Section~\ref{SectionRHG} and answer some questions raised in \cite{DWY}.
In particular, we show that whenever maximal parabolic subgroups are amenable, $\Gamma$ necessarily has a parabolic gap for the Green function, see Corollary~\ref{coramenableparabolic}.
Along the way, we study further the distance $\mathfrak d_r$ and prove that the parabolic gap for the Green function is equivalent to growth tightness for the distance $\mathfrak d_r$.
In particular, we find that for suited $r$, $\Gamma$ is not growth tight for $\mathfrak d_r$, which allows us to partially answer a question raised in \cite{ACT}, see Corollary~\ref{corgrowthtightness}.

\section{The growth rate of the Green function at $r=1$}\label{Sectiongrowthrateat1}
We study $\omega_\Gamma(1)$ in this section.
We first show that $\omega_\Gamma(r)$ always vanishes at 1 and then study continuity at this special value.

\subsection{Nullity of the growth rate at $r=1$}
In this section, we do not need to assume that the random walk is symmetric, nor that it is finitely supported.
We start with the following lower bound.
\begin{lem}\label{lowerboundgrowthrateat1}
Let $\Gamma$ be a finitely generated group and 
$\mu$ an admissible probability measure on $\Gamma$.  Then,
$\omega_\Gamma(1)\geq 0$.
\end{lem}

\begin{proof}
Since the series $\sum_{n = 0}^{\infty} H_1(n) = \sum_{k=0}^{\infty} \sum_{x \in \Gamma} p_k(e, x) = \sum_{k=0}^{\infty} 1 = \infty$, we have that
\begin{equation*}
\omega_{\Gamma}(1) = \limsup_{n \to \infty} \frac{1}{n} \log H_1(n) \geq 0. \qedhere
\end{equation*}
\end{proof}

% \begin{proof}
%The proof is adapted from \cite[Note~1.7]{GouezelLalley}.
\begin{rem}
If $\mu$ is finitely supported, then there is a constant $C > 0$ such that $H_n(1) \geq C$ for all $n \geq 1$.  In fact,
since the random walk driven by $\mu$ is transient and has bounded jumps at most $C_0>0$, it will eventually visit the annulus 
$$A(n,n+C_0)=\{x,n\leq |x|\leq n+C_0\}.$$
Note that
\begin{align*}
    \sum_{x\in A(n,n+C_0)}G(e,x)&\ = \sum_{x\in A(n,n+C_0)}\sum_{k\geq0}p_k(e,x)\\
    &\ =\sum_{x\in A(n,n+C_0)}\sum_{k\geq0}\mathbf P(X_k=x).
\end{align*}
Thus,
\begin{align*}
  \sum_{x\in A(n,n+C_0)}G(e,x)   & \ \geq \mathbf{P}(\text{the random walk ever visits }A(n,n+C_0)) \\
    &\ \geq 1.
\end{align*}
Now by \cite[Lemma~3.1~(1)]{DWY}, $H_1(n)\asymp H_1(n+1)$, hence
$$\sum_{x\in A(n,n+C_0)}G(e,x)\asymp H_1(n),$$
which concludes the proof that $H_n(1) \geq C$.
% \end{proof}
\end{rem}

%For the following result, we do not have either to assume that $\mu$ is finitely supported.

\begin{prop}\label{upperboundgrowthrateat1}
Let $\Gamma$ be a finitely generated group and $\mu$ be an admissible probability measure on $\Gamma$.
Then $\omega_{\Gamma}(1) \leq 0$.
 Moreover, if $\Gamma$ has exponential growth, there exists $C>0$ such that $H_1(n)\leq C n^{3}$.
\end{prop}

\begin{proof}
Note that
  \[
    H_1(n) \leq G(e, e) \sharp S_n,
  \]
  hence, 
  if the volume growth rate
  \[
    v = \limsup_{n \to \infty} \frac{1}{n} \log \sharp S_n = 0, 
  \]
then $\omega_{\Gamma}(1) \leq 0$.
%By Lemma~\ref{lowerboundgrowthrateat1}, we thus have $\omega_\Gamma(1)=0$.

  We assume that $v > 0$ in the remainder of the proof.  % Let $S = \mathrm{supp} \mu$ and let $|\!\cdot\!|_{\mu}$ be the graph distance of the Cayley graph $G(\Gamma, S)$.  Since $S$ is finite, there exists constant $c_1 > 0$ such that $c_1^{-1} |x|_{\mu} \leq |x| \leq c_1 |x|_{\mu}$ for all $x \in \Gamma$.  It follows that there exist $c_2 > 0$ and $v_1$, $v_2 > 0$ such that for all $n \geq 0$, 
%   \[
%     c_2^{-1} \mathrm{e}^{v_1 n} \leq V(n) \leq c_2 \mathrm{e}^{v_2 n},   
%   \]
%   where $V(n)$ is the cardinality of the ball of radius $n$ with respect to the distance $|\!\cdot\!|_{\mu}$.  Let
%   \[
%     \phi(n) = \inf \left\{ r \colon V(r) \geq n \right\}
%   \]
%   be the inverse growth rate.  Then $n \geq V(\phi(n) - 1) \geq c_1^{-1} \mathrm{e}^{v_1 (\phi(n) - 1)}$ and hence
%   \[
%     \phi(n) \leq 1 + v_1^{-1} \log \left( c_2 n \right). 
%   \]
%   Define the expansion profile by
% \[
%   \Phi(u) = \inf \left\{ \frac{p(K, \, K^c)}{|K|} \colon 0 < |K| \leq u \right\},  
% \]
% where $p(K, \, A) = \sum_{x \in K, \, a \in A} p(x, \, a)$.  
% Let $c_0 = \min \left\{ \mu(x) \colon x \in S \right\}$.  Then we have that
% \[
%   p(K, \, K^c) \geq c_0 \left| \partial_V^{\mathrm{in}} K \right|,  
% \]
% where $\partial_V^{\mathrm{in}} K = \left\{ v \in K \colon v \gamma \not\in K \text{ for some } \gamma \in S \right\}$.  
% By~\cite{CSC93}, 
% \begin{equation}
%   \label{e:phir}
%   \Phi(u) \geq \frac{c_3}{2 \phi \left( 2 u \right)} \geq \frac{c_4}{\log \left( c_5 u \right)}. 
% \end{equation}
% By~\cite{MP05} (see also~\cite[Corollary 6.32]{LP16}), we have
Then there exist $c_1 > 1$ and $v_1$, $v_2 > 0$ such that
    \begin{equation}
      \label{hksubexp}
    c_1^{-1} \mathrm{e}^{v_1 n} \leq \sharp B_n \leq c_1 \mathrm{e}^{v_2 n},
  \end{equation}
  where $\sharp B_n$ is the cardinality of the ball of radius $n$.  By~\cite[Theorem 1]{V91}, 
\[
  p_n(x, y) \leq c_2 \mathrm{e}^{- c_3 n^{1/3}}
\]
for some $c_2$, $c_3 > 0$.  Choose $c_4$ so that $c_4^{1/3} c_3 > v_2$.  Then
\begin{align*}
  H_1(n) &\ \leq \sum_{k = 0}^{c_4 n^3} \sum_{x \in S_n} p_k(e, x) + c_2\sum_{x \in S_n} \sum_{k= c_4 n^3+1}^{\infty} c_3 \mathrm{e}^{- c_3 k^{1/3}}\\
 &\ \leq c_4n^3 + 1 + c_5 \sum_{k = c_4 n^3+1}^{\infty} \mathrm{e}^{v_2 n} \mathrm{e}^{- c_3 k^{1/3}}. 
\end{align*}

Note that
\[
  \sum_{k = c_4 n^3+1}^{\infty} \mathrm{e}^{- c_3 k^{1/3}} \leq \int_{c_4 n^3}^{\infty} \mathrm{e}^{-c_3 t^{1/3}} \mathrm{d} t \leq c_6 n^{2} \mathrm{e}^{- c_3 c_4^{1/3} n}.  
\]
Thus we have
\[
  H_1(n) \leq c_4 n^3 + 1+c_7 n^{2} \mathrm{e}^{\left( v_2 - c_3 c_4^{1/3} \right) n} \leq c_8 n^3 % {\color{blue}\leq c_{11} n^4}
\]
and hence $\omega_{\Gamma}(1) \leq 0$.
%By Lemma~\ref{lowerboundgrowthrateat1}, we again find that $\omega_\Gamma(1)=0$.
\end{proof}

Lemma~\ref{lowerboundgrowthrateat1} and Proposition~\ref{upperboundgrowthrateat1} together yield the following corollary.
\begin{cor}
  Let $\Gamma$ be a finitely generated group and let $\mu$ be an %finitely supported
  admissible probability measure on $\Gamma$.
Then, $\omega_\Gamma(1)=0$.
\end{cor}

We also have the following result.
\begin{prop}\label{omeganonincreasing}
Let $\Gamma$ be a finitely generated group and let $\mu$ be an
  admissible probability measure on $\Gamma$.
Then, $r\mapsto\omega_\Gamma(r)$ is monotonically non-decreasing.
\end{prop}

\begin{proof}
Since the Green function itself is non-decreasing in $r$, $H_n(s)\leq H_n(r)$, hence $\omega_\Gamma(s)\leq \omega_\Gamma(r)$ if $s\leq r$.
\end{proof}

\subsection{Continuity of the growth rate}
We start with the following result which holds for every finitely generated group, without assuming that the random walk is finitely supported.
We assume however that it is symmetric.

  \begin{prop}\label{omegaGcont}
Let $\Gamma$ be a finitely generated group and let $\mu$ be a symmetric admissible probability measure on $\Gamma$.
Then, the function $\omega_{\Gamma}(r)$ is continuous for $0 < r < R$.  
  \end{prop}

%{\color{blue}I moved above the non-decreaing part, as it does not require symmetry.}
%% \lw{To ensure that $\omega_{\Gamma}(r)$ is strictly increasing for $0 < r < R$, we need to assume that $\mu$ is finitely supported.  The proof is exactly the same as that of~\cite[Lemma 3.1]{DWY}.  If $\mu$ is super-exponential, then we can prove that $\omega_{\Gamma}(r)$ is strictly increasing for $r \in [1, \, R]$; see the proof of~\cite[Theorem 3.1]{SWX}.  However, if $\mu$ is sub-exponential, that is, $\limsup_{n \to \infty} \frac{1}{n} \log \sum_{x \in S_n} \mu(x) = 0$, then we have $H_r(n) = \sum_{x \in S_n} \sum_{k=0}^{\infty} r^k p_k(e, \, x) \geq \sum_{x \in S_n} r \mu(x)$ and hence $\omega_{\Gamma}(r) \geq 0$ for all $r \in (1, \, R]$, which together with Proposition~\ref{upperboundgrowthrateat1} imply that $\omega_{\Gamma}(r) = 0$ for all $r \in (0, \, 1]$.   }
  \begin{proof}
    We modify the arguments in~\cite[Lemma 3.1]{DWY}, where $\mu$ was assumed to be finitely supported and only the case $1<r<R$ was treated.
    There are constants $c_1 > 0$ and $v_2 \geq 0$ such that $\sharp S_n \leq c_1 \mathrm{e}^{v_2 n}$.  Fix $\delta > 0$.  We choose $c_2$ so that for every $\delta \leq r \leq R - \delta$, 
    \[
      v_2 -  c_2 \left( \log R - \log (R - r) \right) < \omega_{\Gamma}(r). 
    \]
  Note that since the underlying random walk is symmetric, for every $x$ and every $k$, we have
  $p_k(e,x)p_k(e,x)\leq p_{2k}(e,e)$ and by \cite[Lemma~1.9]{Woessbook}, $p_{2k}(e,e)\leq R^{-2k}$.
  Thus, 
\begin{equation}\label{decayheatkernel}
    p_k(e, x) \leq R^{-k}
\end{equation} for every $x \in \Gamma$ and $k \ge 0$.  
  Consequently, we have for $\delta \le r \le R - \delta$,
\[
\sum_{x \in S_n} \sum_{k > c_2 n} r^k p_k(e, x) \le c_1 \mathrm{e}^{v_2 n} \sum_{k > c_2 n} \left( \frac{r}{R} \right)^k \le c_1 \delta^{-1} R  \mathrm{e}^{v_2 n - c_2 \left( \log R - \log r \right) n}.    
\]
By the choice of $c_2$ we have that 
\[
  \omega_{\Gamma}(r) = \limsup_{n \to \infty} \frac{1}{n} \log \sum_{x \in S_n} \sum_{k=0}^{c_2 n} r^k p_k(e, x).  
\]
Now for $0 < s < r \le R - \delta$,
\[
  H_n(s) \geq \sum_{x \in S_n} \sum_{k = 0}^{c_2 n} s^k p_k(e, x)
  \ge \left( \frac{s}{r} \right)^{c_2 n} \sum_{x \in S_n} \sum_{k = 0}^{c_2 n} r^k p_k(e, x),  
\]
and hence $0 \leq \omega_{\Gamma}(r) - \omega_\Gamma(s) \le c_2 \left( \log r - \log s \right)$, since $\omega_\Gamma$ is non-decreasing by Proposition~\ref{omeganonincreasing}.
Now $\delta > 0$ is arbitrary, so we prove that $\omega_{\Gamma}(r)$ is continuous in $0 < r < R$.
  \end{proof}

%{\color{blue} For continuity at 1, I suggest the following formulation. Please tell me if you agree.
%The reason why I changed $k^{-d/2}$ into $f(k)$ with $f$ a sub-exponential function is because of the on-diagonal lower bound with $f(k)=\mathrm{exp}(-k^{1/3})$ in polycyclic groups. But as pointed out by Alexopoulos, an off-diagonal lower bound not involving the volume growth rate cannot hold for polycyclic groups of exponential growth.
%However, maybe such a lower bound can exist for groups of intermediate growth.
%Anyway, if you prefer to formulate everything with the classical function $f(k)=k^{-d/2}$, I would also agree and we can go back to Longmin's original formulation.}
%\lw{I agree to use the general sub-exponential form as you explained.}

The continuity at the inverse of the spectral radius seems to be a difficult problem in general.
It is already known that $\omega_\Gamma$ is continuous at $R$ in hyperbolic groups \cite[Theorem~1.1]{SWX} and in relatively hyperbolic groups \cite[Theorem~1.1]{DWY}.
However, we do not know much beyond these classes of groups.

We now investigate further continuity at $r=1$.
By \cite[Theorem~12.5]{Woessbook}, $R=1$ can only happen if $\Gamma$ is amenable, hence the following discussion only applies to amenable groups.

Assume that the random walk satisfies the following Gaussian lower bound.
There exists a sub-exponential function $f$, i.e.\ $\frac{1}{n}\log f(n)\to 0$, $n\to \infty$, such that for every $k$, for every $x\in \Gamma$,
\begin{equation}\label{Gaussianlowerbound}
      p_k(e, x) \gtrsim f(k) \mathrm{e}^{- c \frac{|x|^2}{k}}.
\end{equation}
Then,
\begin{align*}
  H_r(n) =& \sum_{x \in S_n} \sum_{k = 0}^{\infty} r^k p_k(e, x)
  \gtrsim \sharp S_n \sum_{k = 0} f(k) \mathrm{e}^{- k \log r^{-1} - c \frac{n^2}{k}} \\ 
  \gtrsim & \ f\left( \left \lfloor \frac{n}{\sqrt{\log r^{-1}}} \right \rfloor \right)\sharp S_n  \mathrm{e}^{- n (1 + c) \sqrt{\log r^{-1}}}. 
\end{align*}
%{\color{blue}It seems there misses the term $\sharp S_n$, so the proof only works for groups of sub-exponential growth (which is already great). I'm trying to figure out if such off-diagonal lower bounder would imply sub-exponential growth in literature, but I'm really not sure, so maybe we have to add this assumption.} \lw{It seems that such a lower bound would imply that $v  = 0$; see the modification colored cyan. }
It follows that
\[
  \omega_{\Gamma}(r) \geq -(1 + c) \sqrt{\log r^{-1}} + v. 
\]
Letting $r \uparrow 1$, we see that $0 \geq \limsup_{r \uparrow 1} \omega_{\Gamma}(r) \geq v$.  Therefore $v = 0$ and
\[
  \lim_{r \uparrow 1} \omega_{\Gamma}(r) = 0.  
\]

%For the continuity of $\omega_{\Gamma}(r)$ at $r = 1$, it seems we need some kind of lower bound for the off-diagonal estimates of $p_n(e, x)$.  For example, if we have the Gaussian lower bound
%$$p_k(e, x) \gtrsim k^{-d/2} \mathrm{e}^{- c \frac{|x|^2}{k}}, $$

%then
%\begin{align*}
%  H_r(n) =& \sum_{k = 0}^{\infty} r^k \sum_{x \in S_n} p_k(e, x)
%  \gtrsim \lw{\sharp S_n} \sum_{k = 0} k^{-d/2} \mathrm{e}^{- k \log r^{-1} - c \frac{n^2}{k}} \\ 
%  \gtrsim& \left( \frac{n}{\sqrt{\log r^{-1}}} \right)^{-d/2} \lw{\left( \sharp S_n \right)} \mathrm{e}^{- n (1 + c) \sqrt{\log r^{-1}}}. 
%\end{align*}
%{\color{blue}It seems there misses the term $\sharp S_n$, so the proof only works for groups of sub-exponential growth (which is already great). I'm trying to figure out if such off-diagonal lower bounder would imply sub-exponential growth in literature, but I'm really not sure, so maybe we have to add this assumption.} \lw{It seems that such a lower bound would imply that $v  = 0$; see the modification colored cyan. }
%It follows that
%\[
%  \omega_{\Gamma}(r) \geq -(1 + c) \sqrt{\log r^{-1}} \lw{+ v}. 
%\]
%\lw{Letting $r \uparrow 1$, we see $0 \geq \lim_{r \uparrow 1} \omega_{\Gamma}(1) \geq v$.  Therefore $v = 0$ and}
%\[
%  \lim_{r \uparrow 1} \omega_{\Gamma}(r) = 0.  
%\]

\begin{rem}
We thus recover the fact that a Gaussian lower bound like~(\ref{Gaussianlowerbound}) cannot hold for groups of exponential volume growth, which was already noticed in literature, see for instance the comments after \cite[(0.3)]{Alexopoulospolycyclic}.
\end{rem}

\begin{prop}
    Let $\Gamma$ be a finitely generated virtually nilpotent group endowed with a finite generating set and let $\mu$ be a finitely supported symmetric admissible probability measure on $\Gamma$.
    Then, the function $\omega_\Gamma$ is left-continuous at 1.
\end{prop}

\begin{proof}
    The fact that a Gaussiam lower bound~(\ref{Gaussianlowerbound}) with $f(k)=k^{-d/2}$ holds for virtually nilpotent groups is well known, see for instance \cite[Corollary~1.9]{Alexopoulos}.
\end{proof}

\section{Abelian and nilpotent groups}\label{Sectionabeliannilpotent}
Our goal in this section is to prove the following result.
\begin{thm}\label{thmnilpotent}
Let $\Gamma$ be a finitely generated nilpotent group of nilpotency class $N_\Gamma$ at most 2.
Consider a finitely generating set $S$ and a finitely supported admissible symmetric probability measure $\mu$ on $\Gamma$.
Then,
$$\sum_{x\in S_n}G(e,x)\asymp n.$$
In general, there exists $C\geq 1$ and $0<\beta\leq 1$ that only depends on $N_\Gamma$ such that
$$\frac{1}{C}n\leq \sum_{x\in S_n}G(e,x)\leq Cn^{2-\beta}.$$
Moreover, if $\Gamma=\Z^d$, $d\geq 3$ then there exists $C>0$ such that
$$\sum_{x\in S_n}G(e,x)\sim C n.$$
\end{thm}

\subsection{Random walks in $\Z^d$}
We start with the following result.
\begin{prop}\label{propabelian}
Let $\mu$ be a finitely supported symmetric admissible probability measure on $\Z^d$, $d\geq 3$.
Endow $\Z^d$ with a finite generating set S.
Then, there exists $C>0$ such that
$$\sum_{x\in S_n}G(e,x)\sim C n.$$
\end{prop}

\begin{proof}
Classical estimates show that $G(e,x)\sim C_0\|x\|^{-d+2}$, see \cite[Theorem~25.11]{Woessbook}.
Here, $\|x\|$ is the norm of $x$ given by the inverse of the covariance matrix, which is symmetric and defines a positive definite quadratic form $Q$ associated with the random walk.
These asymptotics actually hold in a more general setting, namely it is only required that $\mu$ has suited finite polynomial moments related to the dimension $d$, as proved by Uchiyama, see in particular \cite[Theorem~2]{Uchiyama}.
The function $x\mapsto \|x\|^{-d+2}$ is $(-d+2)$-homogeneous in the sense that for $a> 0$, $\|ax\|^{-d+2}=a^{-d+2}\|x\|^{-d+2}$.
Let us fix a generating set $S$ for $\Z^d$.
Then, letting $S_n$ be the sphere of radius $n$ centered at $e$ with respect to $S$, we have by \cite[Theorem~1.1]{DLM}
$$\frac{1}{\sharp S_n}\sum_{x\in S_n}\|x\|^{-d+2}\sim C_1 n^{-d+2}.$$
Moreover, by \cite[Theorem~1.4]{DLM},
$$\sharp S_n\sim C_2n^{d-1},$$
hence
$$\sum_{x\in S_n}\|x\|^{-d+2}\sim C_3 n.$$
Consequently,
\begin{equation*}
\sum_{x\in S_n}G(e,x)\sim C_4 n. \qedhere
\end{equation*}
\end{proof}

For sake of completeness, let us consider the case of a (possibly lazy) simple random walk on $\Z^d$, endowed with the standard set of generators.
The quadratic form $Q$ associated with the covariance matrix introduced above is the Hessian at $u=0$ of the function
$$\Phi(u)=\sum_{x\in \Z^d}\mu(x)\mathrm{e}^{u\cdot x},$$
see \cite[Section~8~B, Section~13]{Woessbook} for more details.
Letting $(e_1,...,e_d)$ be standard generators defined by $e_i=(0,...,0,1,0...,0)$, where $1$ is at the $i$th position and assuming that $\mu$ is of the form
$$\mu=\alpha \delta_e+\frac{1-\alpha}{2d}\sum_{i=1}^d(\delta_{e_i}+\delta_{-e_i}),$$
we have
$$\Phi(u)=\alpha+\frac{1-\alpha}{d}\sum_{i=1}^d\mathrm{cosh}(u_i),$$
where $u_i$ are the coordinates of $u$.
In particular,
$$Q(x)=\frac{1-\alpha}{d}\langle x,x\rangle$$
and so the explicit computations of \cite[Theorem~25.11]{Woessbook} yield
$$G(e,x)\sim \frac{\Gamma(\frac{d-2}{2})\pi^{-\frac{d}{2}}(1-\alpha)^{-d+3/2}}{2d^{-d+3/2}}\|x\|_2^{-d+2},$$
where $\|\cdot\|_2$ is the Euclidean norm.

Furthermore, the constant in \cite[Theorem~1.1]{DLM} is given by
$$\int_L \|x\|^{-d+2} d\mu_L,$$
where $L$ is the boundary of the convex hull $\mathbf{C}$ of the generating set $S$ and $d\mu_L$ is the cone volume on $L$.
%Using \cite[Remark~1.2, Theorem~1.3]{DLM}, we see that this constant can be rewritten as
%$$\frac{2}{d}\int_{\mathbf{C}}\|x\|^{-d+2}\mathrm{dLeb},$$
%where $\mathrm{dLeb}$ is the standard Lebesgue measure on $\R^d$.
Applying this to our context where $S$ is the standard generating set and $\|\cdot\|$ is the Euclidean norm $\|\cdot\|_2$, we get
$$\frac{1}{\sharp S_n}\sum_{x\in S_n}\|x\|^{-d+2}\sim \int_{\partial B_1}\|x\|_2^{-d+2} d\mu_1 n^{-d+2},$$
where $B_1$ is the unit ball for the $\|\cdot\|_1$-norm $\{x,\sum |x_i|\leq 1\}$,
$\partial B_1$ is the unit sphere $\{x,\sum |x_i|=1\}$ and $\mu_1$ is the uniform measure on $\partial B_1$.
Also, the constant in \cite[Theorem~1.4]{DLM} is given by $d\cdot \mathrm{vol}(\mathbf{C})$ and
the volume of the ball $B_1$ is given by $\frac{2^d}{d!}$, see for instance \cite[Lecture~1]{Ball}.
We finally get
$$\sum_{x\in S_n}G(e,x)\sim n\cdot \frac{2^d}{d!} \int_{\partial B_1}\|x\|_2^{-d+2} d\mu_1 \frac{\Gamma(\frac{d-2}{2})\pi^{-\frac{d}{2}}(1-\alpha)^{-d+3/2}}{2d^{-d+1/2}}.$$
%{\color{blue}I'm not sure that the last integral is computable in general. Using Sage, I found a very complicated and lengthily formula for $d=3$.
%Maybe we can keep this formula as it is, without computing the integral.
%But if you prefer, we can also omit this explicit computation for the simple random walk on the standard set of generators, as it is not really necessary.
%}

\subsection{Non-abelian nilpotent groups}
In general, for nilpotent groups, we cannot prove such precise asymptotics.
Let $\Gamma$ be a finitely generated nilpotent group.
Set $\Gamma^1=\Gamma$, $\Gamma^2=[\Gamma,\Gamma]$, ..., $\Gamma^n=[\Gamma^{n-1},\Gamma]$.
Let $N_{\Gamma}$ be the nilpotency class of $\Gamma$, that is $N_{\Gamma}$ is the smallest $n$ such that $\Gamma^{n+1}$ is trivial.
Note that the groups $\Gamma^n/\Gamma^{n+1}$ are all finitely generated abelian groups and so have a well defined rank.
The homogeneous dimension of $\Gamma$ is defined as
$$D=\sum_{n=1}^{N_{\Gamma}}n\ \mathrm{rank}\left (\Gamma^{n}/\Gamma^{n+1}\right ).$$

By \cite[Theorem~1.8, Corollary~1.9]{Alexopoulos}, there exists $c$ such that
\begin{equation}\label{AlexopoulosGaussian}
\frac{1}{c}n^{-D/2}\mathrm{e}^{-c\frac{|x|^2}{2n}}\leq p_n(e,x)\leq cn^{-D/2}\mathrm{e}^{-\frac{1}{c}\frac{|x|^2}{2n}},
\end{equation}
where $D$ is the homogeneous dimension of $\Gamma$.
This is enough to deduce rough asymptotics for the Green function as we now show.

\begin{lem}\label{lemmaasympGaussianestimates}
For any positive $a,b,c$, there exists $K>0$ such that as $u$ tends to infinity,
$$\sum_{n\geq1}cn^{-a}\mathrm{e}^{-bu/n}\sim K u^{-(a-1)}.$$
\end{lem}

\begin{proof}
The proof is adapted from \cite[Theorem~25.11]{Woessbook}.
We set
$t_n=\frac{n}{u}$, so that
$$\Delta_n=t_n-t_{n-1}=\frac{1}{u}\underset{u\to \infty}{\longrightarrow} 0.$$
Then,
$$\frac{1}{c}u^{a-1}\sum_{n\geq1}cn^{-a}\mathrm{e}^{-bu/n}=\sum_{n\geq 1}t_n^{-a}\mathrm{e}^{-\frac{b}{t_n}}\Delta_n.$$
The right hand-side is a Riemannian sum of
$$\int_0^{\infty}t^{-a}\mathrm{e}^{-b/t}dt$$
so it converges to a constant that only depends on $a$ and $b$.
\end{proof}

Applying this lemma to $a=D/2$ and $u=|x|^2$, we deduce from~(\ref{AlexopoulosGaussian}) that for every symmetric finitely supported admissible probability measure on a finitely generated nilpotent group,
\begin{equation}\label{asympGreennilpotent}
    G(e,x)\asymp |x|^{-D+2}.
\end{equation}

Guivarc'h \cite{Guivarch73} and Bass \cite{Bass} independently proved that the homogeneous dimension $D$ is the degree of the polynomial growth of $\Gamma$, that is, letting $B_n$ be the ball of radius $n$,
$$\sharp B_n\asymp n^D.$$
Conversely, by the landmark paper of Gromov \cite{Gromovnilpotent}, a finitely generated group of polynomial growth is virtually nilpotent.
Extending the work of Gromov, Pansu \cite{Pansu83} proved deep results about the asymptotic behavior of the word distance in terms of the geometry of the universal cover of $\Gamma$.
As a particular case of his results, we have
$$\sharp B_n\sim cn^D,$$
see \cite[Proposition~(5), Section~(51)]{Pansu83}.
Although one might expect from these asymptotics that $\sharp S_n\sim c'n^{D-1}$, it is in fact much more difficult to obtain precise asymptotics for $\sharp S_n$.
However, Breuillard and Le Donne proved the following.
\begin{prop}\label{propBLD}\cite[Corollary~9, Corollary~11]{BreuillardLeDonne}
Let $\Gamma$ be a finitely generated nilpotent group endowed with a finite generating set.
Let $N_\Gamma$ be its nilpotency class.
\begin{enumerate}
    \item If $N_\Gamma\leq 2$, then
    $$\sharp B_n=Cn^D+O\big(n^{D-1}\big)$$ and
    $$\sharp S_n\asymp n^{D-1}.$$
    \item In general, there exists $0<\beta\leq 1$ that only depends on $N_\Gamma$ such that
    $$\sharp B_n=Cn^D+O\big(n^{D-\beta}\big)$$ and there exists $C>0$ such that
    $$\frac{1}{C}n^{D-1}\leq \sharp S_n \leq C n^{D-\beta}.$$
\end{enumerate}
\end{prop}

This proposition allows us to conclude the proof of Theorem~\ref{thmnilpotent}.

\begin{proof}[Proof of Theorem~\ref{thmnilpotent}]
Combining~(\ref{asympGreennilpotent}) and Proposition~\ref{propBLD}, we have for $N_\Gamma\leq 2$
$$\sum_{x\in S_n}G(e,x)\asymp n$$
and in general we have
$$\frac{1}{C}n\leq \sum_{x\in S_n}G(e,x)\leq C n^{2-\beta}.$$
Together with Proposition~\ref{propabelian}, this concludes the proof of the theorem.
\end{proof}

\medskip
It is also conjectured that one can always take $\beta=1$ in the above result of Breuillard and Le Donne, see \cite[Conjecture~10]{BreuillardLeDonne}.
In such case, we would always have $H_1(n)\asymp n$.
Conversely, note that since we already know that Gaussian estimates~(\ref{AlexopoulosGaussian}) hold, the following conjecture is a reformulation of that of Breuillard and Le Donne.

\begin{conj}\label{conjecturenilpotent}
Let $\Gamma$ be a finitely generated nilpotent group and let $\mu$ be a symmetric admissible finitely supported probability measure and $S$ be a finite generating set.
Then,
$$\sum_{x\in S_n}G(e,x)\asymp n.$$
\end{conj}

Finally, note that for the particular case of the Heisenberg group $\mathbb H_3(\Z)$ whose homogeneous dimension is 4, the fact that $\sharp S_n\sim Cn^3$ for the standard set of generators is well known and was first proved by Shapiro, see \cite[Theorem,~p.607]{Shapiro}.
See also \cite[Theorem~25]{DuchinMooney} which is analogous to \cite[Theorem~1.1]{DLM} for abelian groups, which in turn is a key result that we used within the proof of Proposition~\ref{propabelian}.
However, even for the Heisenberg group, it is unknown  to our knowledge if one can replace $\asymp$ with $\sim$ in the asymptotics of the Green function~(\ref{asympGreennilpotent}).

%{\color{blue}For the abelian case, I'm referring also to Uchiyama, but for finitely supported random walks, the result is folklore and the proof is not difficult. It basically goes as follows.
%One uses the same lemma~\ref{lemmaasympGaussianestimates} and then one uses precise local limit theorems in time and space to deduce that $G(e,x)=\tilde{G}(e,x)+o(|x|^{-d+2})$, where $\tilde{G}(e,x)=\sum_{n\geq1}n^{-d/2}\mathrm{e}^{-|x|^2/(2n)}$.
%Similar precise local limit theorems in time and space for nilpotent groups are yet unproved as far as I know, although I tried at some point to prove them (for Martin boundary purpose).
%Alexpoulos in the above cited paper indeed proved such (little less precise) local limit theorems in space and time, but the estimates are not sharp enough to mimic the proof of the abelian case.

%Beside, I'm surprised, but as explained above, it is not at all a trivial problem to compute asymptotics of $\sharp S_n$.
%Unfortunately, I asked directly to Breuillard and Le~Donne if there was any progress on their conjecture and as far as they know, there is none. So it seems we are stuck here.
%We have a nice reformulation of their conjecture in terms of random walks asymptotics, but we do not have any clear evidence to support our conjecture. This is is still interesting though.
%}

\subsection{Polycyclic groups and groups of intermediate growth}\label{sectionpolycyclic}
Let us conclude this section with a small discussion on possible ways to generalize our results for nilpotent groups.

By a seminal result of Mal'cev \cite{Malcev}, finitely generated torsion free nilpotent groups are exactly the lattices in simply connected nilpotent Lie groups, see also \cite[Theorem~2.18]{Raghunathan}.
%Moreover, any finitely generated nilpotent group has a torsion free subgroup of finite index, see \cite[Lemma~4.6]{Raghunathan}.
A further topic of interest would be the growth rate of the Green function for lattices in simply connected solvable Lie groups.
By~\cite[Theorem~4.28]{Raghunathan}, these are exactly torsion free polycyclic groups.
Moreover, every finitely generated polycyclic group has a subgroup of finite index which is torsion free by \cite[Lemma~4.6]{Raghunathan}.

This gives motivation for studying $H_1(n)$ for polycyclic groups.
The exponent $1/3$ in Varapoulos' on-diagonal upper bound $p_n(e,e)\prec \mathrm{e}^{-c_1n^{1/3}}$ that we used within the proof of Proposition~\ref{upperboundgrowthrateat1} is known to be optimal for polycyclic groups of exponential growth.
Indeed, by \cite[Theorem~1]{Alexopoulospolycyclic}, we also have $p_n(e,e)\gtrsim \mathrm{e}^{-c_2n^{1/3}}$ for such groups.
In order to get asymptotics of $H_1(n)$, optimal off-diagonal upper and lower bounds would be needed.
However, as already mentioned, Gaussian-type estimates like~(\ref{Gaussianlowerbound}) cannot hold for groups of exponential volume growth and finding optimal upper and lower bounds would have to involve somehow the growth rate $v$.
This in turn would require new material.

\medskip
In another direction, it is not hard to see that the rough asymptotics of the Green function in Theorem~\ref{thmnilpotent} hold for any virtually nilpotent groups.
Indeed, let $\Gamma$ be a virtually nilpotent and  $N$ be a finite index nilpotent subgroup of $\Gamma$.
Consider a word distance $|\cdot|_\Gamma$ on $\Gamma$ and a word distance $|\cdot|_N$ on $N$.
Then, the restriction of $|\cdot|_\Gamma$ to $N$ is bi-Lipschitz to $|\cdot|_N$.
Moreover, if $|x|_\Gamma=n$, then $x$ is within a uniform bounded distance of a point $y\in N$.
In particular, $G(e,x)\asymp G(e,y)$ and $|y|_N\asymp n$.
We find that
$$\sum_{x\in \Gamma,|x|_\Gamma=n}G(e,x)\asymp \sum_{x\in N,|x|_N\asymp n}G(e,x).$$
Moreover, $|\cdot|_\Gamma^{-D+2}\asymp |\cdot|_N^{-D+2}$ and given the asymptotics of Proposition~\ref{propBLD}, the sphere $S_n(\Gamma)$ for $|\cdot|_\Gamma$ and $S_n(N)$ for $|\cdot|_N$ satisfy that
$\sharp S_n(\Gamma)\asymp \sharp S_n(N)$.
Thus we have
$$\sum_{x\in \Gamma,|x|_\Gamma=n}G(e,x)\asymp n^{-D+2}\sharp S_n(\Gamma)\asymp  n^{-D+2}\sharp S_n(N).$$

As explained above, virtually nilpotent groups are exactly the groups of polynomial growth, by a celebrated result of Gromov.
Answering Conjecture~\ref{conjecturenilpotent} would thus settle the case of polynomial growth.
It would be interesting to see if the linear asymptotics of $H_1(n)$ also hold for groups of intermediate growth.

\section{The lamplighter group and DL graphs}\label{Sectionlamplighter}
Beyond nilpotent groups, studying the growth rate and the asymptotics of $H_1(n)$ for general finitely generated solvable groups seems to be a difficult task.
Our next goal is to compute precise such asymptotics for the lamplighter group, which is the first step into this direction.

\medskip
The lampligther group $L_q$ is the restricted wreath product $\Z_q\wr \Z$.
The study of this group is a classical topic in geometric group theory and we refer to \cite[Lecture~15]{ClayMargalit} for an overview of its properties.
The usual description of $L_q$ is as follows.
Consider the line $\Z$ and suppose that there is a lamp at each site of $\Z$.
Every lamp can be either off or lit with $q-1$ different colors, so that it has $q$ different possible states.
At the beginning, every lamp is switched off.
The lampligther is at the origin of $\Z$ and can either change the state of the lamp where they stand or take one step to the left or to the right in $\Z$.
Those possible actions are considered as the generators of the group.
Elements of $L_q$ are then described by the position of the lamplighter on $\Z$ and the states of the lamps, which can be seen as a function $\Z\to \Z_q$ with finite support.

There is a one-to-one correspondence between the lamplighter group $\Z_q \wr \Z$ and the vertices of the Diestel-Leader graph $\mathrm{DL}(q,q)$, which is defined as the following horocycle product.
Consider the $(q+1)$-regular tree $\mathbb T_q$ and let $h$ be a horofunction based at some fixed chosen point at infinity (see \cite[Section~2]{BrofferioWoess} for a precise definition and also \cite[Section~15.4]{ClayMargalit}, where $h$ is called a height function).
Then,
$$\mathrm{DL}(q,q)=\big \{(x_1,x_2),x_i\in \mathbb T_q, h(x_1)=-h(x_2)\big \}.$$
Thus, any element $x$ of $\Z_q\wr \Z$ can be written as $x=(x_1,x_2)$, $x_i\in \mathbb T_q$.

Following \cite{BrofferioWoess}, what we call here the simple random walk on $L_q$ is the simple random walk on $\mathrm{DL}(q,q)$ endowed with the product graph distance.
In other words, it is the simple random walk on the graph whose vertices are elements of $\mathrm{DL}(q,q)$ and two vertices $(x_1,x_2)$ and $(y_1,y_2)$ are connected by an edge if $x_i$ and $y_i$ are connected with an edge in $\mathbb T_q$ for $i=1,2$.
This is exactly the simple random walk on $L_q$ endowed with an appropriate set of generators described in \cite[Section~2]{BrofferioWoess}, see in particular \cite[(2.3)]{BrofferioWoess} for a formula for the probability measure $\mu$ (with $\alpha=\frac{1}{2}$ since we do not consider biased random walks here).

In terms of lamplighting, we need to change a bit the above description to understand this random walk.
Think of the lamps not placed at each vertex of $\Z$
but at the middle of each edge. Suppose the current position of the lamplighter is $k \in \Z$.
They first toss a coin. If “head” comes up, they move
to $k+1$ and switch the lamp on the transverse edge to a state chosen at random in
$\Z_q$. Otherwise, they move to $k-1$ and also switch the lamp on the transverse edge to
a random state.

\medskip
Lamplighter groups provide a great source of examples of particular asymptotic behaviors of random walks, see \cite{KaimanovichVershik},  \cite{Erschler1}, \cite{Erschler2}, \cite{Revelle1}, \cite{Revelle2}, \cite{BNW} just to name a few.
We prove here that $H_1(n)$ is still asymptotically linear in $n$ for the simple random walk.
Our study is based on the renewal theorems of \cite{BrofferioWoess}, see also references therein for other significant results related to random walks on these groups.

\begin{thm}\label{thmlamplighter}
Consider the simple random walk on $L_q$.
There exists $C$ such that
$$\sum_{x\in S_n}G(e,x)\sim C n.$$
\end{thm}

We identify $L_q=\Z_q\wr \Z$ with $\mathrm{DL}(q,q)$ as above, with $h$ a fixed horofunction.
By \cite[(4.1)]{BrofferioWoess}, the word distance in $L_q$ is given by
$$d(x,y)=d(x_1,y_1)+d(x_2,y_2)-|h(y_1)-h(x_1)|.$$
Note that $h(y_2)-h(x_2)=-h(y_1)+h(x_1)$, so
$$d(x,y)=d(x_1,y_1)+d(x_2,y_2)-|h(y_2)-h(x_2)|.$$

The identity element $e$ is identified with $(o_1,o_2)$ in $\mathrm{DL}(q,q)$, where $o_i$ is the root of $\mathbb T_q$.
In particular, we have $h(o_1)=h(o_2)=0$.
For $(x_1,x_2)\in \mathrm{DL}(q,q)$, we now set, following \cite{BrofferioWoess},
$$|x_i|=d(o_i,x_i)=\mathfrak d_i+\mathfrak u_i,$$
$$h(x_i)=\mathfrak d_i-\mathfrak u_i$$
and
$$\mathfrak s := \mathfrak u_1+\mathfrak u_2=\mathfrak d_1+\mathfrak d_2.$$
The quantities $\mathfrak d$ and $\mathfrak u$ are defined in terms of geometric features of $\mathbb T_q$, see \cite[Section~2, Figure~1]{BrofferioWoess} for more details.
The elements $x_1$ and $x_2$ are completely determined by $|x_1|$, $|x_2|$ and $h(x_1)=-h(x_2)$.
Thus, any element $(x_1,x_2)$ is completely determined by $\mathfrak u_1$, $\mathfrak d_2$ and $\mathfrak s$.

Let us now describe the sphere of radius $n$ in $L_q$.
By what precedes,
$$d(e,x)=|x_1|+|x_2|-|h(x_1)|=\mathfrak u_1+\mathfrak d_1+\mathfrak u_2+\mathfrak d_2-|\mathfrak d_1-\mathfrak u_1|.$$
We first find the $x\in S_{n}$ with $h(x_1)>0$.
For such $x$,
$$d(e,x)=2 \mathfrak u_1 + \mathfrak u_2 + \mathfrak d_2=\mathfrak u_1 + \mathfrak d_1+2\mathfrak d_2=\mathfrak u_1+\mathfrak d_2 + \mathfrak s.$$
We write $\tilde{S}(n,m)$ the set of $z\in \mathbb T_q$ with $|z|=n$ and $h(z)=m$, $-n\leq m\leq n$, and so $|n-m|$ needs to be even.
Let us compute the cardinality of $\tilde S(n,m)$ for small values of $n,m$, when $q=3$, i.e. $\mathbb T_q$ is the $4$-regular tree, which is the Cayley graph of $F_2$.
We write $F_2=\langle a,b\rangle$ and we fix the point at infinity defining the horofunction $h$ to be $a^{\infty}$.
We have $\tilde S(1,1)=\{a^{-1},b,b^{-1}\}$ and $\tilde S(1,-1)=\{a\}$, so
$$\sharp \tilde{S}(1,1)=3, \ \sharp \tilde{S}(1,-1)=1.$$
For $n=2$, we have $\tilde S(2,2)=\{a^{-2},a^{-1}b,a^{-1}b^{-1},ba,ba^{-1}, b^2,b^{-1}a,b^{-1}a^{-1},b^{-2}\}$,
$\tilde S(2,0)=\{ab,ab^{-1}\}$ and $\tilde S(2,-2)=\{a^2\}$.
Therefore,
$$\sharp \tilde{S}(2,2)=3^2, \ \sharp \tilde{S}(2,0)=2, \ \sharp \tilde{S}(2,-2) =1.$$
Similarly, we prove that
$$\sharp \tilde{S}(3,3)=3^3, \ \sharp \tilde{S}(3,1)=2\times 3, \ \sharp \tilde{S}(3,-1)=2, \ \sharp \tilde{S}(3,-3)=1,$$
$$\sharp \tilde{S}(4,4)=3^4, \ \sharp \tilde{S}(4,2)=2\times 3^2, \ \sharp \tilde{S}(4,0) = 2\times 3, \ \sharp \tilde{S}(4,-2) = 2, \ \sharp \tilde{S}(4,-4) = 1.$$

In general, we have for $0\leq k \leq n$
$$\sharp \tilde{S}(n,n)=q^{n}, \ \sharp \tilde{S}(n,-n)=1$$
and
$$\sharp \tilde{S}(n,n-2k)=(q-1)\times q^{n-1-k}, 1\leq k \leq n-1.$$
Note that if $z\in \tilde{S}(n,n-2k)$, then $\mathfrak u =k$ and $\mathfrak d=n-k$.
In fact, $\mathfrak u$ is the number of steps upward in direction to $a^{\infty}$ and $\mathfrak d$ is the number of steps downward.

\medskip

We now consider the set $A(i,j,k)=\{(x_1,x_2), \mathfrak u_1=i,\mathfrak d_2=j,\mathfrak s = k\}$, with $i\leq k$, $j\leq k$ and $i< k-j$ (the last condition being equivalent to $h(x_1)> 0$).
Then, by what precedes, for $0<j<k$,
$$\sharp A(0,j,k)=(q-1)q^{k - 1}, \ \sharp A(k,j,k)=(q-1)q^{k-1}$$
and for $0<i<k$,
$$\sharp A(i,j,k) = (q-1)^2q^{k-2}.$$
Now, for $j=0$, we have
$$\sharp A(0,0,k)=q^k,\ \sharp A(k , 0 ,k)=1$$
and for $0<i<k$,
$$\sharp A(i,0,k)=(q-1)q^{k-1}.$$
Finally, for $j=k$, we have
$$\sharp A(0,k,k)=(q-1)q^{k-1},\ \sharp A(k,k,k)=q^k,$$
and for $0<i<k$,
$$\sharp A(i,k,k)=(q-1)q^{k-1}.$$

\medskip
We are now ready to prove Theorem~\ref{thmlamplighter}.
\begin{proof}
The following asymptotics of the Green function are proven in \cite[Theorem~4.2]{BrofferioWoess}.
We have, as $|x|$ tends to infinity,
\begin{equation}\label{asympGreenlamplighter}
G(e,x)\sim \frac{C}{\mathfrak s^4q^{\mathfrak s}}\bigg (\frac{q+1}{q-1} \big ( \mathfrak u_1 (\mathfrak s-\mathfrak d_2)+(\mathfrak s -\mathfrak u_1) \mathfrak d_2\big) + \mathfrak s \mathfrak u_1 \mathfrak d_2 + \mathfrak s (\mathfrak s-\mathfrak u_1)(\mathfrak s -\mathfrak d_2)\bigg).
\end{equation}

We fix $(x_1,x_2)\in S_{2n}$ and $h(x_1)>0$, so
$2n=\mathfrak s+\mathfrak d_2+\mathfrak u_1$ and $\mathfrak d_2+\mathfrak u_1<\mathfrak s$.
In particular, $s>n$.
Also, since $\mathfrak d_2,\mathfrak u_1\geq 0$, we have $s\leq 2n$.
Conversely, choose any couple $(\mathfrak d_2,\mathfrak s)$ satisfying $n+1\leq \mathfrak s\leq 2n$ and $0\leq \mathfrak d_2\leq 2n-\mathfrak s$. Set $\mathfrak u_1=2n-\mathfrak s-\mathfrak d_2$.
Note that $\mathfrak s>n$ implies that $2n-\mathfrak s\leq \mathfrak s -1$, so
$0\leq \mathfrak u_1\leq \mathfrak s-1-\mathfrak d_2$ and $0\leq \mathfrak d_2<\mathfrak s$. We find that $\mathfrak u_1<\mathfrak s-\mathfrak d_2$ and so any such triple $(\mathfrak u_1,\mathfrak d_2,\mathfrak s)$ defines a point $(x_1,x_2)$ with $h(x_1)>0$.
By~(\ref{asympGreenlamplighter}), setting $\mathfrak u_1=i$, $\mathfrak d_2=j$, $\mathfrak s =k$ we get the following.
The sum of the Green function along points $x$ in the sphere $S_{2n}$ satisfying $h(x_1)>0$ can be written as
\begin{align*}
\sum_{\underset{h(x_1)>0}{(x_1,x_2)\in S_{2n},}}G(e,x)\sim \sum_{k=n+1}^{2n}\sum_{j=0}^{2n-k}&\frac{C}{k^4q^{k}}\sharp A(i,j,k)\\
&\bigg (\frac{q+1}{q-1} \big ( i (k-j)+(k-i)j\big) \\
&+ kij + k(k-i)(k-j)\bigg).
\end{align*}
Thus, replacing $\mathfrak u_1$ with $2n-\mathfrak s-\mathfrak d_2$, we find
\begin{align*}
\sum_{\underset{h(x_1)>0}{(x_1,x_2)\in S_{2n},}}G(e,x)\sim &\sum_{k=n+1}^{2n}\sum_{j=0}^{2n-k}\frac{C}{k^4q^{k}}\sharp A(2n-k-j,j,k)\\
&\bigg (\frac{q+1}{q-1} \big ( (2n-k-j) (k-j)+(2k +j-2n) j\big) \\
&+ k (2n-k-j) j + k (2k+j-2n)(k-j)\bigg).
\end{align*}

Since the dominant term in $\sharp A(i,j,k)$ is $q^k$, we find
\begin{align*}
\sum_{\underset{h(x_1)>0}{(x_1,x_2)\in S_{2n},}}G(e,x)&\ \sim C_1 \sum_{k=n+1}^{2n}\sum_{j=0}^{2n-k}\frac{1}{k^4}\big((2n-k-j)(k-j)+(2k+j-2n)j\\
&\hspace{2.8cm}+k(2n-k-j)j+k(2k+j-2n)(k-j)\big)\\& \sim C_2 n.
\end{align*}
By symmetry, we have the same asymptotics for $h(x_1)<0$.
Now, if $h(x_1)=0$, then $\mathfrak u_1=\mathfrak s-\mathfrak d_2$.
Combining this with $\mathfrak u_1+\mathfrak d_2+\mathfrak s=2n$, we get $\mathfrak s=n$, so the sum of the Green function along $S_{2n}$, assuming further $h(x_1)=0$ is asymptotic to a constant $C_3$.
Thus, we find that $\sum_{x\in S_{2n}}G(e,x)$ is linear in $n$.
The same proof with more delicate sums to handle, when considering the integer part of $n/2$,  shows that the same is true for $S_{2n+1}$.
This concludes the proof of Theorem~\ref{thmlamplighter}.
\end{proof}

Lamplighter groups are solvable and thus amenable.
They are actually classical examples of non-polycyclic solvable groups, see for instance the comments at the end of \cite[Chapter~13]{DrutuKapovich}.
In light of the discussion in Section~\ref{sectionpolycyclic}, this raises the following question.

\begin{quest}\label{questionsolvable}
Is it true that
$$\sum_{x\in S_n}G(e,x)\asymp n$$
holds for every symmetric finitely supported admissible probability measure on a finitely generated amenable group ?
If not, can a counter-example be found among finitely generated solvable groups ?
\end{quest}

Also, let us conclude this section with further comments on horocycle products.
We refer to \cite{Woesshorocyclic} for a more complete exposition.
Given two hyperbolic spaces $X_1,X_2$, one can perform a construction similar to that of $\mathrm{DL}(q,q)$ by choosing two Busemann functions $h_1$, respectively $h_2$, based at infinity on $X_1$, respectively $X_2$.
The horocycle product of $X_1$ and $X_2$ is the set of pairs $(x_1,x_2)\in X_1\times X_2$ with $h_1(x_1)+h_2(x_2)=0$.
Interesting examples of such spaces are the horocycle product $\mathrm{HT}(p,q)$ of $\mathbb H(p)$ with a homogeneous tree of degree $q+1$ and the horocycle product $\mathrm{Sol}(p,q)$ of $\mathbb H(p)$ and $\mathbb H(q)$.
In both cases, $\mathbb H(p)$ is the standard Poincar\'e upper half-space with suited rescaling of the hyperbolic metric.
Amenable Baumslag-Solitar groups $BS(1,q)$ act properly discontinuously and co-compactly via isometries on $\mathrm{HT}(q,q)$, see \cite{FM98}.
On the other hand, the spaces $\mathrm{Sol}(p,q)$ are classical examples of solvable Lie groups and play an important role in Thurston's geometrization theorem.
We refer to \cite{BSW} and references therein for further details.
A next subject of interest would be the study of $H_1(n)$ on amenable Baumslag-Solitar groups and on lattices of $\mathrm{Sol}(p,q)$.
Note that the later are examples of polycyclic groups.

%\begin{quest}\label{questionamenable}
%Is it true that
%$$\sum_{x\in S_n}G(e,x)\asymp n$$
%holds for every symmetric finitely supported admissible probability measure on a finitely generated amenable group ?
%\end{quest}

\section{Cartesian products of trees}\label{Sectionbitree}
In this section, we study the behavior of $H_r(n)$ for random walks on
%$T\times \Z^d$ and
$T\times T'$, where $T,T'$ are regular trees.
The asymptotics for the Green functions are given by the work of Picardello and Woess \cite{PicardelloWoessbitree}, see also \cite[Section~28]{Woessbook} and references therein for the particular case of $T\times \Z$.

%\subsection{Renewal theory in the bi-tree}
Let $T_1,T_2$ be regular trees of degree $l_1,l_2\geq 3$.
We consider the lazy simple random walk $\mu_i$ on $T_i$ whose transition kernel $p_i(x,y)$ is defined by 
$$p_i(x,y)=\left \{\begin{array}{lll}
        \frac{1}{2l_i} \text{ if } x,y \text{ are connected with an edge in } T_i,\\
        \frac{1}{2} \text{ if }x=y,\\
        0 \text{ otherwise.}
    \end{array}\right .$$
In particular, $\mu_i$ is an admissible symmetric finitely supported probability measures on $T_i$.
For every $\alpha_1,\alpha_2\geq 0$, $\alpha_1+\alpha_2=1$, we let $\mu$ be the probability measure on $T_1\times T_2$ given by
\begin{equation}\label{defmubitree}
\mu=\alpha_1 \mu_1+\alpha_2\mu_2.
\end{equation}
In terms of Markov operators, this means that
$$P_{\mu}=\alpha_1 P_{\mu_1}\otimes I+\alpha_2 I\otimes P_{\mu_2}.$$
As noted in \cite[Section~3]{PicardelloWoessbitree}, the lazy simple random walk on $T_1\times T_2$ with $\mu(x,x)=1/2$ is given by $\alpha_i=\frac{l_i}{l_1+l_2}$.
Set $\rho_i = \frac{1}{2} + \frac{\sqrt{l_i - 1}}{l_i}$ and $R = \frac{1}{\alpha_1 \rho_1 + \alpha_2 \rho_2}$.
Then, $\rho_i$ is the spectral radius of $\mu_i$, see \cite[(2.2)]{PicardelloWoessbitree} and $\rho=R^{-1}$ is the spectral radius of $\mu$, see \cite[Section~3]{PicardelloWoessbitree}.
We prove here the following.
\begin{thm}\label{thmbitree}
If $l_1=l_2$, then for every $r<R$, we have
$H_r(n)\asymp \mathrm{e}^{n\omega_\Gamma(r)}$.
If $l_1>l_2$, then there exists a phase transition at some $r_0\in (1,R)$ such that the following holds.
\begin{itemize}
    \item For every $r<r_0$, we have $H_r(n)\asymp \mathrm{e}^{n\omega_\Gamma(r)}$.
    \item At $r=r_0$, we have $H_r(n)\asymp n^{-1}\mathrm{e}^{n\omega_\Gamma(r)}$.
    \item For every $r_0<r<R$, we have $H_r(n)\asymp n^{-3/2}\mathrm{e}^{n\omega_\Gamma(r)}$.
\end{itemize}
\end{thm}

The remainder of the section is devoted to the proof of this theorem.
By \cite[Theorem~3.1]{PicardelloWoessbitree}, for every $r<R$, for every $\lambda_0\geq 0$, there exist $r_1,r_2$ such that as $x=(x_1,x_2) \in T_1\times T_2$ tends to infinity and $|x_2|/|x_1|=\lambda$ converges $\lambda_0$, we have
$$G(e,x|r)\sim G_1(e,x_1|r_1)G_2(e,x_2|r_2)\bigg(|x_1|+\frac{l_1}{l_1-2}\bigg)\bigg(|x_2|+\frac{l_2}{l_2-2}\bigg)\frac{1}{|x_1|^{5/2}}C(\lambda),$$
where $C(\lambda)$ is a continuous positive function.
The numbers $r_1$ and $r_2$ are the unique solutions of the system
\begin{equation}\label{systemr1r2}
\left\{
    \begin{array}{ll}
        \alpha_1r_1^{-1}+\alpha_2r_2^{-1}=r^{-1}\\
        \alpha_2\sqrt{(r_2^{-1}-\frac{1}{2})^2-\frac{l_2-1}{l_2^2}}=\lambda \alpha_1\sqrt{(r_1^{-1}-\frac{1}{2})^2-\frac{l_1-1}{l_1^2}}
    \end{array}
\right.
\end{equation}
By symmetry, the same holds when $|x_1|/|x_2|=\lambda'$ converges to $\lambda'_0\in [0,+\infty)$, switching the indices $1$ and $2$ and replacing the function $C$ with a function $C'$ which is also continuous and positive.

\medskip
Set $\beta_i = \frac{\sqrt{l_i - 1}}{l_i}$ and 
\begin{gather*}
  F_i(r) = \frac{l_i}{l_i - 1} \left( r^{-1} - \frac{1}{2} - \sqrt{\left( r^{-1} - \frac{1}{2} \right)^2 - \beta_i^2} \right), \\
  G_i(r)  = \frac{r^{-1}}{r^{-1} - (1/2) \left( 1 + F_i(r) \right)}. 
\end{gather*}
By \cite[(2.4)]{PicardelloWoessbitree}, we have that
\[
  G_i(e_i, \, x_i  |  r) = G_i(r) F_i(r)^{|x_i|},
\]
hence in particular $G_i(e_i,e_i|r)=G_i(r)$.
% Note that for $0 < r \le \rho_i^{-1} = 2 - \frac{\sqrt{l_i - 1}}{\rho_i l_i}$,
% \[
%   r \left( 1 + F_i(r) \right)
%   \le r + \frac{l_i}{l_i - 1} - \frac{l_i}{2(l_i - 1)} r
%   \le
%   \frac{l_i}{l_i - 1} + \frac{l_i - 2}{2(l_i - 1)} \rho_i^{-1}
%   = 2  - \frac{l_i - 2}{{\color{blue}2}\rho_i l_i \sqrt{l_i - 1}}. 
% \]
% Therefore
% \[
%   1 \leq G_i(r) \leq \frac{{\color{blue}4} \rho_i l_i \sqrt{l_i - 1}}{l_i - 2}. 
% \]
% {\color{blue}There might be some small computational issues here, since $\rho_i^{-1} = 2 - 2\frac{\sqrt{l_i - 1}}{\rho_i l_i}$ and not $\rho_i^{-1} = 2 - \frac{\sqrt{l_i - 1}}{\rho_i l_i}$. But anyway, $1\leq G(r)\leq G(\rho_i^{-1})$. }
Since $1 \leq G_i(r) \leq G_i(\rho_i^{-1})$, we have
\begin{equation}
  \label{Hniasymp}
  H_n^i(r) = \sum_{x_i \in T_i \colon |x_i| =n} G_i(e_i, x_i | r) = l_i (l_i - 1)^{n - 1} G_i(r) F_i(r)^n \asymp \mathrm{e}^{n \omega_{T_i}(r)}, 
\end{equation}
where
\[
  \omega_{T_i}(r) = \log (l_i - 1) + \log F_i(r).  
\]

\medskip
We now fix $r<R$ and we consider $n\geq0$.
The sphere $S_n$ in $T_1\times T_2$ can be decomposed as
$$S_n=\bigcup_{k=0}^nS_{k}^1\times S_{n-k}^2.$$
Then, for every $k\leq n$, there exists a unique couple $(r_1(\lambda), r_2(\lambda))$ satisfying~(\ref{systemr1r2}) with $\lambda=\frac{n-k}{k}$.
Moreover,
$$\frac{1}{|x_1|^{5/2}}C(\lambda)=\frac{1}{(|x_1|+|x_2|)^{5/2}}(1+\lambda)^{5/2}C(\lambda)$$
and for $k\geq n/2$, $\lambda \leq 1$, hence by the continuity of $C(\lambda)$,
$$\frac{1}{|x_1|^{5/2}}C(\lambda)\asymp \frac{1}{(|x_1|+|x_2|)^{5/2}}.$$
Similarly,
$$\frac{1}{|x_2|^{5/2}}C'(\lambda')=\frac{1}{(|x_1|+|x_2|)^{5/2}}(1+\lambda')^{5/2}C'(\lambda')$$
and so for $k\leq n/2$,
$$\frac{1}{|x_2|^{5/2}}C'(\lambda')\asymp \frac{1}{(|x_1|+|x_2|)^{5/2}}.$$
Setting $\kappa_i=\frac{l_i}{l_i-2}$, we thus have%\ywy{In the next lines, $\kappa_1$ should be $\alpha_1$, $\kappa_2$ to $\alpha_2$, right?}{\color{blue}I don't think so, $\kappa_1$ and $\kappa_2$ are just the constants in the above expression.}\ywy{I do not see why: $\alpha_i=l_i/(l_1+l_2)$ appears in the above formula of $G(e,x|r)$. If they are not $\alpha_i$, we should indicate  what $\kappa_i$ are...  }
%{\color{blue}Oh okay I see where the confusion comes from. In our example, we do not set $\alpha_1$ and $\alpha_2$. We say that for the lazy simple random walk with $\mu(x,x)=\frac{1}{2}$, we would have $\alpha_i=\frac{l_i}{l_1+l_2}$, but our proof applies for any choice of $\alpha_1,\alpha_2$.
%Here, the $\kappa_i$ are just the constants in the expression of the asymptotic of the Green function above.So $\kappa_i=\frac{l_i}{l_i-2}$ as you spotted.I added the definition of $\kappa_i$ above.Sorry for the confusion.}
\begin{equation*}
\begin{split}
\sum_{x\in S_n}G(e,x|r)\asymp &\frac{1}{n^{5/2}}\sum_{k=0}^n\sum_{x_1\in S_k^1}\sum_{x_2\in S_{n-k}^2}G_1(e,x_1|r_1(\lambda))G_2(e,x_2|r_2(\lambda))\\
&\hspace{4cm}(k+\kappa_1)((n-k)+\kappa_2)\\
&\asymp \frac{1}{n^{5/2}}\sum_{k=0}^n(k+\kappa_1)((n-k)+\kappa_2)H^1_k(r_1(\lambda))H^2_{n-k}(r_2(\lambda)).
\end{split}
\end{equation*}
Applying~\eqref{Hniasymp}, we see that
\begin{equation}
  \label{Hnasymp}
  H_n(r) \asymp \frac{1}{n^{5/2}} \sum_{k = 0}^n (k+\kappa_1)((n-k)+\kappa_2) \exp \left( n \Psi(\lambda) \right), 
\end{equation}
with $\lambda = \frac{n - k}{k}$ and 
\[ 
  \Psi(\lambda) = \frac{1}{1 + \lambda} \omega_{T_1}(r_1(\lambda)) + \frac{\lambda}{1 + \lambda} \omega_{T_2}(r_2(\lambda)). 
\]
In order to find the asymptotics of $H_n(r)$, we thus need to find where the function $\Psi(\lambda)$ takes its maximum value.  

\medskip
Now let $t = r^{-1} - \frac{1}{2}$ and $t_i = t_i(\lambda) = r_i(\lambda)^{-1} - \frac{1}{2}$.  Then $(t_1, t_2)$ solves the system of equations
\begin{equation}
  \label{systemt1t2}
  \begin{cases}
    \alpha_1 t_1 + \alpha_2 t_2 = t, \\
    \alpha_2 \sqrt{t_2^2 - \beta_2^2} = \lambda \alpha_1 \sqrt{t_1^2 - \beta_1^2}.  
  \end{cases}
\end{equation}

\begin{lem}\label{t1t2deriviative}
  The functions $\lambda \mapsto t_1(\lambda)$ and $\lambda \mapsto t_2(\lambda)$ are continuously differentiable.  Furthermore,
  \begin{gather*}
    t_1'(\lambda) = - \frac{\lambda \alpha_1 (t_1^2 - \beta_1^2)}{\alpha_2 t_2 + \lambda^2 \alpha_1 t_1}, \\
    t_2'(\lambda) = - \frac{\alpha_1}{\alpha_2} t_1'(\lambda) = \frac{\lambda^{-1} \alpha_2 \left( t_2^2 - \beta_2^2 \right)}{\alpha_2 t_2 + \lambda^2 \alpha_1 t_1}. 
  \end{gather*}
\end{lem}

\begin{proof}
Let $U$ be the open set $(\beta_1,+\infty)\times (\beta_2,+\infty)\times (\beta,+\infty)\times (0,+\infty)$ with $\beta = \rho - \frac{1}{2}$.  
We set 
$$\Upsilon:(t_1,t_2,t,\lambda)\in U\mapsto \bigg(\alpha_1 t_1 + \alpha_2 t_2-2t, \alpha_2 \sqrt{t_2^2 - \beta_2^2}- \lambda \alpha_1 \sqrt{t_1^2 - \beta_1^2}\bigg).$$
Then
$$\frac{\partial \Upsilon}{\partial t}=(-2,0)$$
and
$$\frac{\partial \Upsilon}{\partial \lambda}=\bigg(0,- \alpha_1 \sqrt{t_1^2-\beta^2}\bigg).$$
For $t_1>\beta_1$, the matrix
$$\begin{pmatrix}
    -2&0\\
    0&- \alpha_1 \sqrt{t_1^2-\beta^2}
\end{pmatrix}$$
is invertible.
The implicit function theorem shows that the solution $(t_1,t_2)$ of~(\ref{systemt1t2}) is continuously differentiable in the variables $(t,\lambda)$.
The formulas for $t_1'(\lambda)$ and $t_2'(\lambda)$ are then derived from~(\ref{systemt1t2}).
\end{proof}

Define
\[
  \varphi_i(t) = \log l_i + \log \left( t - \sqrt{t^2 - \beta_i^2} \right). 
\]
Then
\[
  \Psi(\lambda) = \frac{1}{1 + \lambda} \varphi_1(t_1(\lambda)) + \frac{\lambda}{1 + \lambda} \varphi_2(t_2(\lambda)). 
\]
Since 
\begin{equation}
  \label{phideriviative}
  \varphi_i'(t) = - \frac{1}{\sqrt{t^2 - \beta_i^2}},   
\end{equation}
we have that
\[
  \lambda \varphi_2'(t_2(\lambda)) = \frac{\alpha_2}{\alpha_1} \varphi_1'(t_1(\lambda)), 
\]
and hence
\begin{equation}
  \label{Psideriviative}
  \begin{split}
    \Psi'(\lambda) =& \frac{\varphi_1'(t_1(\lambda))}{1 + \lambda}  t_1'(\lambda) + \frac{\lambda \varphi_2'(t_2(\lambda))}{1 + \lambda}  t_2'(\lambda) + \frac{\varphi_2(t_2(\lambda)) - \varphi_1(t_1(\lambda))}{(1 + \lambda)^2} \\
    =& \frac{\alpha_1 t_1'(\lambda) + \alpha_2 t_2'(\lambda)}{\alpha_1 (1 + \lambda)} \varphi_1'(t_1(\lambda)) + \frac{\varphi_2(t_2(\lambda)) - \varphi_1(t_1(\lambda))}{(1 + \lambda)^2} \\  
  =& \frac{\varphi_2(t_2(\lambda)) - \varphi_1(t_1(\lambda))}{(1 + \lambda)^2}.
  \end{split}
\end{equation}
Furthermore,
\begin{equation}
  \label{Psi2deriviative}
  \Psi''(\lambda) = - \frac{2}{(1 + \lambda)^3} \left[ \varphi_2(t_2(\lambda)) - \varphi_1(t_1(\lambda)) \right] - \frac{\alpha_1 \sqrt{t_1^2 - \beta_1^2}}{(1 + \lambda) \left( \alpha_2 t_2 + \lambda^2 \alpha_1 t_1 \right)}. 
\end{equation}

By Lemma~\ref{t1t2deriviative} and~\eqref{phideriviative}, we see that $\varphi_1(t_1(\lambda))$ (resp. $\varphi_2(t_2(\lambda))$) is strictly increasing (resp. decreasing) in $\lambda$.  It follows that there is at most one $\lambda_0 \in [0, \, + \infty)$ such that $\Psi'(\lambda_0) = 0$.  Note that $t_2(0) = \beta_2$, $t_1(0) = \alpha_1^{-1} \left( t - \alpha_2 \beta_2 \right) > \beta_1$, and
\[
  \varphi_2(t_2(0)) = \frac{1}{2} \log \left( l_2 - 1 \right). 
\]
On the other hand, $\lim_{\lambda \to + \infty} t_1(\lambda) = \beta_1$, $\lim_{\lambda \to + \infty} t_2(\lambda) = \alpha_2^{-1}(t - \alpha_1 \beta_1) > \beta_2$, and
\[
  \lim_{\lambda \to + \infty} \varphi_1(t_1(\lambda)) = \frac{1}{2} \log \left( l_1 - 1 \right). 
\]

Assume $l_1 = l_2 = l$.  Then for $\lambda_0 = \frac{\alpha_2}{\alpha_1}$ we have that $t_1(\lambda_0) = t_2(\lambda_0) = t$ and $\Psi'(\lambda_0) = 0$.  Note that
\[
  \Psi''(\lambda_0) = - \frac{\alpha_1^2 \sqrt{t^2 - \beta^2}}{\alpha_2 t} < 0
\]
by Lemma~\ref{t1t2deriviative} and~\eqref{phideriviative}.  By Lemma~\ref{asymptotics}~(ii) below, we can deduce that
\begin{equation}
  \label{Hnl1=l2}
  H_n(r) \asymp \mathrm{e}^{n \Psi(\lambda_0)} = \mathrm{e}^{n \log \left[ l \left( r^{-1} - \frac{1}{2} - \sqrt{\left( r^{-1} - \frac{1}{2} \right)^2 - \beta^2} \right) \right]}. 
\end{equation}
This concludes the proof of Theorem~\ref{thmbitree} for the case $l_1=l_2$.

\medskip
Assume now that $l_1 > l_2$.
Then, we have that $\beta_1 < \beta_2$ and $\varphi_1(\beta_1) > \varphi_2(\beta_2)$.  Since $\lim_{s \to + \infty} \varphi_1(s) = - \infty$, there exists $t_0 > \alpha_1 \beta_1 + \alpha_2 \beta_2 > \beta_1$ such that
\[
  \varphi_1 \left( \alpha_1^{-1} \left( t_0 - \alpha_2 \beta_2 \right) \right) = \varphi_2(\beta_2).  
\]

If $t < t_0$, then
\[
  \varphi_1(t_1(0)) = \varphi_1 \left( \alpha_1^{-1} \left( t - \alpha_2 \beta_2 \right) \right) > \varphi_1 \left( \alpha_1^{-1} \left( t_0 - \alpha_2 \beta_2 \right) \right) = \varphi_2 (t_2(0))
\]
and hence $\Psi'(\lambda) < 0$ for all $\lambda \geq 0$.  It follows that $\Psi(\lambda)$ takes its unique maximum at $\lambda = 0$.  By Lemma~\ref{asymptotics}~(i) below, 
\begin{equation}
  \label{Hnt<t0}
    H_n(r) \asymp n^{-3/2} \mathrm{e}^{n \log \varphi_1 \left( \alpha_1^{-1} \left( t - \alpha_2 \beta_2 \right) \right)}. 
  \end{equation}

  Similarly, if $t = t_0$, then $0$ is also the unique maximum point of $\Psi(\lambda)$, and we have further that $\Psi'(0) = 0$, 
  \[
    \Psi''(0) = - \frac{\sqrt{\left( t_0 - \alpha_2 \beta_2 \right)^2 - \alpha_1 \beta_1^2}}{\alpha_2 \beta_2} < 0.  
  \]
  Thus, Lemma~\ref{asymptotics}~(iii) below shows that
  \begin{equation}
    \label{Hnt=t0}
    H_r(n) \asymp n^{-1} \mathrm{e}^{n \Psi(0)} = n^{-1} \mathrm{e}^{n \log \varphi_1 \left( \alpha_1^{-1} \left( t - \alpha_2 \beta_2 \right) \right)}. 
  \end{equation}

  It remains to consider the case that $t > t_0$.  Since $t_0 > \alpha_1 \beta_1 + \alpha_2 \beta_2$ we have
  \[
    \lim_{\lambda \to + \infty} \varphi_2 \left( t_2 (\lambda) \right) = \varphi_2 \left( \alpha_2^{-1} \left( t - \alpha_1 \beta_1 \right) \right) < \varphi_2(\beta_2) < \varphi_1(\beta_1) = \lim_{\lambda \to + \infty} \varphi_1(t_1(\lambda)),  
  \]
  and 
  \[
    \varphi_1(t_1(0)) = \varphi_1 \left( \alpha_1^{-1} \left( t - \alpha_2 \beta_2 \right) \right) < \varphi_1 \left( \alpha_1^{-1} \left( t_0 - \alpha_2 \beta_2 \right) \right) = \varphi_2(t_2 (0)).  
  \]
  Thus there exists $\lambda_0 > 0$ such that $\varphi_1(t_1(\lambda_0)) = \varphi_2 \left( t_2(\lambda_0) \right)$.  By~\eqref{Psi2deriviative}, $\Psi''(\lambda_0) < 0$, and we see from Lemma~\ref{asymptotics}~(ii) below that
  \begin{equation}
    \label{Hnt>t0}
    H_r(n) \asymp \mathrm{e}^{n \varphi_1(t_1(\lambda_0))}.  
  \end{equation}

To conclude, what is left to do is proving that $r_0>1$, i.e. $t_0<1/2$.
Lengthily computations would prove that
$$\varphi_1 \left( \alpha_1^{-1} \left( \frac{1}{2} - \alpha_2 \beta_2 \right) \right) < \varphi_2(\beta_2),$$ hence we necessarily have $t_0<1/2$.
However, we see that at $t_0$, we have by~(\ref{Hnt=t0})
$ H_r(n) \asymp n^{-1} \mathrm{e}^{n \Psi(0)} = n^{-1} \mathrm{e}^{n \log \varphi_1 \left( \alpha_1^{-1} \left( t - \alpha_2 \beta_2 \right) \right)}. $
In particular,
$$\omega_\Gamma(r_0)=\log \varphi_1 \left( \alpha_1^{-1} \left( t_0 - \alpha_2 \beta_2 \right) \right)=\log \varphi_2(\beta_2)>0.$$
Since $\omega_\Gamma(1)=0$ and $\omega_\Gamma$ is increasing, we see directly that $r_0>1$.
This ends the proof of Theorem~\ref{thmbitree}. \qed

\begin{lem}
  \label{asymptotics}
  Assume that $\Phi \in C^2 \left( [0, +\infty) \right)$ is eventually decreasing and has a unique maximum point at $0 \leq \lambda_0 < \infty$.  Define
  \[
    f(n) = \sum_{k = 0}^n k (n - k) \mathrm{e}^{n \Phi \left( \frac{n - k}{k} \right)}. 
  \]
  \begin{enumerate}[(i)]
  \item If $\lambda_0 = 0$ and $\Phi'(0) < 0$, then
    %\begin{equation}\label{asymp0}
    \[
      f(n) \asymp n \mathrm{e}^{n \Phi(0)}.  
    \]
    %\end{equation}
  \item If $\lambda_0 > 0$ and $\Phi''(\lambda_0) < 0$, then
    \[
      f(n) \asymp n^{5/2} \mathrm{e}^{n \Phi(\lambda_0)}. 
    \]
  \item If $\lambda_0 = 0$, $\Phi'(0) = 0$ and $\Phi''(0) < 0$, then
    \[
      f(n) \asymp n^{3/2} \mathrm{e}^{n \Phi(0)}. 
    \]
    
  \end{enumerate}
\end{lem}

\begin{proof}
  \begin{enumerate}[(i)]
  \item For any $0 < \varepsilon < - \Phi'(0)$, there exists $\delta > 0$ such that $\left| \Phi'(\lambda) - \Phi'(0) \right| < \varepsilon$ for every $0 \leq \lambda \leq \delta$.  By the mean value theorem, for $\lambda \leq \delta$, 
    \[
      \left( \Phi'(0) - \varepsilon \right) \lambda \leq \Phi(\lambda) - \Phi(0) \leq \left( \Phi'(0) + \varepsilon \right) \lambda. 
    \]
    If $\frac{n - k}{k} \leq \delta$, then we have $k \geq (1+\delta)^{-1}n$.  Thus 
    \begin{align*}
      f(n) \geq& \sum_{k \ge (1+\delta)^{-1} n} k (n - k) \mathrm{e}^{n \Phi \left( \frac{n - k}{k} \right)} \\ 
      \succ& n \mathrm{e}^{n \Phi(0)} \sum_{k \geq (1+\delta)^{-1} n} (n - k) \mathrm{e}^{n \left( \Phi'(0) - \varepsilon \right) \frac{n - k}{k}} \\
      \geq& n \mathrm{e}^{n \Phi(0)} \sum_{k \leq \frac{\delta}{1 + \delta} n} k \mathrm{e}^{\left( \Phi'(0) - \varepsilon \right) k} \\
      \succ & n \mathrm{e}^{n \Phi(0)}. 
    \end{align*}
    Since $0$ is the unique maximum point of $\Phi(\lambda)$, there exists $\eta > 0$ such that $\Phi(\lambda) \leq \Phi(0) - \eta$ for all $\lambda \geq \delta$. % {\color{blue}We could have $\Phi(\lambda)\to \Phi(0)$ as $\lambda\to \infty$, but I think this doesn't happen in the examples since every time we have $\Phi'(\lambda)<0$ for $\lambda>\lambda_0$, so $\Phi$ is eventually decreasing.}
    Thus
    \[
      \sum_{k < (1 + \delta)^{-1}n } k (n - k) \mathrm{e}^{n \Phi \left( \frac{n - k}{k} \right)}
      \leq n^3 \mathrm{e}^{n \left( \Phi(0) - \eta \right)}. 
    \]
    Also,
    \begin{align*}
      \sum_{k \geq (1 + \delta)^{-1} n} k (n - k) \mathrm{e}^{n \Phi \left( \frac{n - k}{k} \right)}
      \leq& n \mathrm{e}^{n \Phi(0)} \sum_{k \geq (1 + \delta)^{-1} n} (n - k) \mathrm{e}^{n \left( \Phi'(0) + \varepsilon \right) \frac{n - k}{k}} \\
      \leq& n \mathrm{e}^{n \Phi(0)} \sum_{k \leq \frac{\delta}{1 + \delta}n} k \mathrm{e}^{\left( \Phi'(0) + \varepsilon \right) (1 + \delta)^{-1} k} \\
      \leq& n \mathrm{e}^{n \Phi(0)} \sum_{k = 0}^{\infty} k \mathrm{e}^{\left( \Phi'(0) + \varepsilon \right) (1 + \delta)^{-1} k}
    \end{align*}
    Combining the last two displays yields the desired upper-bound. 
  \item  For $0 < c_1 < - \frac{\Phi''(\lambda_0)}{2} < c_2$, there exists $\delta > 0$ such that
\begin{equation}\label{secondorderapproximationPhi}
 - c_2 \left( \lambda - \lambda_0 \right)^2 \leq \Phi(\lambda) - \Phi(\lambda_0) \leq - c_1 \left( \lambda - \lambda_0 \right)^2
\end{equation}
     for $\left| \lambda - \lambda_0 \right| \leq \delta$.  Clearly, we have that
    \[
      \sum_{k \colon \left| \frac{n - k}{k} - \lambda_0 \right| > \delta} k (n - k) \mathrm{e}^{n \Phi \left( \frac{n - k}{k} \right)} \prec n^3 \mathrm{e}^{n \left( \Phi(\lambda_0) - \eta \right)}
    \]
    for some $\eta > 0$. % {\color{blue} Same thing here, but again, $\Phi$ is eventually decreasing in the examples.}
    Now,
    \begin{align*}
      \sum_{k \colon \left| \frac{n - k}{k} - \lambda_0 \right| \leq \delta} k (n - k) \mathrm{e}^{n \Phi \left( \frac{n - k}{k} \right)} 
      \prec&  n^3 \mathrm{e}^{n \Phi(\lambda_0)} \sum_{k \colon \left| \frac{n - k}{k} - \lambda_0 \right| \leq \delta} \frac{1}{n} \mathrm{e}^{- c_1 n \left( \frac{n}{k} - 1 - \lambda_0 \right)^2} \\ 
      \asymp& n^3 \mathrm{e}^{n \Phi(\lambda_0)} \int_{(1 + \lambda_0 + \delta)^{-1}}^{(1 + \lambda_0 - \delta)^{-1}} \mathrm{e}^{- c_1 n \left( x^{-1} - 1 - \lambda_0 \right)^2} \mathrm{d} x.
    \end{align*}
By a change of variables, we get
\begin{align*}
    \sum_{k \colon \left| \frac{n - k}{k} - \lambda_0 \right| \leq \delta} k (n - k) \mathrm{e}^{n \Phi \left( \frac{n - k}{k} \right)} =& n^3 \mathrm{e}^{n \Phi(\lambda_0)} \int_{- \delta}^{\delta} \mathrm{e}^{-c_1 n y^2} \frac{\mathrm{d} y}{(y + 1 + \lambda_0)^2} \\
      \asymp& n^{5/2} \mathrm{e}^{n \Phi(\lambda_0)} \int_{- \delta \sqrt{n}}^{\delta \sqrt{n}} \mathrm{e}^{- c_1 z^2} \mathrm{d} z \\
      \asymp& n^{ 5 / 2} \mathrm{e}^{n \Phi(\lambda_0)}. 
\end{align*}
    The lower bound can be proved by the same arguments, changing $c_1$ with $c_2$.
  \item The proof of (iii) is similar to that of (ii), except that we need to replace~(\ref{secondorderapproximationPhi}) with
  $$ - c_2  \lambda^2 \leq \Phi(\lambda) - \Phi(0) \leq - c_1 \lambda ^2$$
  and the order of magnitude is not $\frac{n-k}{k}\asymp \lambda_0\pm \delta$ anymore, but $k\geq (1+\delta)^{-1}n$ as in~(i). \qedhere
  \end{enumerate}
\end{proof}

\section{Twisted Patterson-Sullivan measures}\label{Sectionpattersonsullivan}
%{\color{blue}This is great ! I started cleaning a bit, in particular making consistent notations with other sections. I changed all $G_r(e,x)$ to $G(e,x|r)$. I know we used $G_r(e,x)$ in our previous paper, but at several places here we use subscript and rely on notations of Candellero Gilch and Woess, where they use $G(e,x|r)$. If you do not agree feel free to go back to $G_r(e,x)$, but we need a small sentence explaining the change of notations.}

In this and next   sections, we will give applications of our results to relatively hyperbolic groups.
We start here by investigating the properties of Patterson-Sullivan like measures on the Bowditch boundary of such groups, which we define by using suited Poincar\'e series.

\subsection{Some background on relatively hyperbolic groups}
Since their introduction  by Gromov, relatively hyperbolic groups were studied by many authors through several equivalent definitions.
We will mainly use the viewpoint of Bowditch \cite{Bowditch} and Gerasimov-Potyagailo \cite{GePoJEMS}, \cite{GePoGGD}, \cite{GP16} in the sequel.
Consider a finitely generated group $\Gamma$ acting properly via isometries on a proper geodesic Gromov hyperbolic space $X$.
Define the \textit{limit set} $\Lambda_\Gamma$ as the closure of $\Gamma$ in the Gromov boundary $\partial X$ of $X$, that is, fixing a base point $x_0$ in $X$, $\Lambda_\Gamma$ is the set of all possible limits of sequences $g_n\cdot x_0$ in $\partial X$, $g_n\in \Gamma$. The proper action of $\Gamma$ on $X$ by isometries extends to a convergence group action on $\Lambda_\Gamma$ by homeomorphisms, which means that the induced action on the space of distinct triples is properly continuous (see \cite[Proposition~1.11]{Bo99} for example). If $\Lambda_\Gamma$ contains at least three points, then $\Gamma$ acts minimally on $\Lambda_\Gamma$. 

A loxodromic element $x\in \Gamma$ is an infinite order element with exactly two fixed points $x_-\ne x_+$ in $\Lambda_\Gamma$.  Moreover, $x$ acts via North-South dynamics on $\Lambda_\Gamma$ in the sense that for any $\xi\neq x_{\pm}$, $x^{\mp n}\xi$ converges to $x_{\mp}$ as $n$ goes to $+\infty$. Then $x_+$ is called the attracting fixed point of $x$ and $x_-$ is called its repelling fixed point. The fixed points of any two loxodromic elements are either the same  or disjoint. So $\Gamma$ contains infinitely many loxodromic elements with pairwise disjoint fixed points.

A point $\xi\in \Lambda_\Gamma$ is called \textit{conical} if there is a sequence $g_{n}$ of $\Gamma$ and distinct points $\xi_1,\xi_2$ in $\Lambda_\Gamma$ such that
$g_{n}\xi$ converges to $\xi_1$ and $g_{n}\zeta$ converges to $\xi_2$ for all $\zeta\neq \xi$ in $\Lambda_\Gamma$.
A point $\xi\in \Lambda_\Gamma$ is called \textit{parabolic} if its stabilizer in $\Gamma$ is infinite, fixes exactly $\xi$ in $\Lambda_\Gamma$ and contains no loxodromic element.
If, in addition, its stabilizer in $\Gamma$ acts co-compactly on $\Lambda_\Gamma \setminus \{\xi\}$, then $\xi$ is called \textit{bounded parabolic}.
Say that the action of $\Gamma$ on $X$ is \textit{geometrically finite} if the induced convergence group action on the limit set $\Lambda_
\Gamma$ is geometrically finite: $\Lambda_
\Gamma$ only consists of conical limit points and bounded parabolic limit points. 
 See \cite{Bo99} for more   general facts on convergence groups.

\medskip
A group $\Gamma$ is called \textit{relatively hyperbolic} with respect to a collection of subgroups $\mathbb P$ if it acts geometrically finitely on a proper geodesic hyperbolic space $X$ such that the stabilizers of parabolic limit points are exactly the conjugates of the elements of $\mathbb P$.
Elements of $\mathbb P$ are called maximal parabolic subgroups.
We will write $\mathbb P_0$ for the choice of a set of representatives of conjugacy classes of elements of $\mathbb P$.
By \cite[Proposition~6.15]{Bowditch}, such a set $\mathbb P_0$ is finite.

The limit set $\Lambda_\Gamma$ in the Gromov boundary of $X$ is called the \textit{Bowditch boundary} of $\Gamma$.
By \cite[Theorem~9.4]{Bowditch}, it is unique up to equivariant homeomorphism and in particular does not depend on the choice of a proper geodesic Gromov hyperbolic space $X$ on which $\Gamma$ acts geometrically finitely.
We will write $\partial\Gamma$ for the Bowditch boundary of $\Gamma$ in the sequel.
A relatively hyperbolic group is called \textit{non-elementary} if its Bowditch boundary is infinite; equivalently, $\sharp \partial\Gamma>2$.

%\ywy{I adapted some materials from \cite{YangPS}. The conformal density for general potential in \cite{PPS} is now standard, but here we use potential from random walks with Floyd boundary etc. 

%Another relevant work  is "The Manhattan curve, ergodic theory of topological flows and rigidity" by Stephen Cantrell, Ryokichi Tanaka (https://arxiv.org/abs/2104.13451) among many other  works. 

%Our function $r\mapsto \omega_\Gamma(r)$ is called there Manhattan curve for the word metric and Green metric. They obtained some regularity of Manhattan curve for strongly hyperbolic metric on hyperbolic groups, which are continuously differentiable etc. They use automatic structure to prove such results, but perhaps we can find  other methods to approach such results in RHGs...}

\medskip
We now fix a finite set $\mathbb P_0$ of representatives of conjugacy classes of maximal parabolic subgroups.
When $P\in \mathbb P_0$ and $g\in \Gamma$, we call $gP$ a maximal parabolic coset.
\begin{defn}
Let $gP$ be a maximal parabolic coset and $\eta, L>0$ be fixed constants. A point $p$ on a geodesic $\alpha$ is called \textit{$(\eta, L)$-deep} in $gP$ if
$$B(p, 2L)\cap \alpha\subseteq N_\eta(gP).$$
It is called an \textit{$(\eta, L)$-transition point} if it is not \textit{$(\eta, L)$-deep} in any maximal parabolic coset $gP$.  
\end{defn}

The following set of inequalities is called weak relative Ancona inequalities and will be helpful below.
They extend similar inequalities proved for hyperbolic groups by Ancona \cite{Ancona} for $r=1$, by Gou\"ezel-Lalley \cite{GouezelLalley} on co-compact Fuchsian groups for $r\leq R$ and by Gou\"ezel \cite{Gouezel} in full generality for $r\leq R$.
The version for relatively hyperbolic groups that we use here were first proved for $r=1$ by Gekhtman-Gerasimov-Potyagailo-Yang \cite{GGPY} and then by Dussaule-Gekhtman \cite{DG21} for $r\leq R$.

\begin{lem}\label{weakAncona}\cite[Theorem~1.1]{GGPY}\cite[Theorem~1.6]{DG21}
Let $\Gamma$ be a relatively hyperbolic group and let $\mu$ be a finitely supported admissible and symmetric probability measure on $\Gamma$.
Then, for every $c,\eta,L\geq 0$ there exists $C>0$ such that for every $r\leq R$, the following holds.
For every $x,y,z\in \Gamma$, if $y$ has a distance at most $c$ to an $(\eta,L)$-transition point on a geodesic from $x$ to $z$, then
$$\frac{1}{C}G(x,y|r)G(y,z|r)\leq G(x,z|r)\leq C G(x,y|r)G(y,z|r).$$
\end{lem}

Note that the constant $C$ is independent of $r\in [1,R]$.
In other words, the Green function is roughly multiplicative along transition points on a geodesic.
In both \cite{GGPY} and \cite{DG21}, these inequalities are formulated in terms of the Floyd distance, which is a suited rescaling of the word distance.
However, the statement for transition points directly follows from \cite[Corollary~5.10]{GePoGGD} which relates transition points with the Floyd distance.
We also refer to \cite[Section~9]{GGPY} for more details.

We will also use the following at some point.

\begin{lem}\cite[Lemma 2.14]{YangPS}\label{ConCharLem}
There exist universal constants $\eta, L$ with the following property. Let   $\gamma$ be a  geodesic ray  ending at a conical point $\xi\in\pG$.  Then   $\gamma$ contains a unbounded sequence of $(\eta,L)$-transition points $x_n$.% where $\delta_{x_n}(\gamma_-, \xi)\ge \delta_0$. 
\end{lem}

In the remainder of this section, we consider a finitely generated relatively hyperbolic group $\Gamma$.
When speaking of a transition point, we mean an $(\eta,L)$-transition point with $(\eta,L)$ satisfying Lemma~\ref{ConCharLem}.
We also fix a finitely supported symmetric and admissible probability measure $\mu$ on $\Gamma$.

\subsection{Busemann cocyles}
Given $x,y,z \in \Gamma$, let $B_z(x, y) := d(x,
z) -d(y,z)$ and $K_z(x,y|r)=\frac{G(x,z|r)}{G(y,z|r)}$.
The function $B_z$ is called the Busemann function associated with the distance $d$ at $z$.

Following \cite{BlachereBrofferio}, we define the $r$-Green distance by
$$
d_r(x,y) = - \log F_r(x,y)=-\log \frac{ G(x,y|r)}{G(e,e|r)}.
$$
Then $K_z(x,y|r)=\mathrm{e}^{-[d_r(x,z)-d_r(y,z)]}$ is the exponential of the Busemann function for the $r$-Green distance. We also write $|x^{-1}y|_r=d_r(x,y)$, and $|x^{-1}y|=d(x,y)$.

Consider the distance for $x,y\in \Gamma$: $$\mathfrak {d}_r(x,y):=\omega_\Gamma(r)|x^{-1}y| +|x^{-1}y|_r.$$ 
\begin{lem}\label{quasiisometricdistances}
If $1\leq r<R$, then the distance $\mathfrak d_r$ is proper and quasi-isometric to the word distance.
\end{lem}

\begin{proof}
The proof is standard, but we write a complete argument for sake of completeness.
First we prove that for every $r<R$, there exist $C_1> 0$ and $\alpha_1>0$ such that for every $x\in \Gamma$
\begin{equation}\label{upperboundquasiisometry}
G(e,x|r)\leq C_1\mathrm{e}^{-\alpha |x|}.
\end{equation}
Since $\mu$ is finitely supported, there exists $c> 0$ such that
$$G(e,x|r)=\sum_{n\geq c|x|}r^np_n(e,x).$$
Moreover, by~(\ref{decayheatkernel}),
$$p_n(e,x)\leq R^{-n}.$$
Therefore,
$$G(e,x|r)\leq \sum_{n\geq c|x|}\left (\frac{r}{R}\right )^n\leq C_1\left (\frac{r}{R}\right )^{c|x|}.$$
This proves~(\ref{upperboundquasiisometry}).

Second, we prove that for every $r\geq 1$, there exists $C_2>0$ and $\alpha_2>0$ such that for every $x\in \Gamma$,
\begin{equation}\label{lowerboundquasiisometry}
G(e,x|r)\geq C_2 \mathrm{e}^{-\alpha_2 |x|}.
\end{equation}
Indeed, since the support of $\mu$ generates $\Gamma$, there exists a path $x_0=e,x_1,...,x_n=x$ such that $n\asymp |x|$ and $\mu(x_k^{-1}x_{k+1})>0$.
In particular, we find that
$$G(e,x|r)\geq G(e,x|1)\geq \mu(x_0^{-1}x_1)\cdots \mu(x_{n-1}^{-1}x_n),$$
which proves~(\ref{lowerboundquasiisometry}).
We conclude that for $1\leq r<R$, the Green distance is quasi-isometric to the word distance and that $G(e,x|r)$ vanishes at infinity.
Consequently, the distance $\mathfrak d_r$ also is quasi-isometric to the word distance and satisfies that as $x$ goes to infinity,
$\mathfrak d_r(e,x)$ tends to infinity.
In particular, any ball for $\mathfrak d_r$ is contained in a larger ball for the word distance and thus is finite, so $\mathfrak d_r$ is proper.
\end{proof}

\begin{rem}
According to \cite[Lemma~2.1]{GouezelLalley}, for any non-amenable group $\Gamma$ and any finitely supported symmetric admissible probability measure $\mu$, $G(e,x|R)$ converges to 0 as $|x|$ goes to infinity.
As a consequence, the distance $\mathfrak d_R$ is proper, although it might not be quasi-isometric to the word distance.
\end{rem}
%As word metrics are quasi-isometric to $r$-Green metrics $d_r$ for $1\le r<R$, $\mathfrak d_r$ is  quasi-isometric to the word distance $d$. %{\color{blue}I don't think this is true, it would imply that the Green metric and the word distance are quasi-isometric, but this needs not be true for $r=R$.
%The fact that $G_r(x,y)\geq \mathrm{e}^{-\alpha d(x,y)}$ is true as long as $r\geq 1$ and the fact that $G_r(x,y)\leq \mathrm{e}^{-\beta(r) d(x,y)}$ holds for fixed $r<R$ but strongly uses the fact that $r<R$. We do not necessarily have that $G_R$ has sub-exponential decay in the word distance.}\ywy{Okay, Thanks. I also doubt about it, but did not investigate through.}. 
Define the corresponding  Busemann cocycle
\begin{equation}\label{defBusemann}
\mathfrak B_\xi(x, y;r) = \omega_\Gamma(r)B_\xi(x, y)-\log K_\xi(x,y|r)
\end{equation}

%In what follows, unless explicitly stated, we denote by $\pG$   either the Floyd boundary with Floyd metric or the Bowditch boundary with shortcut metric.

%\begin{rem}
%The constant  can be made uniform for all conical point if the action of $\Gamma$ on the Floyd boundary $\pG$ is geometrically finite. In particular, a uniform constant $\delta_0$ with the shortcut metric exists for any conical point in the Bowditch boundary.     
%\end{rem}

 %If $\xi$ is a conical point of Bowditch boundary, then $C$ can be chosen uniformly on the set of conical points.

\begin{lem}\label{BusemanEstLem}
There exists a constant $C>0$ with the following property.
Let $\xi \in \pG$ be a conical point, and $x, y \in \Gamma$. There exists    
a neighborhood $V=V(x,y)$ of $\xi$ in $\Gamma\cup\pG$ such that   for any $ z \in V \cap \Gamma$:
\begin{align*}
|B_\xi(x, y)-B_z(x, y)| \le C,  \\
|\log K_{\xi}(x, y|r)-\log K_{z}(x, y|r)| \le C.
\end{align*} 
\end{lem}

\begin{proof}
%{\color{blue}I think the proof can be shorten drastically and does not need to involve the Floyd distance,as we can rely on Wenyuan's paper for $B_\xi$ and on convegence of Martin kernels for $K_\xi$.This would go as follows. Yes, I agree. Thanks.

The statement for $B_\xi$ is proved in \cite[Lemma~2.20]{YangPS}.
Also, by \cite[Proposition~4.1]{DG21}, the Martin kernel $K_z(\cdot,e|r)=G(\cdot, z|r)/G(e, z|r)$ extends continuously  to $K_\xi(\cdot,e|r)$ as $z$ converges to a conical limit point $\xi$.
This follows from weak relative Ancona inequalities.
In particular, $K_z(x,y|r)$ converges to $K_\xi(x,y|r)$ as $z$ converges to $\xi$, so the statement for $K_\xi$ also holds.
%Recall that $K_\xi(x,y|r)$ is
%the limit of $K_z(x,y|r)$ if $\xi$ is a conical point. If $\gamma$ is a geodesic between $x$ and $\xi$, then for any $\epsilon
%>0$, there exists a number $L>0$ such that
%\begin{equation}\label{vlimit}
%|\log K_\xi(x,y|r) + [d_r(x, v)-d_r(y, v)]| \le \epsilon
%\end{equation}
%for any $v\in \gamma$ with $d(x, v) >L$. 
%Note that there exists a sequence of  
%points $x_n$ in $\gamma$ tending to $\xi$ such that $\delta_{x_n}(x,\xi)\ge \delta_0$ for a positive constant $\delta_0>0$. Note that $\delta_0$ may depend on $\xi$. 
%According to definition of Floyd distance, we then can choose    $v \in \gamma$ such that $d(x, v) \ge L$,  and \begin{equation}\label{rephorofuncE1}
%\delta_v(x, y) \le \delta_0/4.
%\end{equation}
%Let  $V = B_{\delta_x}(\xi, \lambda^{-|x^{-1}y|}\delta_0/4)$ be the ball at $\xi$ of radius $\lambda^{-d(x, v)}\delta_0/4$ with respect to
%the distance $\delta_x$. We shall show that $V$ is the desired
%neighborhood.
%For any $z \in V \cap G$, the equivariance of Floyd distance implies  
% $\delta_v(z, \xi) \le \delta_0/4$. We then obtain by (\ref{rephorofuncE1})  
%$\delta_v(y, z) \ge \delta_0/2, \; \delta_v(x, z) > \delta_0/2.$
%By the relative Anacona inequality, for some $C=C(\delta_0)>0$,  %we have 
%\begin{align*}
%|d_r(x,z) - d(x,v)-d(v,z)|\le C/2\\
%|d_r(y,z) - d(y,v)-d(v,z)|\le C/2.
%\end{align*} By (\ref{vlimit}), $|\log K_\xi(x, y|r)+\big(d_r(x, z)-d_r(y, z)\big)|\le C + \epsilon$.  Letting $\epsilon\to 0$ concludes   the proof.
\end{proof}

As a consequence, the Busemann cocyle $\mathfrak B_z(x,y)$ extends to a coarse cocycle $\mathfrak B_\xi(x,y)$ at a conical point $\xi$. That is, 
$$
\mathfrak B_\xi(x,y) := \limsup_{z\to \xi} \mathfrak B_\xi(x,y)
$$
does not depend on the choice of $z\to \xi$, up to a bounded additive error $C$ independent of $\xi$.

\subsection{Quasi-conformal densities}
A Borel measure $\mu$ on a topological Hausdorff space $T$ is \textit{regular}
if $\mu(A) = \inf\{\mu(U): A \subset U, U$ is open$\}$ for any Borel
set $A$ in $T$. It is called \textit{tight} if $\mu(A) =
\sup\{\mu(K): K \subset A, K$ is compact$\}$ for any Borel set $A$
in $T$. A finite Borel measure is called \textit{Radon} if   it is tight and regular. It is well-known that all
finite Borel measures on compact metric spaces are Radon, see \cite[Theorem~1.1, Theorem~1.3]{Bill99}. 

Denote by
$\mathcal M(\pG)$ the set of finite positive Radon measures on
$\pG$. Then $\Gamma$ acts on $\mathcal M(\pG)$ via
$g_*\nu(A) = \nu(g^{-1}A)$ for any Borel set $A$ in $\pG$.

\begin{defn}
    We call a map $x \mapsto \nu_x$ equivariant if for every $x,g\in \Gamma$, we have
$$\nu_{gx}(A) = \nu_{x}(g^{-1}A)$$ for every Borel set $A \subset\pG$.
\end{defn}

\begin{defn} 
Let $\omega \in [0, \infty[$. We call a $\Gamma$-equivariant map $$\nu: \Gamma \to
\mathcal M(\pG),\; x \mapsto \nu_x$$ an
\textit{$\omega$-dimensional quasi-conformal density} if for any $x,
y \in \Gamma$ the following holds
\begin{equation}\label{cdensitydef}
\frac{d\nu_{x}}{d\nu_{y}}(\xi) \asymp \mathrm{e}^{-\omega B_\xi (x, y)}K_{\xi}(x,y|r),
\end{equation}
for $\nu_y$-a.e. points $\xi \in \pG$, where the implicit constant does neither depend on $x, y$, nor on $\xi$.
\end{defn}
% {\color{blue}I think that in this context, quasi conformal density makes more sense, as we introduce the metric $\mathfrak d_r$ and mimic the work of Patterson and Sullivan.
% In fact, conformal means that the measure behaves well with respect to balls, whereas having the Gibbs property comes historically from thermodynamical formalism. This is why Tanaka uses this word in my opinion, since with co-authors, they use thermodynamical formalism at some point in their work.}

By the equivariant property of $\mu$, we see the following result.
\begin{lem}
Let $\{\nu_x\}_{x\in\Gamma}$ be a $\sigma$-dimensional quasi-conformal density on $\pG$.
Then the support of any $\nu_x$ is $\pG$.
\end{lem}

\begin{proof}
By definition, the support $\supp(\mu_x)$ is a maximal closed subset such that
any point in $\pG \setminus \supp(\mu_x)$ has an open
neighborhood which is $\nu_x$-null.
It is well-known that $\pG$ is a
minimal $\Gamma$-invariant closed set, see for instance \cite[Section~2]{Bo99}.
Thus, it suffices to prove that the support of
$\nu_x$ is $G$-invariant. This follows from equivariance and quasi-conformality, since $\nu_x(A)=\nu_{gx}(gA)$ for any Borel subset $A\subset \pG$ and $\nu_x$ and $\nu_{gx}$ are absolutely continuous with respect to each other.
\end{proof}

As explained in the introduction, we associate the following Poincar\'e series to $\mu$ and to the word distance, by setting
\begin{align*}
\Theta_\Gamma({r, s}) := \sum_{x\in \Gamma} G(e, x |r) \mathrm{e}^{-s|x|} 
= \sum_{n\ge 0} H_r(n) \mathrm{e}^{-sn}
=G(e,e|r)\cdot \sum_{x\in \Gamma} \mathrm{e}^{-s|x|-|x|_r}
\end{align*}
where we recall that $H_r(n)=\sum_{x\in S_n} G(e,x|r)$
and that the
\textit{critical exponent}  $\omega_\Gamma(r)$ is defined by
$$\omega_\Gamma(r) = \limsup\limits_{n \to \infty} \frac{1}{n} \log H_r(n).$$
The group $\Gamma$ is of \textit{divergent} (resp.
\textit{convergent}) type {for the Green function} if $\Theta_\Gamma({r,s})$ is divergent (resp. convergent) at
$s=\omega_\Gamma(r)$. 

Recall $\mathfrak {d}_r(x,y):=\omega_\Gamma(r)|x^{-1}y| +|x^{-1}y|_r$. 
\begin{lem}\label{NewPoincareSeries}
The series defined as follows
$$\forall s>0,\;\mathcal P_{\Gamma}(s):=\sum\limits_{x \in \Gamma}  \mathrm{e}^{- s \mathfrak d_r(x,y)}$$
has critical exponent $1$, and the divergence of $\mathcal P_{\Gamma}(s)$   at $s=1$ is equivalent to that of $\Theta_\Gamma({r,s})$ at $s=\omega_\Gamma(r)$.    
\end{lem} 
\begin{proof}
%{\color{blue}This new Poincar\'e series can be rewritten as
%$$\mathcal{P}_\Gamma(s)=\sum_{x\in \Gamma}G(e,x|r)^s\mathrm{e}^{-s\omega_\Gamma(r) |x|}.$$
%It is not a priori obvious that this series has critical exponent 1, i.e.
%$$\limsup \frac{1}{n\omega_\Gamma(r)}\log \sum_{x\in S_n}G(e,x|r)=\limsup \frac{1}{n}\log \sharp \{x,\omega_\Gamma(r)|x|+|x|_r\in [n,n+1)\}.$$
By \cite[Lemma~2.1]{GouezelLalley}, $G(e,x|R)$ goes to 0 as $|x|$ tends to infinity.
Thus, $G(e,x|r)\leq G(e,x|R)\leq 1$ for large enough $x$.
In particular, we see that for $s> 1$,
$$\sum_{x\in \Gamma}G(e,x|r)^s\mathrm{e}^{-s\omega_\Gamma(r) |x|}\leq C \sum_{x\in \Gamma}G(e,x|r)\mathrm{e}^{-s\omega_\Gamma(r) |x|}.$$
The right-hand side converges, so the left-hand side also converges.
Conversely, for $s<1$,
$$\sum_{x\in \Gamma}G(e,x|r)\mathrm{e}^{-s\omega_\Gamma(r) |x|}\leq C \sum_{x\in \Gamma}G(e,x|r)^s\mathrm{e}^{-s\omega_\Gamma(r) |x|}.$$
The left-hand side diverges, so the right-hand side also diverges.
Thus, the critical exponent of $\mathcal{P}_\Gamma(s)$ is 1.

Also, note that $\mathcal{P}_\Gamma(1)=\Theta_\Gamma(r,\omega_\Gamma(r))$, so the second conclusion of the lemma follows.
\end{proof}

 Write
$\mu(f) = \int f d\mu$ for a continuous function $f \in C(\pG)$.
We endow $\mathcal M(\pG)$ with the weak topology.
That is, a sequence $\mu_n \in \mathcal M(\pG)$ converges to $\mu$ if $\mu_n(f)
$ converges to $ \mu(f)$ for any $f \in C(\pG)$.
Equivalently, by the Portmanteau Theorem \cite[Theorem~2.1]{Bill99}, $\mu_n$ converges to $\mu$ if
$\liminf\limits_{n \to \infty}\mu_n(U) \ge \mu(U)$ for any open set
$U \subset \pG$.

We start by constructing a family of measures $\{\nu_x^s\}_{x\in \Gamma}$ supported on $\Gamma$ for any $s >1$. First,
assume that $\mathcal P_\Gamma(s)$ is divergent at $s=1$. Set
$$\nu^s_x = \frac{1}{\mathcal P_{\Gamma}(s)} \sum\limits_{z \in \Gamma} \mathrm{e}^{-s\mathfrak d_r(x,z)} \cdot \dirac{z},$$
where $s >1$ and $x \in \Gamma$. Note that $\nu^s_x$ is a probability
measure.

On the contrary, assume that $\mathcal P_{\Gamma}(s)$ is convergent at $s=1$, Patterson introduced in
\cite[Lemma 3.1]{Patt} a monotonically increasing function $H: \mathbb R_{\ge 0} \to \mathbb R_{\ge 0}$ with the
following property:
\begin{equation}\label{patterson}
\forall \epsilon >0, \;\exists t_\epsilon,\; \forall t
> t_\epsilon,\; \forall a>0:\quad H(a+t) \le \exp(a\epsilon) H(t).
\end{equation}
and such that the following modified series $$\mathcal P_{\Gamma}'(s):=\sum\limits_{x \in \Gamma} H(\mathfrak d_r(x,z)) \cdot \mathrm{e}^{-s\mathfrak d_r(x,z)}$$ is divergent for $s \le 1$ and convergent for $s>1$.
Then define measures as follows:
$$\nu^s_x = \frac{1}{\mathcal P_{\Gamma}'(s)} \sum\limits_{z \in \Gamma} \mathrm{e}^{-s \mathfrak d_r(x,z)} \cdot H(\mathfrak d_r(x,z)) \cdot \dirac{z},$$
where $s >1$ and $x \in \Gamma$.

Choose $s_i \searrow 1$ such that $\nu_x^{s_i}$ are convergent in
$\mathcal M(\pG\cup\Gamma)$ for all $x\in \Gamma$.  Let $\nu_x = \lim \nu_x^{s_i}$ be the limit
measures, which are called \textit{Patterson-Sullivan measures} associated with the Poincar\'e series $\mathcal{P}_\Gamma$. Note that forcing the Poincar\'e series to be divergent at 1, we have $\nu_x(\pG) = 1$. In the sequel, we write PS-measures as shorthand for
Patterson-Sullivan measures.
We also write $\partial\Gamma^{\mathrm{con}}$ for the set of conical limit points in the Bowditch boundary.

\begin{prop}\label{prequasiconf}
The PS-measures $\{\nu_x\}_{x\in \Gamma}$ on the Bowditch boundary are absolutely continuous with respect to each other and satisfy 
\begin{align}
\label{LowerBoundEQ} \forall \xi\in \pG: \quad \frac{d\nu_x}{d\nu_y}(\xi) \ge \mathrm{e}^{-\mathfrak d_r(x,y)},\\
\label{cdensity}\forall \nu_y\; \text{a.e.}\; \xi\in \pG^{\mathrm{con}}: \quad \frac{d\nu_{x}}{d\nu_{y}}(\xi) \asymp \mathrm{e}^{-\omega_\Gamma(r) B_\xi (x, y)}K_{\xi}(x,y|r), 
\end{align}
where the implicit constant is independent of $x,y$ and $\xi$.
\end{prop}

\begin{rem}
If $\Gamma$ is of divergent type for Green function, then Theorem \ref{divergencetypeandatoms} below says that PS-measures have no atoms on Bowditch boundary and give full measure to conical limit points, so (\ref{cdensity})  holds for $\nu_y$-a.e. $\xi\in \pG$. In this case, $\nu$ is an 
$\omega_\Gamma(r)$-dimensional quasi-conformal density.

%\item
%As Floyd boundary surjects onto Bowditch boundary, we could see  that if $\Gamma$ is of divergent type, the PS-measures on Floyd boundary have no atoms and thus form also a quasi-conformal density.
%\end{enumerate}
\end{rem}
\begin{proof}
Since $\mathfrak d_r$ satisfies the triangle inequality and $\lim_{t\to\infty} \frac{H(a+t)}{H(t)}=1$, 
we see that $\{\nu_x: x\in \Gamma\}$ are absolutely continuous with respect to each
other,\begin{align*}
\mathrm{e}^{-\mathfrak d_r(x,y)}\le \frac{d\nu_x}{d\nu_y}(\xi) \le \mathrm{e}^{\mathfrak d_r(x,y)}.   
\end{align*}

We now verify the quasi-conformality at conical limit points. 
We only consider the case where $\Gamma$ is of convergent type, the divergent type being simpler. 
Let $\epsilon> 0$ and $t_\epsilon$ the number be given by~(\ref{patterson}) for the function $H$. 
Let $\phi=\frac{d\nu_x}{d\nu_y}$ be the Radon-Nikodym derivative, uniquely defined up to a $\nu_y$-null set.
Let $\xi \in \pG$ be a conical limit point and consider the open neighborhood $V$ of $\xi$ and the uniform constant $C$ given by Lemma~\ref{BusemanEstLem}. 

Let $f$ be a
continuous function supported in $V$.
One can choose $V$ also such
that $\mathfrak d_r(y, z) > t_\epsilon$ for all $z \in V$. If  $z \in
V$ satisfies  $\mathfrak d_r(x, z) >\mathfrak d_r(y, z)$, then  we have
\begin{align*}
H(\mathfrak d_r(x,z)) &= H(\mathfrak d_r(x,z)-\mathfrak d_r(y,z)+\mathfrak d_r(y,z)) \\
& \le \mathrm{e}^{\epsilon[\mathfrak d_r(x,z)-\mathfrak d_r(y,z)]} \cdot H(\mathfrak d_r(y,z)) \\
& \le \mathrm{e}^{\epsilon (C + \mathfrak B_\xi(x, y))}\cdot H(\mathfrak d_r(y, z)).
\end{align*}
%Let $C_\epsilon= \max\{1,  \mathrm{e}^{\epsilon (C + B_\xi(x, y))}\}$
Since $H$ is increasing,
we have
\begin{equation}\label{prequasiconfeq}
C_\epsilon^{-1} H(\mathfrak d_r(y, z))\le H(\mathfrak d_r(x, z)) \le C_\epsilon  H(\mathfrak d_r(y, z)), 
\end{equation}
where $C_\epsilon= \mathrm{e}^{\epsilon (C + \mathfrak B_\xi(x, y))}>1$ depends on $\epsilon, C$ and $ (x,y)$, but not on $z\in V$.
Note that $C_\epsilon\to 1$ as $\epsilon\to 0$.
By symmetry, the   conclusion~(\ref{prequasiconfeq}) also holds if
$\mathfrak d_r(x, z) < \mathfrak d_r(y, z)$ for  $z \in V$.

Using~(\ref{prequasiconfeq}),  we get the following estimates.
First,
\begin{align*}
\nu^s_x(f) & = \frac{1}{\mathcal P_{\Gamma}(s)} \sum\limits_{z \in V}
\mathrm{e}^{-s \mathfrak{d}_r(x,z)} H(\mathfrak d_r(x,z)) {f(z)} \\
& \leq  C_\epsilon \mathrm{e}^{-s\mathfrak B_{z}(x, y))}\cdot \frac{1}{\mathcal P_{\Gamma}(s)}
\sum\limits_{z \in V}
\mathrm{e}^{-s\mathfrak d_r(y, z)} H(\mathfrak{d}_r(y, z)) {f(z)}\\
& \asymp   C_\epsilon  \mathrm{e}^{-s\mathfrak B_\xi(x, y)} \nu^s_y(f) 
\end{align*}
and second
\begin{align*}
\nu^s_x(f) & = \frac{1}{\mathcal P_{\Gamma}(s)} \sum\limits_{z \in V}
\mathrm{e}^{-s \mathfrak{d}_r(x,z)} H(\mathfrak d_r(x,z)) {f(z)} \\
& \geq  C_\epsilon^{-1} \mathrm{e}^{-s\mathfrak B_{z}(x, y))}\cdot \frac{1}{\mathcal P_{\Gamma}(s)}
\sum\limits_{z \in V}
\mathrm{e}^{-s\mathfrak d_r(y, z)} H(\mathfrak{d}_r(y, z)) {f(z)}\\
& \asymp   C_\epsilon  \mathrm{e}^{-s\mathfrak B_\xi(x, y)} \nu^s_y(f) 
\end{align*}
where the implicit constants  depend only on $C$ but not on $\epsilon$.
Letting $s\to 1$, we get
$$C_\epsilon^{-1} \mathrm{e}^{-\mathfrak B_\xi(x, y)}\nu_y(f)\prec \nu_x(f) \prec C_\epsilon \mathrm{e}^{-\mathfrak B_\xi(x, y)} \nu_y(f)$$
for any continuous function $f$ supported in $V$. 
As $\epsilon\to 0$, $C_\epsilon\to 1$, hence it follows that
$$ \phi(\xi) \asymp \mathrm{e}^{-
\mathfrak B_\xi (x, y)}$$ for $\nu_y$-a.e. conical limit point $\xi \in \pG$.
\end{proof}

\subsection{Shadow Lemma}
The key observation in the theory of PS-measures is the Sullivan
Shadow Lemma shedding some light on the relation between the PS-measure and
the geometric properties of the boundary.

Recall that $\partial \Gamma$ is a visual boundary for the Cayley graph $\mathrm{Cay}(G,S)$: any two distinct points $x,y\in \mathrm{Cay}(G,S)\cup \partial \Gamma$ can be connected by a geodesic.  See \cite[Proposition 2.4]{GePoJEMS}.

\begin{defn}
Let $C>0$ and $x \in\Gamma$. The \textit{shadow $\Pi_C(x)$} at $x$ is the set of points
$\xi \in \pG$ such that there exists some geodesic $\gamma=[e, \xi]$
intersecting $B(x, C)$.

The \textit{partial shadow $  \Psi_{C}(x)$} at $x$
is the set of points $\xi \in \pG$ such that   some geodesic
$[e,\xi]$ contains a transition point $C$-close to $x$.
\end{defn}

We now state the Shadow Lemma in our context, whose proof follows closely  the   proofs  of   \cite[Lemmas 4.1 \& 4.2]{YangPS} with Lemma~\ref{BusemanEstLem} replacing Lemma~2.19 there.
Let us denote by $\Psi^{\mathrm{con}}_C(g)$ the set of all conical limit points in
$\Psi_C(g)$.

\begin{lem}[Shadow Lemma]\label{ShadowLem}
Let $\{\nu_x\}_{x \in \Gamma}$ be an $\omega_\Gamma(r)$-dimensional PS measures on the Bowditch boundary $\pG$. Then
there exists $C_0> 0$ such that for any $C\ge C_0$ and $x\in \Gamma$ the following two inequalities hold
\begin{align}
\label{LowBoundEQ} \mathrm{e}^{-\omega_\Gamma(r) |x|}G(e,x|r)\prec_C \,  
&\nu_e(\Psi_{C}(x))\le \nu_e(\Pi_{C}(x)),\\
\label{UpperBoundEQ}
&\nu_e(\Psi^{\mathrm{con}}_{C}(x)) \prec_C \, \mathrm{e}^{-\omega_\Gamma(r) |x|} G(e,x|r).
\end{align}
%% for any $C\ge C_0$ and $x\in \Gamma$.
\end{lem}

\begin{rem}
If $\nu_e$ has no atoms at parabolic points which form a countable subset of the Bowditch boundary, then we obtain the full strength of the partial shadow lemma without having to restrict our attention to conical points. The upper bound~(\ref{UpperBoundEQ}) for the partial shadow uses the relative Ancona inequalities (Lemma~\ref{weakAncona}), while  it is unknown whether the upper bound holds for the usual shadow.     
\end{rem}

\begin{proof}
%{\color{blue}Maybe we can add the following introductory explanations, as we haven't introduced loxodromic elements before
%An element $x$ of $\Gamma$ is called loxodromic if it has exactly two fixed points $x_+$ and $x_-$ in $\partial\Gamma$.
%In such case, $x$ acts via North-South dynamics on $\partial\Gamma$ in the sense that for any $\xi\neq x_{\pm}$, $x^{\mp n}\xi$ converges to $x_{\mp}$ as $n$ goes to infinity.
%Then $x_+$ is called the attracting fixed point of $x$ and $x_-$ is called its repelling fixed point.
%It is well known that the set of attracting and repelling fixed points of loxodromic elements is dense in $\pG$, see for instance \cite[Proposition~4.7]{Quint}.\ywy{These facts follow from the fact that the action on Gromov boundary is a convergence group action. One should cite Tukia and Bowditch in this coarse setup. I will find appripriate referece later on. I added more explanation at the begining of this section. Please let me know if we need more clarification.}}

Let $F$ be a set of three loxodromic elements with pairwise disjoint fixed points. For each     $f\in F$, let $\alpha:=\cup_{i\in \mathbb Z} f^i[e,f]$ be an $\langle f\rangle$-invariant quasi-geodesic between   two fixed points $f_-,f_+\in\pG$. Let $U_f\subset \pG$ be an open neighborhood of $f_+$ so that for any $\eta\in U_f$, the projection of $\eta$ to the axis $\alpha$ has a distance to $e$ at least $C$.  By \cite[Lemma 2.4]{DWY}, for any $x\in \Gamma$,  there exists $f\in F$ so that $[x^{-1},\eta]$ contains a transition point $C$-close to  $e$.
Thus, $U_f\subset x^{-1}\Psi_{C}(x)$. 

As $\Gamma$ acts minimally with a dense orbit in  $\pG$, the $\Gamma$-orbit 
of any open set $U\subset \pG$ covers   $\pG$, so $U$ have positive $\nu_e$-measure. 
Hence, setting
$$D=\min\{\nu_e(U_f): f\in F\}>0$$
which is independent of $x$, we have  
$$
\nu_e(x^{-1}\Psi_{C}(x)) \ge D.
$$ 
Since $\nu_x $ is  equivariant,   the lower bound     in (\ref{LowerBoundEQ}) implies 
\begin{align*}
\nu_e(\Psi_{C}(x))={\nu_{x^{-1}}(x^{-1}\Psi_{C}(x))}&\ge   \mathrm{e}^{-\omega_\Gamma(r) |x|}\frac{G(e,x|r) }{G(e,e|r)}\cdot {\nu_e( x^{-1}\Psi_C(x)) }\\
&\ge D \mathrm{e}^{-\omega_\Gamma(r) |x|}\frac{G(e,x|r) }{G(e,e|r)}.
\end{align*}
Since $G(e,e|r)$ is bounded by $G(e,e|R)$, this concludes the
proof of the lower bound.
 
For any $\xi \in x^{-1} \Psi^{\mathrm{con}}_C(x)$, there is a geodesic
$\gamma$ from $x^{-1}$ to $\xi$ which intersects   $B(e, C)$ and contains a transition point. Thus, $|B_\xi(x^{-1},e)
- d(x^{-1}, e)| \le 2C.$ By the relative Ancona inequalities (Lemma~\ref{weakAncona}), there is a   constant $C_1$ independent of $r$ such that $$K_\xi(x^{-1}, e|r) = \lim_{z\to \xi}\frac{G(x^{-1},z|r)}{G(e,z|r)}
\le C_1 G(e,x|r).$$
Also, by (\ref{cdensity}) there
is a constant $C_2>0$ such that for $\nu_e$-a.e. conical limit point $\xi \in \pG^{\mathrm{con}}$, 
$$
\frac{d\nu_{x^{-1}}}{d\nu_e}(\xi) \le C_2 \mathrm{e}^{-\omega_\Gamma(r) B_\xi(x^{-1},e)} K_\xi(x^{-1},
e| r).
$$
Combining together the above estimates, we have 
\begin{align*}
\nu_e(\Psi^{\mathrm{con}}_C(x)) & = \nu_{x^{-1}}(x^{-1}\Psi^{\mathrm{con}}_C(x)) \\
& \le  C_2\int_{x^{-1}\Psi^{\mathrm{con}}_C(x)} \mathrm{e}^{-\omega_\Gamma(r) B_\xi(x^{-1},e)} K_\xi(x^{-1},
e|r) d \nu_e(\xi)  \\
& \leq C_1C_2 \mathrm{e}^{2C\omega_\Gamma(r)} \cdot \mathrm{e}^{-\omega_\Gamma(r) |x|} G(e,x|r), 
\end{align*}
which finishes the proof of the upper bound.
\end{proof}

\begin{prop}\label{PosMeasonConcialProp}
Suppose that $\nu_e$ gives positive measure to the set of conical limit points.
Then $\Gamma$ is of divergent type for the Green function.
\end{prop}
\begin{proof}
List $\Gamma = \{x_1, \dotso, x_i, \ldots\}$ such that for all $i$, $|x_i|\le | x_{i+1}|$.  Let $C_0$ be given by Lemma~\ref{ShadowLem}. For any $C>C_0$, set  
\begin{equation}\label{conicalset}
A_{C} := \bigcap_{n=1}^{\infty} \bigcup_{i=n}^{\infty}
\Psi_{C}^{\mathrm{con}}(x_i). 
\end{equation}
By Lemma
\ref{ConCharLem}, we have $\partial^{\mathrm{con}}\Gamma =A_{C}$. %{\color{blue}we have in fact this equality here, no need to take union over $C$, right ?} \ywy{Yes.}
In
other words, any conical limit point is shadowed infinitely many times
by elements of $\Gamma$.

%Let us choose  $C$ large enough  such that  $\nu_e(\partial^{\mathrm{con}} \Gamma)\ge \mu_e(A_{C}) > 0$. 
We claim that $\Theta_{\Gamma}(r,\omega_\Gamma(r))$ is divergent. Recall that
$$\Theta_{\Gamma}(r,\omega_\Gamma(r)) = \sum\limits_{x \in \Gamma} \mathrm{e}^{-\omega_\Gamma(r)|x|}G(e,x|r).$$
By Lemma~\ref{ShadowLem}, we see that  
\begin{align*}
\sum\limits_{|x|\geq n} \mathrm{e}^{-\omega_\Gamma(r)|x|}G(e,x|r) 
& \succ \sum\limits_{|x|\ge n} \nu_e(\Psi^{\mathrm{con}}_{C}(x)) \\
&\succ \nu_e\big(\bigcup\limits_{|x|\ge n} \Psi^{\mathrm{con}}_C(x)\big) \succ \nu_e(A_{C}) > 0.
\end{align*}
This is true for all $n > 0$ and $\nu_e(A_{C})$ is independent of $n$.
Thus, $\Theta_{\Gamma}(r,s)$ is indeed divergent at
$s=\omega_\Gamma(r)$.
\end{proof}

\begin{thm}\label{divergencetypeandatoms}
%Let $\pG$ be the Bowditch boundary of $\Gamma$.
If $\Gamma$ is of divergent type for Green function, then $\nu_e$ has no atoms. Otherwise, $\nu_e$ is purely atomic and  supported on the set of parabolic points.  %\ywy{I'm not sure what you mean by "fully atomic". It seems to say that every parabolic points are atoms. This may not be true : it is possible that if there are two conjugacy class of parabolic points, one class of parabolic points has zero measure, and the PS-measure is supported on the other one.} {\color{blue}By fully atomic, I meant that $\nu_e$ only has atoms, not that all parabolic points are atoms. Sorry, the more standard term is purely atomic.}
\end{thm}
\begin{proof}
First of all, the otherwise statement follows from Proposition~\ref{PosMeasonConcialProp}.
Indeed, if $\Gamma$ is of convergent type, then $\nu_e$ cannot give positive measure to conical limit points, hence it is supported on the set of parabolic limit points.
Since this set is countable, $\nu_e$ is necessarily purely atomic.

Assume now that $\Gamma$ is of divergent type.
Let $q\in \pG$ be a bounded parabolic point so that the stabilizer $P\in \mathbb P$ acts co-compactly on $\pG\cup \Gamma \setminus \{q\}$. Let $K\subset \Gamma\cup\pG\setminus \{q\}$ be a compact fundamental domain.
%As a probability measure takes at most countable atoms, we can     enlarge $K$ by a finite neighborhood  so that    the     boundary $\partial K$ is $\nu_e$-null.
For a point $y\in \Gamma$, we let $\pi_P(y)$ be the set of nearest point-projections $$\pi_P(y)=\{p\in P: d(y,p)=d(y,P)\}.$$ Define $\pi_P(A):=\cup_{a\in A} \pi_P(a)$. As $\partial K$ is disjoint with $\partial P=\{q\}$, the shortest projection $Z:=\pi_P(K\cap \Gamma)$ has  bounded  diameter by \cite[Proposition 3.3]{GP16}. By enlarging $Z$, assume without loss of generality that $1\in Z$. 

Note that  a maximal parabolic $P$ has  the contracting property by \cite[Proposition 8.5]{GP16}: any geodesic $[x,y]$ with large projection to $P$ is uniformly close to $\pi_P(x)$ and $\pi_P(y)$. Thus, for any $y\in \Gamma\cap K$ and $p\in P$,  any geodesic $[e,py]$    passes within uniformly bounded distance of $p\in pZ$. Consequently, $|p y| \simeq |p| + |y|$. Moreover, $[e,py]$ exits $P$ with bounded distance to $pZ$, so  $p$ is within a bounded distance of a transition point on $[e,py]$. Now, by the relative Ancona inequalities in Lemma~\ref{weakAncona}, $G(e, p y|r) \asymp G(e, p|r) G(e, y|r)$. Here the implicit constants in $\simeq$ and $\asymp$ are independent of $y$ and $p$.

% {\color{blue}We need to explain what the contracting property is}, there exists a constant $C$ depending on $P$ so that for  any $y\in K\cap \Gamma$ % and any $p\in Z$,
%we have that % $|d(y,e) - d(y, p)-d(e,p)|\le C$  and
%$[e,y]$ contains a transition point $z$ which is $C$-close to $Z$. %We thus have $G_r(e,y) \asymp G_r(e,z)G_r(z,y)$.
%Thus for any $y \in K \cap \Gamma$ and $p \in P$ with $|p|$ sufficiently large, we have from the relative thin triangle property for transition points (see for instance~\cite[Proposition~7.1.1]{GP16}, %\cite[Proposition 4.6]{Sistometricrelhyp})
%that $p$ is within a bounded distance of a transition point on $[e, p y]$. As a consequence, $|p y| \simeq |p| + |y|$ and by the relative Ancona inequalities (Lemma~\ref{weakAncona}), $G(e, p y|r) \asymp G(e, p|r) G(e, y|r)$, where the implicit constants depend on $P$ and $Z$.

We can estimate the $\nu_e^s$-measure of an open neighborhood 
$$U_n= \{q\}\cup \{pK: p\in P; |p|\ge n\}$$   as follows
$$\nu_e^s(U_n) \le \sum_{|p|>n} \nu_e^s(pK) \le {\frac{1}{\mathcal{P}_{\Gamma}(s)}} \sum_{|p|>n} \sum_{y\in K} \mathrm{e}^{-s\omega_\Gamma(r)|p y|} [G(e, p y|r)]^s,$$
hence
\begin{align*}
\nu_e^s(U_n) &\prec {\frac{1}{\mathcal{P}_{\Gamma}(s)}}  \sum_{|p|>n} \sum_{y\in K} \mathrm{e}^{-s\omega_\Gamma(r)|p|} \mathrm{e}^{-s\omega_\Gamma(r)|y|} [G(e,y|r)]^s [G(e,p|r)]^s\\
& \prec \nu_e^s(K) \sum_{|p|>n} \mathrm{e}^{-s\omega_\Gamma(r) |p|}  [G(e,p|r)]^s.    
\end{align*}

{By~\cite[Corollary 3.9]{DWY}},   $\Theta_P(r,\omega_\Gamma(r))$ is convergent.
Also, by the Portmanteau Theorem \cite[Theorem~2.1]{Bill99}, $\limsup_{s\to 1}\nu_e^s(K)\leq \nu_e(K)$.
Letting $s\to 1$ and then $n\to\infty$, we see that $\nu_e(U_n)\to 0$. 
Thus, $\nu_e$ has  no atoms at parabolic limit points.
By Lemma~\ref{ConCharLem},  a  conical limit point  $\xi$ is contained in infinitely many partial shadows $\Psi_{C}(x_n)$. By Lemma~\ref{ShadowLem}, as $x_n\to \xi$, the $\nu_e$-measure of $\Psi_C(x_n)$ tends to $0$, so conical limit points are not atoms as well.
\end{proof}

%{\color{blue}I have no objection with going back to fundamental domains.
%Also, using the Portmanteau Theorem, we do not need to prove a priori that $K$ has $\nu_e$-null boundary, so I edited accordingly the proof.}\ywy{Thanks, you are right!}

\section{Convergent Poincar\'e series and applications}\label{SectionRHG}
%{\color{blue}Maybe another name for this final section would be more suited. I would suggest something referring to the parabolic gap condition, as this is our main goal.}

%\ywy{
%Transition of Parabolic gap Property? Both are fine with me. Another application I liked is about growth tightness. 
%}

This final section is devoted to answering some questions raised in \cite{DWY} where we initiated the study of branching random walks on relatively hyperbolic groups.
In particular, we end here the proof of our main result, Theorem~\ref{mainthm}.

One important notion that was coined in \cite{DWY} is the \textit{parabolic gap for the Green function} whose definition we now recall.
Let $\mu$ be a probability measure on $\Gamma$,
let $\mathbb P_0$ be a finite set of representatives of conjugacy classes of maximal parabolic subgroups and let $P\in \mathbb P_0$.
We set
$$H_{P,r}(n)=\sum_{x\in S_n\cap P}G(e,x|r)$$
and
$$\omega_P(r)=\limsup_{n\to \infty}\frac{1}{n}\log H_{P,r}(n).$$
\begin{defn}
If $\omega_P(r)<\omega_\Gamma(r)$, we say that $\Gamma$ has a \textit{parabolic gap} along $P$ for the Green function at $r$.
If for every $P$, for every $r\in (1,R]$, $\omega_P(r)<\omega_\Gamma(r)$, then we say that $\Gamma$ has a \textit{parabolic gap} for the Green function.
\end{defn}

One of the consequences of having a parabolic gap for the Green function is that $H_r(n)$ is roughly multiplicative and the Green function has purely exponential growth.
Namely, by \cite[Theorem~1.8]{DWY}, if $\Gamma$ is a non-elementary relatively hyperbolic group and if $\mu$ is a finitely supported admissible and symmetric probability measure on $\Gamma$, then for every $1<r\leq R$, there exist $C$ and $C'$ such that for all $n$,
\begin{equation}\label{roughmultiplicativity}
\frac{1}{C}H_{r} (n+m)\leq  H_r(n)H_r(m)\le C H_{r} (n+m)
\end{equation}
and
\begin{equation}\label{purelyexponentialgrowth}
\frac{1}{C'}\mathrm{e}^{n\omega_\Gamma(r)}\leq H_r(n)\leq C' \mathrm{e}^{n\omega_\Gamma(r)}.
\end{equation}
As proved in \cite{SWX}, these two properties~(\ref{roughmultiplicativity}) and~(\ref{purelyexponentialgrowth}) hold for all hyperbolic groups.
%and the constants $C$ and $C'$ can be chosen independently of $r$
It was proved in \cite{DWY} that if maximal parabolic subgroups are amenable and if $r<R$, then the parabolic gap condition holds, hence so do the properties~(\ref{roughmultiplicativity}) and~(\ref{purelyexponentialgrowth}).
Under additional assumptions on the random walk, this was also proved at $R$.
Among the unanswered problems in \cite{DWY} are the following questions.
Does the parabolic gap condition holds at $R$ as soon as maximal parabolic subgroups are amenable ?
%Can the constants $C$ and $C'$ in~(\ref{roughmultiplicativity}) and~(\ref{purelyexponentialgrowth}) be chosen independently of $r$ ?
Does there exist examples of relatively hyperbolic groups endowed with a finitely supported admissible symmetric probability measures such that the properties~(\ref{roughmultiplicativity}) and~(\ref{purelyexponentialgrowth}) fail ?
As particular cases of our work, we answer these two questions here.

\subsection{First return  to maximal parabolic subgroups}
We gather the results of Section~\ref{Sectiongrowthrateat1} to prove that whenever maximal parabolic subgroups are amenable, the parabolic gap for the Green functions holds.

Consider the first return kernel $p_{r,P}$ to $P$ for $r \mu$ defined by
\[
  p_{r,P}(x,y)=\sum_{n\geq 1}\sum_{z_1, \ldots ,z_{n-1}\notin P} r^n p(x,z_1)p(z_1,z_2) \cdots p(z_{n-1},y).
\]
Denote by $p_{r, P}^{(n)}$ the $n$th convolution power of this transition kernel and by $G_{r,P}$ the associated Green function.
As explained in the introduction, for every $x$, $y\in P$, $G_{r,P}(x,y|1)=G(x,y|r)$ by~\cite[Lemma~4.4]{DG21}.
Set $t_{r, P} = \sum_{x \in P} p_{r, P}(e, x)$.  Then $t_{r, P}^{-1} p_{r, P}$ is a symmetric admissible and $P$-invariant transition kernel, thus defines a random walk on $P$.   

\begin{prop}
  \label{omegaPlesst}
  Let $\Gamma$ be a relatively hyperbolic group and let $P\in \mathbb P_0$.
Consider an admissible and symmetric probability measure $\mu$ on $\Gamma$. If $t_{r, P} \leq 1$, then  we have $\omega_P(r) \leq 0$;  in particular, $\omega_P(r) < \omega_{\Gamma}(r)$. 
\end{prop}

\begin{proof}
    Denote $t = t_{r, P}$ for simplicity and let $G_t$ be the Green function associated with $t^{-1} p_{r, P}$.  By Proposition~\ref{upperboundgrowthrateat1}, $\sum_{x \in P, \, |x| = n} G_t(e, x | 1)$ has growth rate at most 0.  Since $t \leq 1$ we have that
    \[
      G(e, x | r) = G_{r, P}(e, x | 1) = G_t(e, x | t) \leq G_t(e, x | 1). 
    \]
    Therefore $\omega_P(r) \leq 0$. 
  \end{proof}

\begin{prop}\label{propamenableparabolic}
Let $\Gamma$ be a relatively hyperbolic group and let $P$ be a maximal parabolic subgroup.
Consider an admissible and symmetric probability measure $\mu$ on $\Gamma$.
If $P$ is amenable, then for every $r\leq R$, $\omega_P(r)\leq 0$.  
\end{prop}

\begin{proof}
  Since $G(e, x | r) < \infty$, we deduce that the spectral radius of $p_{r, P}$ is at most $1$.  
  By~\cite[Corollary~12.5]{Woessbook} and the fact that $P$ is amenable, the spectral radius of $t_{r, P}^{-1} p_{r, P}$ is $1$ and hence $t_{r, P} \leq 1$.  The result follows from Proposition~\ref{omegaPlesst}.  
\end{proof}

Note that we do not need to assume that $\mu$ is finitely supported in this proposition, although we need this assumption in the following corollary, which  also relies on \cite[Theorem~1.8]{DWY} mentionned above, where the assumption is crucially used.
% \begin{proof}
% Since $\omega_P(r)$ is non-decreasing, we only need to prove that $\omega_P(R)\leq 0$.
% Consider the first return kernel $p_{R,P}$ to $P$ defined by
% $$p_{R,P}(x,y)=\sum_{n\geq 1}\sum_{z_1,...,z_{n-1}\notin P}R^np(x,z_1)p(z_1,z_2)...p(z_{n-1},y).$$
% Denote by $G_{R,P}$ the associated Green function.
% As explained in the introduction, for every $x,y\in P$, $G_{R,P}(x,y|1)=G(x,y|R)$ by \cite[Lemma~4.4]{DG21}.
% Since $G(x,y|R)$ is finite, we deduce that the spectral radius of $p_{R,H}$ is at most 1.
% Let $t=\sum_{x\in P}p_{R,P}(e,x)$, so that $p_t=t^{-1}p_{R,P}$ is a symmetric admissible and $P$-invariant transition kernel on $P$ with total mass 1, thus defines a random walk on the amenable group $P$.
% By \cite[Corollary~12.5]{Woessbook}, the spectral radius of $p_t$ is 1, hence $t\leq 1$.
% Let $G_t$ be the Green function associated with $p_t$ and note that for $x,y\in P$,
% $$G_t(x,y|t)=G_{R,P}(x,y|1)=G(x,y|R).$$
% We deduce that $\omega_P(R)$ coincides with the growth rate at $t$ of $G_t$.
% Consequently, Proposition~\ref{upperboundgrowthrateat1} shows that $\omega_P(R)\leq 0$.
% \end{proof}

\begin{cor}\label{coramenableparabolic}
Let $\Gamma$ be a relatively hyperbolic group endowed with a finitely supported admissible and symmetric probability measure $\mu$.
Assume that maximal parabolic subgroups of $\Gamma$ are amenable.
Then $\Gamma$ has a parabolic gap for the Green function and so~(\ref{roughmultiplicativity}) and~(\ref{purelyexponentialgrowth}) hold.
\end{cor}

%{\color{blue}
%It is great that you both could figure a proof for the upper bound $|x|^\beta G(e,x|r)$ below.
%I suggest the following organisation for this section, what do you think ?
%In the meantime, if we cannot prove that $r_*$ is the transition for $\Lambda$ intersecting parabolic limit points, we can still prove that it is the transition for the BRW accumulating in $P$.
%I suggest the following formulation, in place of what is written below, if we cannot figure out a complete proof.
%Also, $r_*$ is already the transition point for other properties and $r_0$ is the transition point in Theorem~\ref{thmbitree}, so I chose to write $r_\sharp$ for the transition point for $t_{r,P}$.

%\begin{rem}
Note that $t_{1, P}$ is the probability that the random walk associated to $\mu$ eventually returns to $P$. We see that $t_{1, P} < 1$ for all $P \in \mathbb{P}_0$.
Otherwise, the random walk would visit $P$ infinitely many times with positive probability, which in turn would imply that it accumulates at the parabolic limit points fixed by $P$.
This would contradict the fact that the random walk almost surely converges to a conical limit point \cite[Theorem~9.14]{GGPY}.

More generally, consider a branching random walk on $\Gamma$ whose step distribution is given by $\mu$ and with mean offspring $r$.
Following \cite{CandelleroGilchMuller}, by collecting all particles returning to $P$, one gets a Galton-Watson process with mean offspring $t_{r,P}$, see precisely the proof of \cite[Proposition~4.3]{CandelleroGilchMuller}.
Consequently, the branching random walk returns to $P$ infinitely many times if and only if $t_{r,P}>1$.
In such case, the branching random walk must accumulate in $P$.
Moreover, the same arguments as in \cite[Proposition~4.4]{CandelleroGilchMuller} show that there are almost surely infinitely many cosets $gP$, such that the branching random walk accumulates in $gP$.
On the contrary, if $t_{r,P}\leq 1$, then for all coset $gP$, the branching random walk eventually leaves $gP$.
This follows from the fact that $g$ is visited finitely many time almost surely and that starting a branching random walk at $g$, it comes back to $gP$ only finitely many times, see also \cite[Proposition~4.5]{CandelleroGilchMuller} for a more detailed proof.

Define now
\[
r_{\mathcal{P}} = \sup \{ r > 1 \colon t_{r, P} \leq 1 \text{ for all } P \in \mathbb{P}_0 \}.
\]
Then for $1<r\leq r_\mathcal{P}$, we have $\omega_P(r)\leq 0$ for all $P\in \mathbb P_0$, hence
$\Gamma$ has a parabolic gap for the Green function and $H_r(n)$ has purely exponential growth.
Moreover, $r_\mathcal{P}$ is the transition for the branching random walk spending infinite times in maximal parabolic subgroups : if $r \leq r_\mathcal{P}$, then the branching random walk eventually leaves every coset $gP$, $P\in \mathbb P_0$, while if $r>r_\mathcal{P}$, it accumulates in infinitely many cosets $gP$ for at least one of the $P\in \mathbb P_0$. 

Recall that the limit set $\Lambda$ of a branching random walk inside the Bowditch boundary $\partial\Gamma$ is the set of accumulation points in $\partial\Gamma$ of the trace of $\mathrm{BRW}(\Gamma,\nu,\mu)$, which is the set of elements of $\Gamma$ that are ever visited by the branching random walk.
We take the occasion to derive the following consequence which sheds some light on the geometry of the limit set.
\begin{prop}\label{propparaboliclimitset}
Let $\Gamma$ be a relatively hyperbolic group endowed with a finitely supported admissible and symmetric probability measure $\mu$.
Consider a probability measure $\nu$ on $\mathbb N$ with mean $r\leq R$.
Let $\Lambda$ be the limit set inside the Bowditch boundary of $\Gamma$ of the branching random walk associated with $\mu$ and $\nu$.
\begin{enumerate}
%    \lw{\item Almost surely $\Lambda$ does not contain any conical limit point. } This is not true, because the set of conical limit points is not countable. What we can indeed say is that given a conical limit point, almost surely it is not in $\Lambda$ (which roughly means that the limit set "has no atoms"). I'm not sure we should include this anyway. And note that in fact when $\Lambda$ does not contain parabolic limit points, it is included in the set of conical limit points. \lw{Yes, I agree.  It would be better to remove this. }
    \item If $r<r_\mathcal{P}$, then almost surely, $\Lambda$ does not contain any parabolic limit point.
    \item If $r>r_\mathcal{P}$, then almost surely, $\Lambda$ contains an infinite number of parabolic limit points.
\end{enumerate}
\end{prop}

Note that $\Lambda$ cannot contain all parabolic limit points in case~(2).
Otherwise, since parabolic limit points are dense, $\Lambda$ would coincide with the whole Bowditch boundary, but this is impossible since its Hausdorff dimension with respect to the shortcut distance is at most half the Hausdorff dimension of the whole boundary by \cite[Theorem~1.1, Theorem~1.2]{DWY}.

This result is a consequence of the following one.
For $x\in \Gamma$ and $C\geq0$, we denote by $\Omega(x,C)$ the \textit{partial cone} at $x$ of width $C$, which is the set of points $y\in \Gamma$ such that $x$ is within $C$ of a transition point on a geodesic from $e$ to $y$.

\begin{prop}\label{proppartialcone}
Let $\Gamma$ be a relatively hyperbolic group endowed with a finitely supported admissible and symmetric probability measure $\mu$.
For every $r\in [1,R]$, there exist  $\beta>0$ with the following property.
Consider a probability measure $\nu$ on $\mathbb N$ with mean $r\leq R$.
For any $x\in \Gamma$, the probability that the branching random walk visits $\Omega(x,C)$ is at most $C_1(1+|x|^\beta )G(e,x|r)$,
where $C_1$ is a constant.
\end{prop}

The proof of this proposition relies mostly on material from \cite{DWY} and its proof is postponed to the Appendix.

\begin{proof}[Proof of Proposition~\ref{propparaboliclimitset}]
% \lw{
% Let $\xi$ be a conical limit point.  By~\cite[Lemma~2.20]{YangPS}, $[e, \xi]$ contains a sequence of $(\eta, L)$-transition point $x_n$ and hence $\xi \in \Omega(x_n, C)$.  By Proposition~\ref{proppartialcone} we have
% \[
% \mathbf{P}(\xi \in \Lambda) \leq \mathbf{P}(\text{BRW visits } \Omega(x_n, C)) \leq |x_n|^{\beta} G(e, x_n | r), 
% \]
% which converges to $0$ as $n \to \infty$.  
%  }
  
As we saw above, if $r>r_\mathcal{P}$, then the branching random walk accumulates in infinitely many cosets $gP$ for at least one $P\in \mathbb P_0$.
In particular, the limit set $\Lambda$ contains all parabolic limit points fixed by $gPg^{-1}$, for every such coset $gP$.  %% \lw{We can only deduce that $\Lambda$ contains infinitely many such parabolic limit points.  Since the set of parabolic limit points fixed by $gPg^{-1}$ is dense in $\partial \Gamma$, if $\Lambda$ contains all of them, then $\Lambda = \partial \Gamma$. }{\color{blue}Yes, this is what is written. I added an explicit comment.} \lw{I misunderstood here.  Sorry for the noise. But the previous comment you added is great. }
Thus, we only need to prove~(1) to conclude and we assume that $r<r_\mathcal{P}$.

Recall that a sequence $x_n$ in $\Gamma$ converges to a parabolic limit point $\xi$ fixed by $gPg^{-1}$, $P\in \mathbb P_0$, if and only if the sequence of projections of $x_n$ on $gP$ tends to infinity.
Thus, $\xi\in \Lambda$ if and only if the branching random walk visits infinitely many $\Omega(x,C)$, with $x\in gP$.
Fix $g\in \Gamma$ and $P\in \mathbb P_0$.
Denote by $A_n$ the event
$$A_n=\big\{\mathrm{BRW}(\Gamma,\nu,\mu)\text{ visits } \Omega(x,C) \text{ for some } x\in gP,\text{ with } |g^{-1}x|=n\big\}.$$
Then, by Proposition~\ref{proppartialcone},
$$\mathbf P(A_n)\leq C_1(1+n)^\beta \sum_{x\in P,|x|=n}G(e,gx|r).$$
Since $g$ is a transition point on a relative geodesic from $e$ to $gx$, by relative Ancona inequalities we get
$$\mathbf P(A_n)\leq C_2(1+n)^\beta H_{P,r}(n).$$
Since $t_{r,P}<1$, $\omega_P(r)<0$ by Proposition~\ref{omegaPlesst} and so $H_{P,r}(n)$ decays exponentially fast as $n$ goes to infinity.
This proves that
$$\sum_{n}\mathbf P(A_n)<\infty.$$
By the Borel-Cantelli lemma, we deduce that almost surely, the branching random walk only visits finitely many $\Omega(x,C)$, $x\in gP$, hence $\xi\notin \Lambda$.
Since parabolic limit points are countable, this settles the proof.
\end{proof}

%%  It remains to consider the critical case $t_{r,P}=1$, i.e. $r=r_\sharp$.
%%  Do you think we can prove lower semi-contitnuity of $r\mapsto \mathbf P(\xi\in \Lambda)$ ?
%% \lw{There is an issue when we view $\mathbf{P}(\xi \in \Lambda)$ as a function of $r$: in fact, I think, when $\mathbf{P}(\xi \in \Lambda) \neq 0$, it might be a function of the offspring distribution, not only the mean $r$.  If we denote $h_{\nu}(x)$ the probability that $\xi$ in the limit set of a BRW starting at $x$ with offspring distribution $\nu$, then we have
%% \[
%% 1 - h_{\nu}(x) = \varphi (1 - P h_{\nu}(x)), 
%% \]
%% where $P h(x) = \sum_{y \in \Gamma} h(x y) \mu(y)$ is the transition operator associated to $\mu$ and $\varphi(s) = \sum_k s^k \nu(k)$ is the generating function of $\nu$. I am not sure if this would help. 
%% }
%% %\end{rem}

%% Okay this seems like a very difficult thing to prove. I would suggest that we do not try to figure the critical case, but instead add a sentence like :
The critical case $r=r_\mathcal{P}$ remains open.
In the context of free products, the branching random walk needs to visit $gP$ infinitely many times in order to accumulate at $\xi$.
As a consequence, the authors of \cite{CandelleroGilchMuller} prove in Propositions~4.3,~4.4 and~4.5 that for $r=r_\mathcal{P}$, almost surely, $\Lambda$ does not contain any parabolic limit point.
However, for arbitrary relatively hyperbolic groups, the branching random walk can visit infinitely many partial cones $\Omega(x,C)$, $x\in gP$, without entering $gP$ at all.
Thus, we might need new material to figure this critical case.

\subsection{Convergent Poincar\'e series}\label{SectionconvergentPoincare}
Let $\Gamma$ be a relatively hyperbolic group and let $\mu$ be a finitely supported symmetric admissible probability measure on $\Gamma$.
We consider the Poincaré series $\Theta_\Gamma(r,s)$ defined above and for $P$ a maximal parabolic subgroup, the Poincar\'e series $\Theta_P(r,s)$ defined by
$$\Theta_P(r,s)=\sum_{x\in P}G(e,x|r)\mathrm{e}^{-sd(e,x)}.$$

In \cite[Example~C]{DWY}, we proved that if a finitely generated group $\Gamma_0$ can be endowed with a symmetric finitely supported admissible probability measure $\mu_0$ such that $\Theta_{\Gamma_0}(r_0,\omega_{\Gamma_0}(r_0))$ is convergent for some $r_0<R_0$, then the free product
$\Gamma=\Gamma_0*\mathbb Z^d$ can also be endowed with a symmetric finitely supported admissible probability measure $\mu$ such that $\Theta_{\Gamma}(r,\omega_{\Gamma}(r))$ is convergent for some $r\leq R$ depending on $r_0$.
Here, $R_0$ denotes the inverse of the spectral radius of $\mu_0$ and $R$ the inverse of the spectral radius of $\mu$.

Note that in this situation, $\Gamma$ is relatively hyperbolic with respect to the conjugates of $\Gamma_0$ and $\mathbb Z^d$.
When considered as a maximal parabolic subgroup of $\Gamma$, we will write $P$ for $\Gamma_0$ in the sequel, for sake of consistency with the previous sections.

The question of whether such a couple $(\Gamma_0,\mu_0)$ exists was left unanswered, but we announced that it was possible to construct one.
We provide the details of this construction here and prove a more precise result.

\medskip
Let $\mathbb F_n$ be the free group with $n$ generators and let $\Gamma_0=\mathbb F_n\times \mathbb F_m$, $n\neq m$.
The Cayley graph of $\Gamma_0$ is a Cartesian product of two regular trees $T_{l_1}$, $T_{l_2}$  {with respective degrees $l_1 = 2n$ and $l_2 = 2m$}.
We consider the measure $\mu_0$ on $\Gamma_0$ defined by~(\ref{defmubitree}).
We deduce the following from Theorem~\ref{thmbitree}.

\begin{prop}\label{transitionphaseparabolicsubgroup}
There exists $1<r_0 <R_0$ such that for $r\leq r_0$, the Poincar\'e series $\Theta_{\Gamma_0}(r,\omega_{\Gamma_0}(r))$ is divergent and for $r_0<r<R_0$, it is convergent.
\end{prop}

We now recall how the measure $\mu$ is constructed on the free product $\Gamma=\Gamma_0*\Gamma_1$ where $\Gamma_1=\mathbb Z^d$.
Let $\mu_1$ be any finitely supported symmetric admissible probability measure on $\mathbb Z^d$.
Following \cite{DWY}, we assume that $d\geq3$ for convenience, so that the random walk associated with $\mu_1$ is transient at the spectral radius, i.e.\ the Green function $G_{\mu_1}(e,e|R_1)$ is finite, where $R_1$ is the inverse of the spectral radius of $\mu_1$.

For $0\leq \alpha\leq 1$, we set
$$\mu_\alpha=\alpha \mu_1+(1-\alpha)\mu_0.$$
For simplicity, we write $\mu=\mu_\alpha$ below.
We write $G$ for the Green function associated with $\mu$ and $G_{i}$ for the Green function associated with $\mu_i$.
By \cite[(10),(11)]{DWY}, there exist two numbers $w_{0,\alpha,r}$ and $w_{1,\alpha,r}$ and continuous non-decreasing functions $\zeta_{i,\alpha}$ of $r\leq R$ such that
\begin{equation}\label{equationG0}
    G(e,x|r)=\frac{1}{1-w_{0,\alpha,r}}G_0(e,x|\zeta_{0,\alpha}(r)),\ x\in \Gamma_0
\end{equation}
and
\begin{equation}\label{equationG1}
    G(e,x|r)=\frac{1}{1-w_{1,\alpha,r}}G_1(e,x|\zeta_{1,\alpha}(r)),\ x\in \Gamma_1=\mathbb Z^d.
\end{equation}
Furthermore, we have $\zeta_{\alpha, 0}(r) = \frac{(1 - \alpha) r}{1 - w_{0,\alpha,r}}$, and 
\begin{equation}\label{wralpha}
w_{0, \alpha,r} = \sum_{n \geq 1} \mathbf{P} \left( X_n = e, X_k \neq e, 1 \leq k < n, \text{ first step chosen using } \alpha \mu_1 \right) r^n. 
\end{equation}
Similar expressions hold for $w_{1,\alpha,r}$ and $\zeta_{\alpha,1}(r)$.

\begin{lem}
For fixed $\alpha$, the functions $w_{i,\alpha,r}$ and $\zeta_{\alpha,i}(r)$ are (strictly) increasing in $r$.
\end{lem}

\begin{proof}
The functions $w_{i,\alpha,r}$ are power series in $r$ with positive coefficients, so they are increasing.
It follows from the above expression of $\zeta_{\alpha,i}$ that these functions are increasing too.
\end{proof}

We prove here the following.
\begin{thm}\label{transitionphasePoincare}
If $\alpha$ is small enough, then there exist $r_*(\alpha)<r_\sharp(\alpha) <R$ so that the following holds.
For $r\leq r_*(\alpha)$, the Poincaré series $\Theta_{\Gamma}(r,\omega_{\Gamma}(r))$ is divergent and $\Gamma$ has a parabolic gap for the Green function.
On the other hand, at $r=r_\sharp(\alpha)$, it is convergent and $\omega_P(r)=\omega_\Gamma(r)$, where $P=\Gamma_0$.
\end{thm}

In the proof of Theorem~\ref{transitionphasePoincare}, we shall also use the following result which is an enhanced version of \cite[(14)]{DWY}.

\begin{lem}\label{lemmaDWY(14)}
For every $\epsilon>0$, there exists $\alpha_0$ such that for $\alpha\leq \alpha_0$, the following holds.
For every $r\in [0,R]$ and for every $x\in \Gamma_1\setminus\{e\}$, we have
$$G(e,x|r)\leq \epsilon.$$
\end{lem}

\begin{proof}
By \cite[Lemma~3.15]{DWY}, if $\alpha\leq \alpha_0$, then $w_{1,\alpha,r}$ stays bounded away from 1 and $\zeta_{1,\alpha}(r)$ converges to 0 as $\alpha$ tends to 0 uniformly over $r\leq R$.
We conclude from~(\ref{equationG1}) that for small enough $\alpha$, independently of $r$, for every $x\neq e\in \Gamma_1$ we have
\begin{equation*}\label{thirdboundpoincare}
    G(e,x|r)\leq \frac{G_1(e,x|\zeta_{1,\alpha}(r))}{1-w_{1,\alpha,r}}\leq \epsilon. \qedhere
\end{equation*}
\end{proof}

%Theorem~\ref{mainthm} is a direct conclusion of this theorem.
We are ready to complete the proof of Theorem~\ref{transitionphasePoincare}.
Let us first explain briefly how the quantities $\alpha$, $r_*(\alpha)$ and $r_\sharp(\alpha)$ are chosen.

By \cite[Lemma~3.14]{DWY}, as $\alpha$ converges to 0, $\zeta_{0,\alpha}(R)$ converges to $R_0$. Let us now fix any $r_1>r_0$, where $r_0$ is given in Proposition \ref{transitionphaseparabolicsubgroup}. Thus, there exists $\alpha_1$ so that $\zeta_{\alpha, 0}(R)>r_1$ holds for any   $\alpha\le \alpha_1$.
Since $\zeta_{0,\alpha}(r)$ is increasing in $r$, there exist $r_*<r_\sharp <R$ depending on $\alpha$ such that $\zeta_{\alpha,0}(r_*)=r_0$ and $\zeta_{0,\alpha}(r_\sharp)=r_1$.

We will also need to choose some $\epsilon>0$ only depending on $\mu_0$, $\mu_1$ and $r_1$ such that Equation~(\ref{equationepsilonpoincare}) holds below.
Then, we choose $\alpha$ small enough such that the conclusions of Lemma~\ref{lemmaDWY(14)} holds for such $\epsilon$ and such that there exist $r_*<r_\sharp <R$ with $\zeta_{\alpha,0}(r_*)=r_0$ and $\zeta_{0,\alpha}(r_\sharp)=r_1$.
\begin{proof}[Proof of Theorem~\ref{transitionphasePoincare}]
The proof follows the lines of \cite[Example~C]{DWY}.
We will write as announced above $P$ for $\Gamma_0$ when it is considered as a maximal parabolic subgroup of $\Gamma$.
In particular, we write
$$\omega_{\Gamma_0}(s)=\limsup_{n\to\infty} \frac{1}{n}\log \sum_{x\in \Gamma_0,|x|=n}G_0(e,x|s)$$
for the growth rate of the Green function $G_0$ associated with $\mu_0$ and
$$\omega_P(r)=\limsup_{n\to\infty} \frac{1}{n}\log \sum_{x\in  P,|x|=n}G(e,x|r)$$
for the growth rate of the Green function $G$ associated with $\mu$, induced on $P$.
By~(\ref{equationG0}), we see that
\begin{equation}\label{equationgrowthrates}
\omega_P(r)=\omega_{\Gamma_0}(\zeta_{0,\alpha}(r)).
\end{equation}
We will also write $\Theta_{\Gamma_0}(r,s)$ for the Poincar\'e series associated with $\mu_0$ on $\Gamma_0$ and $\Theta_P(r,s)$ for the Poincar\'e series associated with $\mu$ induced on $P$, defined as above by
$$\Theta_P(r,s)=\sum_{x\in P}G(e,x|r)\mathrm{e}^{-s|x|}.$$

If $\alpha$ is small enough, there exist $r_*<r_\sharp <R$ with $\zeta_{\alpha,0}(r_*)=r_0$ and $\zeta_{0,\alpha}(r_\sharp)=r_1$.
First, if $r\leq r_*$, then $\zeta_{0,\alpha}(r)\leq r_0$.
By Proposition~\ref{transitionphaseparabolicsubgroup}, $\Theta_{\Gamma_0}(\zeta_{0,\alpha}(r),\omega_{\Gamma_0}(\zeta_{0,\alpha}(r)))$ diverges,
hence $\Theta_P(r,\omega_{P}(r))$ also diverges.
Consequently, according to \cite[Corollary~3.9]{DWY}, $\omega_P(r)<\omega_\Gamma(r)$.
Moreover, since $\Gamma_1$ is amenable, we also deduce from Proposition~\ref{propamenableparabolic} that the induced growth rate on $\Gamma_1$ considered as a maximal parabolic subgroup is smaller than $\omega_\Gamma(r)$.
In other words, $\Gamma$ has a parabolic gap for the Green function on $(0,r_*]$.
Therefore, \cite[Lemma~3.7]{DWY} yields that the Poincar\'e series $\Theta_{\Gamma}(r,\omega_\Gamma(r))$ is divergent for $r\leq r_*$.

Next, fix $r=r_\sharp$ so that $\zeta_{0,\alpha}(r)=r_1>r_0$.
According to \cite[(15)]{DWY},
$$\Theta_{\Gamma}(r,s)\leq G(e,e|r)\sum_{k\geq 0}\left (\sum_{x\in \Gamma_0\setminus\{e\}}\frac{G(e,x|r)}{G(e,e|r)}\mathrm{e}^{-s|x|}\right)^k\left (\sum_{y\in \Gamma_1\setminus\{e\}}\frac{G(e,y|r)}{G(e,e|r)}\mathrm{e}^{-s|y|}\right)^k$$
By~(\ref{equationgrowthrates}), $\omega_P(r)=\omega_{\Gamma_0}(\zeta_{0,\alpha}(r))=\omega_{\Gamma_0}(r_1)$ and so
$\omega_{\Gamma}(r)\geq \omega_{\Gamma_0}(\zeta_{0,\alpha}(r))$.
Moreover, by Theorem~\ref{thmbitree}, $r_0>1$, so that $r_1>1$ and so we deduce from Lemma~\ref{lowerboundgrowthrateat1} and Proposition~\ref{omegaGcont} that $\omega_{\Gamma_0}(r_1)\geq c>0$.
Therefore,
\begin{equation}\label{equationgrowthratepositive}
    \omega_{\Gamma}(r)\geq c>0.
\end{equation}
Using~(\ref{equationG0}) and Proposition \ref{transitionphaseparabolicsubgroup}, we see that
\begin{equation}\label{firstboundpoincare}
    \sum_{x\in \Gamma_0\setminus\{e\}}\frac{G(e,x|r)}{G(e,e|r)}\mathrm{e}^{-\omega_\Gamma(r)|x|}\leq \frac{1}{G_0(e,e|r_1)}\Theta_{\Gamma_0}(r_1,\omega_{\Gamma_0}(r_1))<\infty.
\end{equation}
Also, $\mathbb Z^d$ has polynomial growth and so by~(\ref{equationgrowthratepositive})
\begin{equation}\label{secondboundpoincare}
    \sum_{y\in \Gamma_1\setminus\{e\}}\frac{1-w_{1,\alpha,r}}{G(e,e|r)}\mathrm{e}^{-\omega_\Gamma(r)|y|}\leq \frac{1}{G_1(e,e|\zeta_{1,\alpha}(r_\sharp))}\sum_{y\in \Gamma_1\setminus\{e\}}\mathrm{e}^{-c|y|}\leq C_1.
\end{equation}
We then choose $\epsilon>0$ such that
\begin{equation}\label{equationepsilonpoincare}
    C_2:=\epsilon\frac{1}{G_0(e,e|r_1)}\Theta_{\Gamma_0}(r_1,\omega_{\Gamma_0}(r_1))C_1<1.
\end{equation}
Note that $\epsilon$ does not depend on $\mu$ but only on $C_1$, $\mu_0$ and $r_1$.%\ywy{I'm confused here:  the choice of $r_*$ depends on $\alpha$, and then $r_\sharp$ and $r_1$ does on $\alpha$. So $\epsilon$ depends on $\alpha$, right?}
In particular, it does not depend on $\alpha$ and so by Lemma~\ref{lemmaDWY(14)}, we can choose $\alpha$ small enough so that for every $y\in \Gamma_1\setminus \{e\}$,
$$G(e,y|r)\leq \epsilon.$$

Combining~(\ref{firstboundpoincare}),~(\ref{secondboundpoincare}),~(\ref{equationepsilonpoincare}) and~(\ref{thirdboundpoincare}), we get
\begin{align*}
    \Theta_{\Gamma}(r,\omega_\Gamma(r))&\leq G(e,e|r)\sum_{k\geq 0}\left(\epsilon\frac{1}{G_0(e,e|\zeta_{0,\alpha}(r))}\Theta_{\Gamma_0}(\zeta_{0,\alpha},\omega_{\Gamma_0}(\zeta_{0,\alpha}(r)))C_1\right)^k\\
    &\leq G(e,e|r)\sum_{k\geq 0}C_2^k,
\end{align*}
so that
$\Theta_{\Gamma}(r,\omega_\Gamma(r))$ is finite.

Finally, we deduce from  \cite[Lemma~3.7]{DWY} that at $r=r_\sharp$, $\Gamma$ does not have a parabolic gap for the Green function.
Since $\Gamma_1$ is amenable, we necessarily have $\omega_\Gamma(r)=\omega_P(r)$ by Proposition~\ref{propamenableparabolic}.
This concludes the proof.
\end{proof}

\begin{rem}\label{remarkchoiceofr1}
In the proof of Theorem~\ref{transitionphasePoincare}, note that we can choose $r_1$ arbitrarily close to $r_0$.
Unfortunately, as $r_1$ goes to $r_0$, $\Theta_{\Gamma_0}(r_1,\omega_{\Gamma_0}(r_1))$ tends to infinity, so $\epsilon$ satisfying~(\ref{equationepsilonpoincare}) converges to 0.
Consequently, the parameter $\alpha$ also tends to 0.
In other words, as $r_1$ tends to $r_0$, we need to choose a measure $\mu$ that tends to the measure $\mu_0$ distributed on $\Gamma_0$.
Now, since the functions $\zeta_{\alpha,i}$ of $r$ depend on $\alpha$, we cannot guarantee that $r_{\sharp}$ tends to $r_*$.
In particular, we cannot prove that there is a true phase transition for the convergence of the Poincar\'e series $\Theta_\Gamma(r,\omega_\Gamma(r))$ at $r_*$.
\end{rem}

\begin{cor}\label{corosubmultiplicativityfails}
If $\alpha$ is small enough, then the following holds.
For $r=r_\sharp(\alpha)$, $\Gamma$ does not have a parabolic gap for the Green function and~(\ref{roughmultiplicativity}) and~(\ref{purelyexponentialgrowth}) do not hold.
\end{cor}

 \begin{proof}
 By Theorem~\ref{transitionphasePoincare}, if $r=r_\sharp(\alpha)$, then the Poincar\'e series $\Theta_\Gamma(r,\omega_\Gamma(r))$ is convergent and $\Gamma$ does not have a parabolic gap.
 Assume by contradiction that there exists $C$ such that
 $$H_r(m+n)\leq C H_r(n)H_r(m).$$
 Then, the quantity $CH_r(n)$ is sub-multiplicative, hence by Fekete's lemma,
 $$\omega_\Gamma(r)=\lim_{n\to\infty} \frac{1}{n}\log CH_r(n)=\inf_{n\ge 1} \frac{1}{n}\log CH_r(n).$$
 Thus, for every $n$, we have
 $$CH_r(n)\geq \mathrm{e}^{n\omega_\Gamma(r)}.$$
 This implies that $\Theta_\Gamma(r,\omega_\Gamma(r))$ diverges, which is a contradiction.
 In particular, we see that~(\ref{roughmultiplicativity}) fails.
 Now,~(\ref{roughmultiplicativity}) is a direct consequence of~(\ref{purelyexponentialgrowth}), with $C=(C')^3$, hence~(\ref{purelyexponentialgrowth}) also fails.
 \end{proof}

We also deduce the following from Theorem~\ref{divergencetypeandatoms} and Theorem~\ref{transitionphasePoincare}.
\begin{cor}\label{transitionphasePattersonSullivan}
If $\alpha$ is small enough, then the following holds.
For $r\leq r_*(\alpha)$, the measure $\nu_e$ on the Bowditch boundary has no atom and is supported on the set of conical limit points.
For $r=r_\sharp(\alpha)$, it is purely atomic and is supported on the set of parabolic limit points.
\end{cor}

\begin{rem}\label{remarkCGM}
In \cite{CandelleroGilchMuller}, the authors prove that in the context of free products, we always have
$\omega_P(r)<\omega_\Gamma(r)$ for every maximal parabolic subgroup $P$, i.e.\ $\Gamma$ has a parabolic gap for the Green function.
However, their proof relies on an unproved statement, namely that the quantity $H_r(n)$ is sub-multiplicative for every finitely generated group and then apply this property to the maximal parabolic subgroup $P$, see precisely the proof of \cite[Lemma~4.7]{CandelleroGilchMuller} and also \cite[Remark~3.17]{DWY}.
However, by Theorem~\ref{thmbitree}, we see that sub-multiplicativity fails for the Cartesian product of two regular trees if $r\geq r_0$.
Moreover, by Corollary~\ref{corosubmultiplicativityfails}, in the above example, if $r=r_\sharp(\alpha)$, then $\Gamma$ does not have a parabolic gap for the Green function.
\end{rem}

%{\color{blue}I would suggest not to mention anything related to choosing uniformly the constant $C$ and $C'$ in the sub-multiplicativity and the purely exponential growth of the Green function.
%Neither concerning the uniformity when parabolic subgroups are amenable, nor concerning the non-uniformity when a parabolic subgroup is a bi-tree.
%I think this is an interesting question, but I cannot see any direct consequence.
%}

\subsection{The growth tightness property}

Let $d$ be a proper left invariant distance on $\Gamma$. The \textit{growth rate} of $\Gamma$ for $d$ is defined as follows:
$$
\delta(\Gamma,d) :=\limsup_{n\to\infty}\frac{\log \sharp\{x\in \Gamma: d(e,x)\le n\}}{n}
$$
A \textit{nontrivial quotient} $\bar \Gamma$ of $\Gamma$ means that the kernel of the canonical projection $\Gamma\to \bar \Gamma$ is an infinite normal subgroup of $\Gamma$.
We say that $\Gamma$ is  \textit{growth tight} for the distance $d$ if for every nontrivial  quotient $\bar \Gamma$ of $\Gamma$, endowed with the quotient distance $\bar d$ from $d$, we have $\delta(\bar \Gamma,\bar d)<\delta(\Gamma,d)$.

Let us assume that $\Gamma$ is a relatively hyperbolic group.
Whenever a maximal parabolic group $P$ has growth rate $\delta(P,d)$ strictly less than $\delta(\Gamma, d)$, we say that $\Gamma$ has a  parabolic gap along $P$ fir the distance $d$. When $\Gamma$  has parabolic gap along every maximal parabolic subgroup, we say that $\Gamma$ has  the parabolic gap property.
%It is proved in that 

%{\color{blue}Could you provide a reference for that ? I was thinking that we can say a few words about the proof. I can do it if you want.}\ywy{Yes, we need a few words, as this metric is not geodesic. I'm adapting some argument from \cite{YangSCC}. I found I can only do it for proper metrics such as $\mathfrak d_r(x,y)$, which is  additive along the transition points on word geodesics. If you can do more, please let me know.}

Recall that by Lemma~\ref{quasiisometricdistances}, $\mathfrak d_r(x,y)= {\omega_{\Gamma}(r)} |x^{-1}y|+|x^{-1}y|_r$ is quasi-isometric to the word distance for $r<R$. By Lemma \ref{NewPoincareSeries}, we have $\delta(\Gamma,\mathfrak d_r)=1$.  
The following result relates the   gap property for Green functions to the gap property for the distance $\delta(\Gamma,\mathfrak d_r)$.
\begin{prop}\label{PropTwoPGs}
Let $\Gamma$ be a group endowed with a probability measure $\mu$ such that the $\mu$-random walk is transient at the spectral radius, i.e.\ $G(e,e|R)$ is finite.
Let   $A\subset \Gamma$ be any subset. 
If  $\omega_A(r)<\omega_\Gamma(r)$ for some $1<r\le R $, then $\delta(A,\mathfrak d_r)<1$.    
\end{prop}

We need the following lemma, which generalizes \cite[Lemma~3.1]{Tanaka17} in hyperbolic groups with a similar proof.
\begin{lem}\label{ContinuitythetaLem}
Under the assumption of Proposition \ref{PropTwoPGs}. 
For any $\theta\in\mathbb R$, $r\leq R$ and $A\subset \Gamma$, define   
$$
\omega_{A, r}(\theta) := \limsup_{n \to \infty} \frac{1}{n} \log \sum_{x \in A, |x| = n} \left[ G(e, x | r) \right]^\theta. 
$$
Then   $\omega_{A, r}( \theta)$ is a convex function on $\mathbb R$. If $\Gamma$ is a relatively hyperbolic group, then  $\omega_{\Gamma, r}( \theta)$ is a true limit. 
\end{lem}
\begin{proof}
Denote $H^\theta_r(n):=\sum_{x \in A, |x| = n} \left[ G(e, x | r) \right]^\theta$. For $\theta_0$, $\theta_1 \in \mathbb{R}$ and $0 < t < 1$, by the H\"older inequality,
  \[
    H_r^{t \theta_0 + (1 - t) \theta_1}(n) 
    \leq (H_r^{\theta_0}(n))^t  (H_r^{\theta_1}(n))^{1 - t}. 
  \]
  Thus $\omega_{A, r}$ is convex: 
  \[
    \omega_{A, r}(t \theta_0 + (1 - t) \theta_1) \leq t \omega_{A, r}(\theta_0) + (1 - t) \omega_{A, r}(\theta_1). 
  \]
If $\Gamma$ is a relatively hyperbolic group, the same proof as in \cite[Lemma 3.2]{DWY} shows (there with $\theta=1$) that for $A=\Gamma$, the sequence $H_r^\theta(n)$ is sub-multiplicative, that is, $H_r^\theta(n+m)\le C H_r^\theta(n) H_r^\theta(m)$ for some $C>0$ . Thus, the limit exists by Feketa's lemma. 
\end{proof}
As a convex function on $\mathbb R$, $\omega_{\Gamma, r}( \theta)$ is a continuous function of $\theta\in \mathbb R$, and is differentiable, except maybe at countably infinitely many points.

\begin{proof}[Proof of Proposition \ref{PropTwoPGs}]
To show $\delta(A,\mathfrak d_r)<1$, it suffices to find some $\epsilon > 0$ such that
\[
    \sum_{x \in A} \mathrm{e}^{- (1 - \epsilon) \mathfrak d_r(e, x)} = \sum_{x \in A} \mathrm{e}^{- (1 - \epsilon) \omega_{\Gamma}(r) |x|} \left[ G(e, x | r) \right]^{1 - \epsilon} < \infty.
\]

By  Lemma~\ref{ContinuitythetaLem}, the function $\omega_{A,r}(\theta)$ is continuous in $\theta\in\mathbb R$. 
If $\omega_A(r) < \omega_{\Gamma}(r)$, we can choose $\epsilon,\eta>0$  small enough so that
$$\omega_A(r) + \eta < (1 - \epsilon) \omega_{\Gamma}(r),$$
and at the same time, by continuity of $\omega_{A,r}(\theta)$ at $\theta=1$, the following holds : for large enough $n$,
$$\sum_{x \in A, |x| = n} \left[ G(e, x | r) \right]^{1 - \varepsilon} \leq  \mathrm{e}^{n \left( \omega_{A,r}(1) + \eta \right)}.$$ 
By definition, $\omega_{A,r}(1)=\omega_A(r)$, so the two inequalities above  yield
  \[
    \sum_{x \in A} \mathrm{e}^{- (1 - \epsilon) \omega_{\Gamma}(r) |x|} G(e, x | r)^{1-\epsilon} < \infty,
  \]    
which is the desired inequality.
\end{proof}

Given $f\in \Gamma$ and $\epsilon>0$, let $\mathcal V_{\epsilon, f}$ be the set of \textit{barrier-free} elements $x\in \Gamma$, that is, elements for which the $\epsilon$-neighborhood of some geodesic $[e,x]$ contains a geodesic segment representing $f$. The following is analogous to \cite[Theorem C]{YangSCC}.
\begin{lem}\label{BarrierFreeLem}
Let $\Gamma$ be a relatively hyperbolic group  with parabolic gap property for Green function.
Then there exists some $\epsilon>0$ such that  the   set $\mathcal V:=\mathcal V_{\epsilon, f}$  has growth rate strictly less than 1 for any  element $f\in \Gamma$: $\delta(\mathcal V, \mathfrak d_r)<1$.    
\end{lem}
\begin{proof}
By  Proposition \ref{PropTwoPGs}, it suffices to prove $\omega_{\mathcal V}(r)<\omega_{\Gamma}(r)$.
Set
$$a^\omega(n) = \mathrm{e}^{-\omega n}\sum_{x\in  \mathcal V, |x|=n} G(e,x;r).$$
Assume that $\omega_{\mathcal V}(r) >\omega_P(r)$ for every maximal parabolic subgroup $P$; otherwise the parabolic gap concludes the proof. If   $\omega_{\mathcal V}(r)>\omega >\omega_P(r)$, one obtains
$$
a^\omega(n+m) \le c_0 \sum_{1\le i\le n} a^\omega(i) \sum_{1\le j\le n} a^\omega(j) 
$$
by the same argument of \cite[Lemma~3.7]{DWY} where $a^\omega(n)$ is summed up over $\Gamma$ instead of $\mathcal V$. This implies via a variant of Feketa's lemma in \cite[Lemma 4.3]{DPPS} that the series $\sum_{x\in \mathcal V} \mathrm{e}^{-s |x|}G(e,x;r)$ diverges at $s=\omega_{\mathcal V}(r)$. 

Fix any $L>0$. We choose an \textit{$L$-separated net} $A\subset \mathcal V$ in word distance:  if for any $x, y\in A$ we have $|x^{-1}y|>L$ and for any $y\in \mathcal V$, there exists $x\in A$ such that $|x^{-1}y|\le L$. Note that if $|x^{-1}y|\le L$, then $G(e, x;r)\asymp_L G(e,y;r)$. Since any ball of radius $L$ in word distance contains a fixed number of elements, we deduce that   $\Theta_A(r, s)\asymp_L \Theta_A(r, s)$ whenever they are finite.   Thus, $\omega_A(r)=\omega_{\mathcal V}(r)$.

Following \cite[Section~4.2]{YangSCC}, we use a ping-pong argument to construct a \textit{free product of sets} inside $\Gamma$: if $L$ is large enough, there exist a finite set of elements $B\subset \Gamma$     such that  the set $\mathcal W(A, B)$ of alternating words over $A$ and $B$ embeds into $\Gamma$ as a free semi-group under the evaluation map. This construction uses only the word distance. Now,  by \cite[Lemma~3.8]{DWY}, we have $\omega_A(r)<\omega_\Gamma(r)$ and then $\delta(\mathcal V, \mathfrak d_r)<1$ by Proposition~\ref{PropTwoPGs}.
\end{proof}

\begin{prop}\label{growthrightnessandparabolicgap}
If a relatively hyperbolic group $\Gamma$ has parabolic gap for the Green function, then it is growth tight for the distance $\mathfrak d_r$.   Otherwise, there exists a nontrivial quotient $\bar \Gamma$ such that $\delta(\bar \Gamma, \mathfrak {\bar {d_r}})=\delta(\Gamma, \mathfrak d_r)$.
\end{prop}
\begin{proof}
(1). We  follow the proof of \cite[Corollary~4.6]{YangSCC} in our setup. Let $N$ be the infinite kernel of $\Gamma \to \bar \Gamma$. We form a set $A$ by choosing a shortest representative $h\in hN$ for each $hN\in \bar G$ so that $\mathfrak d_r(e, h)=\mathfrak d_r(e, hN)$. By definition of the quotient distance, the growth rate of the set $A$ for $\mathfrak d_r$ is exactly the growth rate of $\bar \Gamma$ for $\mathfrak {\bar d_r}$.  

We now choose a sufficiently long loxodromic element $f\in N$, which exists since $N$ is infinite.  If $|f|/ \epsilon$ is large enough, we see that  any geodesic $[e, h]$ cannot contain $f$ in its $\epsilon$-neighborhood. Indeed, if not, the loxodromic element $f$ produces two transition points on some $[e,h]$ with a distance comparable with $|f|$.  Now, we   use the following fact given by Lemma~\ref{weakAncona}: if $u, v$ are two transition points in this order on a word geodesic $[x,y]$, then 
\begin{align}\label{almostEqualityEQ}
\mathfrak d_r(x,y) \simeq \mathfrak d_r(x,u)+\mathfrak d_r(u,v)+\mathfrak d_r(v,y)    
\end{align} where $\simeq$ denote the equality up to a uniform additive constant.  We could then shorten $\mathfrak d_r(e,h)$ by an amount $\mathfrak d_r(e,f)=\omega_\Gamma(r)|f|+|f|_r$, giving a contradiction with the above choice of  $h\in hN$ as the shortest one.  

In other words, we proved that $A\subset \mathcal V_{\epsilon, f}$. Hence, $\delta(A, \mathfrak d_r)\le \delta(\mathcal V_{\epsilon, f},\mathfrak d_r)<1$ by Lemma~\ref{BarrierFreeLem}. The growth tightness follows.  

(2). Assume that $\omega_P(r)=\omega_\Gamma(r)$ for a maximal parabolic subgroup $P$. Then,
\[
\sum_{x \in P} 
 \mathrm{e}^{- \mathfrak d_r(e, x)}=\sum_{x \in P} \mathrm{e}^{-\omega_{P}(r)} G(e,x|r)
\]
hence, we see that the growth rate for $\mathfrak d_r$ induced on $P$ equals $1$. 

Fix a loxodromic element $f\in G$. For any large enough $n$, the quotient group $\bar \Gamma$ defined  as $G/\langle\langle f^n\rangle\rangle$  is again a relatively hyperbolic group, and $P\cap \langle\langle f^n\rangle\rangle$ is trivial (see \cite[Lemma~8.9]{YangPS}). Thus, the set of elements in $P$ embeds into $\bar \Gamma$ whose image we denote by $\bar P$, so $\delta(\bar P, \mathfrak {\bar d_r})\ge \delta(P, \mathfrak d_r)=1$. Therefore, $\delta(\bar \Gamma, \mathfrak {\bar d_r})=1$.   
\end{proof}

%To achieve the growth tightness, we might do further free product to make parabolic subgroups to be relatively hyperbolic. I will think more carefully.
Relatively hyperbolic groups endowed with a word distance are always growth tight by \cite{YangTight, ACT}. In fact,  any       co-compact action of  a relatively hyperbolic group on a proper geodesic  space contains a contracting element and thus is growth tight. Here, the existence of a contracting element in the co-compact action   follows    from the fact that in a relatively hyperbolic group, a loxodromic element is  contracting with respect to all word quasi-geodesics: any $c$-quasi-geodesic outside the $C$-neighborhood of the axis has $C$-bounded projection for  some $C=C(c)$.  See \cite[Proposition 8.5]{GP16}.  

On the contrary, as a corollary of Theorem~\ref{transitionphasePoincare} and Proposition~\ref{growthrightnessandparabolicgap},  growth tightness for $\mathfrak d_r$ may fail and depends on $r$.

\begin{thm}\label{transitionphasegrowthtightness}
There exists a relatively hyperbolic group $\Gamma$ endowed with a finitely supported symmetric and admissible probability measure $\mu$ such that the following holds.
There exist $1<r_*<r_\sharp<R$ such that $\Gamma$ endowed with the distance $\mathfrak d_r$ is growth tight for  $r\le r_*$, but is not for $r=r_\sharp$.   
\end{thm}

Note that the proper distance $\mathfrak d_r$ is quasi-isometric to any word distance for $r<R$ by Lemma~\ref{quasiisometricdistances}. We say that a metric space $(X,d)$ is \textit{$D$-coarsely geodesic} for some $D>0$ if for any two points $x,y\in X$, there exists a $(1,D)$-quasi-isometric embedding $\phi: [0, l]\to X$ for $l:=d(x,y)$ so that $\phi(0)=x,\phi(l)=y$, and $$|d(\phi(m), \phi(n))-|m-n||\le D$$ for any $0\le m\le n\le l$.   It is an open question   whether the Green distance is a geodesic distance on hyperbolic groups, see \cite[Section~1.7]{BHM11}. We shall however derive the following corollary from Theorem \ref{transitionphasegrowthtightness}.%, that is, the answer is negative for certain relatively hyperbolic groups.

\begin{cor}\label{drnotgeodesic}
For $r=r_\sharp$,   $(\Gamma,\mathfrak d_r)$  is not  a coarsely geodesic metric space.    
\end{cor} 

The proof requires the following observation of independent interest.
Recall that an element of infinite order $g$ in a finitely generated group $\Gamma$ is called contracting for a distance $d$ on $\Gamma$ if any $d$-metric ball in $\Gamma$ disjoint with the subgroup $\langle g\rangle$ has $C$-bounded projection to $\langle g\rangle$ for some universal constant $C>0$.
In a $D$-coarsely geodesic metric space, this is equivalent to the bounded image property: there exists $C=C(D)>0$ such that any $D$-coarse geodesic outside the $C$-neighborhood of $\langle g\rangle$ has shortest projection of diameter at most $C$ to it. For simplicity, we can  take the same $C$ for both statements. 
\begin{lem}\label{contractingdrmetric}
Any loxodromic element in a relatively hyperbolic group is   contracting with respect to $\mathfrak d_r$ where $1\le r\le R$.     
\end{lem}
\begin{proof}
By (\ref{almostEqualityEQ}),  the proper distance $\mathfrak d_r$ is coarsely additive along the set of transition points on the word geodesic. That is, if $z$ is a transition point on $[x,y]$ we have $\mathfrak d_r(x,y)\ge \mathfrak d_r(x,z)+\mathfrak d_r(z,y)-D$ for some universal $D>0$.
Let $\gamma$ be a quasi-geodesic preserved by a loxodromic element $g$. Then $\gamma$ is $C_0$-  contracting with respect to word distance for some $C_0>0$.
We claim that the shortest projection $z$ of any point $x$ to $\gamma$ for the distance $\mathfrak d_r$ is $D_0$-close to the shortest projection $w$ of $x$ to $\gamma$ for the word distance.
Indeed, as $w$ is uniformly close to a transition point on $[x,z]$, we see that $\mathfrak d_r(x,z)+D_0\ge \mathfrak d_r(x,w)+\mathfrak d_r(w,z)$  for some $D_0=D_0(C_0)>0$. By the definition of $\mathfrak d_r$-shortest projection, we have   $\mathfrak d_r(x,z)\le \mathfrak d_r(x,w)$ and thus the claim follows.

We now prove that $\gamma$ is   contracting for $\mathfrak d_r$. Pick any  $\mathfrak d_r$-distance ball  $B$ centered at $x$ disjoint with $\gamma$. Let $y\in B$ so that the projections denoted by $u, v$ respectively  of $x,y$ to $\gamma$ realizes the $\mathfrak d_r$-diameter of that of $B$ to $\gamma$. By Lemma \ref{quasiisometricdistances}, $\mathfrak d_r$ is quasi-isometric to word metric for $1\le r<R$. For $r=R$, by \cite[Lemma 2.1]{GouezelLalley}, $f(n):=\max\{G(e,x|R): x\in S_n\}\to 0$ as $n\to \infty$, so $f(n)$ is a proper function. We  can thus choose  $\mathfrak d_r(u, v)$ large enough so that $|u^{-1}y|\ge C_0$, hence the   contracting property of $\gamma$ in word distance implies that $u,v$ are uniformly close to  transition points on $[x,y]$. By the additive property of $\mathfrak d_r(x,y)$ along transition points, we obtain $\mathfrak d_r(x,y)\ge \mathfrak d_r(x,u)+\mathfrak d_r(u,v)+\mathfrak d_r(v,y)-D_1$ for some $D_1>0$. As $B$ is a $\mathfrak d_r$-distance ball disjoint with $\gamma$, we have $\mathfrak d_r(x,y)\le \mathfrak d_r(x,u)$. We then obtain a contradiction if $\mathfrak d_r(u,v)>D_1$.
Thus, the   contracting property for $\mathfrak d_r$ follows.
\end{proof}%\ywy{I think we can prove that it is not geodesic metric by using it is not growth tight. Do you know  any application? }
%{\color{blue}Why is that so ? Maybe I didn't understand, but I thought that the growth tightness for co-compact actions on geodesic metric space was also an open question. }\ywy{Here, using Ancona inequality, we could show that loxodromic elements are also strongly contracting for Green metric, as word metric agrees with Green metric up to additive error. }
%{\color{blue}I still don't understand. Could you provide a complete argument please ?}\ywy{Sorry, I will try to make it more precise later on. After some thought, I can only prove that for $r$-Green metric, loxodromic elements are strongly contracting: any metric ball away from the axis has uniform bounded projection. If the $r$-Green metric is geodesic, then we can prove an equivalent version of strong contracting property: if a $r$-Green metric geodesic is outside a uniform neighborhood of the axis, then it projects to a bounded set to the axis. From this we can show that Green metric is growth tight. But this does not contradict to the non-growth tightness of the linear combination $\mathfrak d_r$ of Green metric and word metric. So I could not obtain that the $r$-Green metric is not geodesic by using the contradiction to the following corollary. I will think more if I can find other solution. }
%{\color{blue}Okay but this is very interesting to keep in mind ! Because I think it's not known if the Green metric itself is geodesic.}

\begin{proof}[Proof of Corollary \ref{drnotgeodesic}]
Assume on the contrary that $\mathfrak d_r$  is a coarsely geodesic distance for $r>r_\star$.
By Lemma \ref{contractingdrmetric}, the axis of any loxodromic element satisfies the bounded image property for coarse geodesics, so any loxodromic element is contracting in the sense of \cite{YangTight}. As the action on $\Gamma$ is co-compact, the same argument as in \cite{YangTight} holds verbatim by replacing word geodesics with coarse $\mathfrak d_r$-geodesics and we can show that $\mathfrak d_r$ on $\Gamma$ is growth tight. However, this contradicts Theorem \ref{transitionphasegrowthtightness}. Thus the corollary follows.
\end{proof}

In \cite[Question 1]{ACT}, Arzhantseva-Cashen-Tao asked whether growth tightness is invariant among cocompact actions   on geodesic metric spaces. Cashen-Tao \cite{CaTao16} showed the first examples  of product groups with growth tightness for one generating set but not for another generating set.
Examples of non-growth tight relatively hyperbolic groups with   non-cocompact actions were already considered in \cite[Obs. 7.9]{ACT}.
As the action   is  not-cocompact, the induced pseudo-distance on the group pulled back from the action is not quasi-isometric to the word distance.
It is natural to  ask the following variant of Arzhantseva-Cashen-Tao's question for relatively hyperbolic groups about quasi-isometry invariance of growth tightness :
if a proper (pseudo-)distance $d$ on $\Gamma$ is quasi-isometric to the word distance, does the growth tightness for $d$ hold ?
In our last   corollary, we produce examples  of non-geodesic distances  on relatively hyperbolic groups $\Gamma$ that answer negatively this question.
\begin{cor}\label{corgrowthtightness}
There exists a relatively hyperbolic group with a proper left invariant distance quasi-isometric to the word distance which does not have the growth tightness property.
\end{cor}

\appendix
\section{Probability that the branching random walk visits partial cones}
%\ywy{I'm adapting the argument from our \cite{DWY} to prove Proposition~\ref{proppartialcone}, the above Longmin's estimates. Up to me, everything was contained in our previous paper, but is arranged here in a slightly different order. I highlighted below some geometric facts which may require further clarification.  }

 We consider a relatively hyperbolic group $\Gamma$.
Our goal is to prove Proposition~\ref{proppartialcone}.
The proof basically consists of a reorganization of arguments of \cite{DWY}.
We first recall some notations and definitions.

Given a subset $A$ of $\Gamma$, we write $G(x,y;A|r)$ for the Green function restricted to paths staying in $A$, expect maybe the first and last point.
That is,
$$G(x,y;A|r)=\sum_{n\geq 0}\sum_{z_1,...,z_{n-1}\in A}r^n\mu(x^{-1}z_1)\mu(z_1^{-1}z_2)...\mu(z_{n-1}^{-1}y).$$

Fix  $C>0$ and $x\in \Gamma$.
The \textit{$C$-partial cone} $\Omega(x,C)$ consists  of points $z\in G$ such that $x$ is within $C$ of an $(\eta, L)$-transition point on the geodesic $[e,z]$. Let $C>0$ be any sufficiently large constant given by \cite[Lemma 2.9]{DWY} so that the relative thin triangle property holds for $(\eta, L)$-transition points : for every triple $(x,y,z)$ of $\Gamma$, any $(\eta,L)$-tranition point on $[x,y]$ is within $C$ of either an $(\eta,L)$-transition point on $[x,z]$ or an $(\eta,L)$-transition point on $[y,z]$.

Let $B([x,z])$ be the ball centered at the middle point of $[x,z]$  of radius ${d(x,z)/2}$. Define $U(x)$ to be the union of  the balls $B([x,z])$ for all geodesics $[x,z]$ between $x$ and $z\in \Omega(x,C)$. That is,
$$
U(x):= \bigcup\left\{B([x,z]): \forall [x,z], \forall z\in  \Omega(x,C)\right\}. 
$$
It is clear that  $\Omega(x, C)$ is contained in $U(x)$.
 
Fix $\epsilon\in (0, 1/2)$.
Let $U_\epsilon(x)$ be the set of points $z\in U(x)$ such that $[x,z]$ contains a transition point $w$ being at distance at least $\epsilon d(x,z)$ to one of the endpoints:  $$ \max\{d( w, x), d(w,z)\}\ge \epsilon d(x,z).$$  
For any $m\ge 1$, let $U_{\epsilon}(x,m)$ be the set of elements  $z\in U_{\epsilon}(x)$ such that $d(x,z)\ge m$. 
 
% By triangle inequality, we see that any (connected) subpath is also $\tau$-taut.
 
We now consider a finitely supported symmetric and admissible probability measure $\mu$ on $\Gamma$ and a probability measure $\nu$ on $\mathbb N$.
We denote by $\mathrm{BRW}(\Gamma,\nu,\mu)$ the branching random walk associated with $\nu$ and $\mu$.
In what follows, we shall often use the following estimates proved in~(\ref{lowerboundquasiisometry}). There exists  $\alpha>0$ such that for any $x\in \Gamma$, for every $r\geq 1$,
\begin{align}\label{GrLBD}
G(e, x|r)\ge \mathrm{e}^{-\alpha |x|}.
\end{align}
Also, by Lemma~\ref{purelyexponentialgrowth} there exists $c_1>0$ such that for any $n\ge 1$,
\begin{align}\label{SnGrowth}
\sharp S_n\le c_1 e^{vn}
\end{align}

\begin{lem}\label{TraceAvoidTransBall}
For any $\epsilon\in (0,1/2)$, there exists $\kappa_0>0$ such that for every $\kappa\ge \kappa_0$, the following holds.
For all but finitely many $x\in \Gamma$: the  following event   
\begin{align*}
E_1:=\{\mathrm{BRW}(\Gamma,\nu,\mu)~\text{first enters}~U(x)~\text{at a point}~z\in U_{\epsilon}(x,\kappa \log |x|)\} 
\end{align*} has probability at most $G(e,x|r)$.  %the   trace     of BRW intersects $U_\epsilon(x)$, then it does not intersect the subset $U_{\epsilon}(x,\kappa \log |x|)$. 

\end{lem}
 
\begin{proof}
Let us freeze all particles of $\mathrm{BRW}(\Gamma,\nu,\mu)$ when the event $E_1$ happens, and denote    by $\mathcal Z$  the   collection of  frozen particles. Set $m:=\kappa \log n$ for simplicity, where $n=|x|$. Then for $z \in \mathcal{Z}$ we have 
\begin{enumerate}
    \item 
     $d(x,z)\ge m$, 
    \item 
      $\max\{d(y,x), d(y,z)\}>\epsilon d(x,z)-3C$ where $y$ is an $(\eta,L)$-transition point on $[e,z]$ given by \cite[Lemma~6.6]{DWY}.  
\end{enumerate}    
As the genealogy path from $e$ to $z$ does not intersect $B(y,\epsilon d(x,z)-3C)$, the expected number of  particles frozen  at $z\in U_\epsilon(x)$   is upper bounded by
$$
G(e,z; [U_\epsilon(x)]^c|r) \le G(e,z; [B(y,\epsilon d(x,z)-3C)]^c|r) \le e^{-   e^{\delta [\epsilon d(x,z)-3C]}}. 
$$  where $\delta=\delta(\eta, L)$ be given by \cite[Proposition~3.5]{DG21}.
As a consequence, there exist $\epsilon_1=\epsilon_1(\epsilon,\delta,v)$ and $n_0>0$ such that for any $m>\kappa\log n_0$, we have by (\ref{SnGrowth}) that
$$
\mathbf E[\sharp \mathcal Z]\le \sum_{z\in U_\epsilon(x,m)} G\bigl (e,z; [B(y,\epsilon d(x,z)-3C]^c\big|r\bigr) \le \sum_{k=m}^{\infty} c_1\cdot \mathrm{e}^{kv-   e^{\delta(\epsilon k-3C)}} \le e^{-e^{\epsilon_1 m}}.
$$
Choose $\kappa$ so that $n^{\kappa \epsilon_1} >\alpha n$ holds for any $n>n_0$.
Then,
$$
 \mathbf P(E_1) \le  \mathbf E(\sharp  \mathcal Z\ge 1)  \le e^{-n^{\epsilon_1\kappa}} \le  \mathrm{e}^{-\alpha n} \le G(e,x|r) 
$$ 
where the last inequality uses (\ref{GrLBD}).
\end{proof}

Similarly, we prove the following.
\begin{lem}\label{TraceAvoidTransBall2}
For every $K\geq 1$ and $\hat C\geq 0$, there exists $\kappa_0>0$ such that for all $\kappa\geq \kappa_0$, the following holds.
For any sufficiently large $n\ge 1$, the following event 
\begin{align*}
E_2:=\{&\mathrm{BRW}(\Gamma,\nu,\mu)~\text{eventually visits a point}~z\in \Omega(x,\hat C)~\text{with}~d(e,z)\leq Kd(e,x)\\
&\hspace{.3cm}\text{but without entering}~B(x,\kappa \log |x|)~\text{where}~|x|\ge n\}     
\end{align*} 
has probability at most $G(e,x|r).$ 
\end{lem}

\begin{proof}
We freeze particles when the event $E_2$ happens and denote by $\mathcal Z$ the set of frozen particles. By \cite[Lemma~2.11]{DWY}, if $y\in [e,z]$ is a transition point $\hat C$-close to $x$,    the expected number of  particles frozen at $z\in B(e,K|x|)$   is upper bounded by
$$
G(e,z; [B(y,\kappa \log |x|-\hat C)]^c|r) \le e^{-   n^{\delta \kappa }}
$$  
where $c_2$ depends on $\hat C, \delta$.
Thus,   we have
$$
\mathbf E[\sharp \mathcal Z]\le    \sum_{m\ge n}\sum_{z\in S_{Km}} G(e,z; [B(y,\kappa \log m-\hat C)]^c|r) \le  \sum_{m\ge n} c_1 e^{vKm} e^{-   c_2m^{\delta\kappa}}.   
$$
Choose $\kappa, n_0>0$ so that $c_2 m^{\delta\kappa}-vKm>\alpha m$ holds for any $m\ge n_0$. The conclusion follows again.
\end{proof}

\begin{lem}\label{exponentialmoments}
For any $K>1$, there exist $\epsilon_0, \kappa_0$ such that for all $\epsilon\leq \epsilon_0$ and $\kappa\geq \kappa_0$, the following holds.
There exists $c<0$ such that for all but finitely many  $x\in \Gamma$, the following event   
\begin{align*}
E_3:=\{&\mathrm{BRW}(\Gamma,\nu,\mu)~\text{first enters}~U(x)~\text{at a point}\\
&\hspace{.2cm}z\in U(x)\setminus U_{\epsilon}(x,\kappa \log |x|)~\text{with}~|z|\ge K|x|\}    
\end{align*}
has probability at most $c\cdot G(e,x|r)$.

\end{lem}

\begin{proof}
Let $V$ be the set of $z\in U(x)\setminus U_{\epsilon}(x,\kappa \log |x|)$ satisfying  $|z|\ge K|x|$. By definition,  
\begin{enumerate}
    \item 
     either $d(z, x)\le  \kappa \log |x|$,  
    \item
    or the $[\epsilon, 1-\epsilon]$-percentage  of $[x,z]$ does not contain any $(\eta,L)$-transition point.  
\end{enumerate}
If   $K $ and $ \kappa$ are fixed, noticing that 
$$(K-1)|x|\le  d(x,z)\le  (1+1/K)|z|,$$
the    case (1) is impossible for  sufficiently large $|x|$. Thus, it suffices to consider the case (2). Set $K_1=\epsilon(1+1/K)$ and $K_2=(1-2\epsilon)(1-1/K)$. By \cite[Lemma~6.1]{DWY}, there exist a unique   coset $P_z\in \mathbb P$ such that if $y_1,y_2$ are the entrance and exit points of  $[x,z]$ in $N_\eta(P_z)$, then
\begin{align}\label{y1y2Def}
\max\{d(x,y_1), d(y_2,z)\}\le \epsilon d(x,z)\le \epsilon(1+1/K)|z|\le K_1|z|,    
\end{align}
and  
\begin{align}\label{y1y2EQ}d(y_1,y_2)\geq (1-2\epsilon)d(x,z) \ge (1-2\epsilon)(1-1/K)|z|\ge K_2|z|.
\end{align}
%we take $\epsilon>0$ small enough so that \begin{align}\label{y1y2EQ}
%d(y_1,y_2)\ge \max\{d(x,y_1), d(y_2,z)\}.
%\end{align}

Before moving on, we need the following facts about $y_1$.
By relative thin triangle property, there exists  a  constant $\hat C_1$ depending only on $C$ so that $x$ is $\hat C$-close to a transition point on $[e,y_1]$. Moreover, there exists $\hat C_2$ depending on $C, \eta$ so that the projection $\pi_{N_{\eta}(P_z)}(e)$ of $e$ to $N_{\eta}(P_z)$ is within $\hat C_2$ of $y_1$ and   $y_1$ is $\hat C_2$-close to a transition point on $[e,w]$ for any $w\in N_\eta(P_z)$.
For given $P\in \mathbb P$, let $P(y_1)$ denote the set of $w\in N_\eta(P)$ with $d(w,y_1)\leq d(e,y_1)$, where $y_1$ is $\hat C_2$-close to $\pi_{N_\eta(P)}(e)$. 

Consider first the sub-event $E_{30}$ of $E_3$, where $\mathrm{BRW}(\Gamma,\nu,\mu)$ enters $N_\eta(P_z)$ at a point $w\in P_z(y_1)$.
Thus, $d(e,w)\le d(w,y_1)+ d(e,y_1)\le 2|y_1|$.
The same proof of \cite[Lemma 6.6]{DWY} implies that $B(y_1,  d(y_1,x)-3C)$ is contained in $U(x)$. Thus, the particle does not visit $B(y_1,  d(y_1,x)-3C)$. Assume first that $|y_1|>2|x|$. Note that there exists $n_0, \kappa_0>0$ so that for any $|x|\ge n_0$, we have    $$\kappa_0 \log |y_1| \le |y_1|-|x|-3C\le  d(y_1,x)-3C. $$
In particular,  the branching random walk does not visit $B(y_1,\kappa_0 \log |y_1|)$ before arriving at $w$. Now, if  $|y_1|\le 2|x|$, then $d(e,w)\le  2|y_1|\le 4|x|$. By definition of $E_3$, the branching random walk does not visit $B(x,\kappa_0 \log |x|)$.
In summary,  this sub-event $E_{30}$ is included   into the event $E_2$ in   Lemma~\ref{TraceAvoidTransBall2} with constants $\hat C$ and $K=4$, so there exists $\kappa_0\ge \kappa>0$ so that the probability of   $E_{30}$  is at most $G(e,x|r)$. 

Now, it remains  to  consider the particles of $\mathrm{BRW}(\Gamma,\nu,\mu)$ in the event $E_3$ that do not   enter $N_\eta(P_z)$ at some point $w\in P_z(y_1)$.
Then, we have the following two sub-events denoted by $E_{31}$ and $E_{32}$ respectively: the particles either do not visit $N_\eta(P_z)$ at all or  do visit   $N_\eta(P_z)$ but at a first entrance point $w$ not in $P_z(y_1)$. 
Let us denote by $W$ the set of points $w\in N_\eta(P)\setminus P(y_1)$  for all $P\in \mathbb P$ where $y_1$ is $\hat C_2$-close to $ \pi_{N_\eta(P)}(e)$ and $|y_1|>|x|$.
In the first case, we freeze particles when they first enter the set $V$.
In the second case, we freeze particles when they first enter the set $W$.
We denote by $\mathcal Z_{31}$ and  $\mathcal Z_{32}$ respectively the sets of frozen particles.
We have
\begin{align}
\label{Z31} \mathbf E[\sharp \mathcal{Z}_{31}]&\leq \sum_{z\in V}  G\big(e,z;[N_\eta(P_z)]^c\big|r\big),\\
\label{Z32} \mathbf E[\sharp \mathcal{Z}_{32}]&\leq \sum_{w\in W} G\big(e,w;[N_\eta(P_z)]^c\big|r\big),
\end{align} 
 %As above, $W\subset \Omega(x, \hat C)$.
%% \lw{Here I do no quite understand.  We have considered the case that the particle enters $N_{\eta}(P_z)$ at a point $w$ with $d(w, y_1) \leq d(e, y_1)$ and in above inequality we estimate those particles do not visit $N_{\eta}(P_z)$.  But how about the particles first visit a point $w \in N_{\eta}(P) \setminus U(x)$ and then visit $z$? }

We first bound $\mathbf E[\sharp \mathcal{Z}_{32}]$ in~(\ref{Z32}). Recall that  $\pi_{N_{\eta}(P_z)}(e)$  is $\hat C_2$-close to  $y_1$.  If  $\mathrm{BRW}(\Gamma,\nu,\mu)$     enters  $N_\eta(P_z)\setminus P_z(y_1)$ at a point $w$, by  \cite[Lemma~2.12]{DWY}, for every $M\geq 0$, there exists $\eta_0$ such that for all $\eta\geq \eta_0$,
$$G\big(e,w;[N_\eta(P_z)]^c\big|r\big)\leq \mathrm{e}^{-M d(\pi_{N_{\eta}(P_z)}(e),w)}\leq c_0 \mathrm{e}^{-M d(w,y_1)}$$ for some $c_0=c_0(\hat C_2)>0$.  Using (\ref{GrLBD}), we first sum up over $y_1$ with $|y_1|>|x|$ and then $w\in P_z(y_1)$ with $d(y_1, w)>d(e, y_1)$:
$$\mathbf E[\sharp \mathcal{Z}_{32}]\leq \sum_{n\ge |x|}  c_0 \cdot \mathrm{e}^{vn} \sum_{m\ge n} \mathrm{e}^{(v- M)m}.$$ 
Choosing $M>2v+\alpha$ that is is independent of $\eta$, we have  that for $\eta\geq \eta_0$, 
$$
\mathbf E[\sharp \mathcal{Z}_{32}]\leq c\mathrm{e}^{-\alpha |x|} \le c G(e,x|r)
$$
where $c$ depends on $c_0$, thus on $\eta$.

We are left to bound $\mathbf E[\sharp \mathcal{Z}_{31}]$ in~(\ref{Z31}).    
As the support of $\mu$ is finite,  we can replace each edge in the geodesic from $z$ to $y_2\in N_\eta(P_z)$ by a $\mu$-trajectory with uniformly bounded length. Moving possibly  the endpoint $y_2$ up to a   bounded distance depending on $\supp(\mu)$, this  produces   a trajectory outside $N_\eta(P_z)$ for the $\mu$-random walk  from   $z$ to $y_2$ so that its   length is bounded above by a  linear function of   $d(y_2,z)$.
This implies the existence of a  positive $\beta$ independent on $z, y_1,y_2$ and $\eta$ such that 
$$G\big(z,y_2;[N_\eta(P_z)]^c\big|r\big)\geq \mathrm{e}^{-\beta d(y_2,z)}\ge \mathrm{e}^{-\beta K_1 |z|}.$$
Taking into account that
$$
G\big(e,z;[N_\eta(P_z)]^c\big|r\big) \cdot G\big(z,y_2;[N_\eta(P_z)]^c\big|r\big) \le G(e,e|r) \cdot G\big(e,y_2;[N_\eta(P_z)]^c\big|r\big) 
$$
we obtain
\begin{equation}\label{proofexponentialmoments}
    G\big(e,z;[N_\eta(P_z)]^c\big|r\big)\leq   G\big(e,y_2;[N_\eta(P_z)]^c\big|r\big)\cdot \mathrm{e}^{\beta K_1 |z|},
\end{equation}
As above, the projection $\pi_{N_{\eta}(P_z)}(e)$  has a distance at most $\hat C_2$ depending on $\eta$ to $y_1$. By  \cite[Lemma~2.12]{DWY}, for every $M\geq 0$, there exists $\eta_0$ such that for all $\eta\geq \eta_0$,
$$G\big(e,y_2;[N_\eta(P_z)]^c\big|r\big)\leq \mathrm{e}^{-M d(\pi_{N_{\eta}(P_z)}(e),y_2)}\leq c_0 \mathrm{e}^{-M d(y_2,y_1)}\le c_1\mathrm{e}^{-K_2 M|z|}$$ for some $c_0=c_0(\eta)>0$.
Summing over $z\in V$ with $|z|>K|x|$, choosing $M>0$ so that $M K_2>\beta K_1+v+\alpha$, we have  by~(\ref{proofexponentialmoments}) and (\ref{GrLBD})
that for $\eta\geq \eta_0$,
$$\mathbf E[\sharp \mathcal{Z}_{31}]\leq \sum_{n\ge K|x|}  c_0c_1 \cdot \mathrm{e}^{(v+\beta K_1- M K_2)n}\le c\mathrm{e}^{-\alpha |x|} \le c G(e,x|r).$$ 
where $c$ depends on $c_0, c_1$.
The lemma is proved.
\end{proof}

%We now consider a tight augmentation of the trace of the branching random walk.
%\begin{defn}\label{TightAugDef2}
%Let $X\subset \Gamma$ and $X_n=X\cap S_n$. Denote by ${\rm Tran}_\epsilon(X_n)$ the union of transition points $x$ on  $[e,z]$ for some $z\in X_n$ and $|x|\le \epsilon n$. Let $\widetilde{X}_\epsilon$ be the union of all $X_n$ with $ {\rm Tran}_\epsilon(X_n)$ over $n\ge 1$. 
%\end{defn}

%\begin{lem}\label{ProbShadowLem}
%There exist  $\beta>0$ with the following property.
%For any $x\in \Gamma$, the probability of the BRW visiting $\Omega(x,C)$ is at most $|x|^\beta G_r(e,x)$ for $r\le R.$ 
%\end{lem}
We can now finish the proof of Proposition~\ref{proppartialcone}.
\begin{proof}[Proof of Proposition~\ref{proppartialcone}]
Using the thin-triangle property, we see that if  $z$ is contained in $U(x)\setminus U_{\epsilon}(x,\kappa \log |x|)$, then  $z\in \Omega(x, \hat C)$ for some uniform $\hat C$.

Fix any $K>1$. Let $\epsilon$ be small enough and $\kappa$ be large enough such that the conclusions of Lemmas~\ref{TraceAvoidTransBall},~\ref{TraceAvoidTransBall2} and~\ref{exponentialmoments} hold. Using (\ref{GrLBD}), the probability that the branching random walk visits $B(x, \kappa \log |x|)$ is bounded by $$\sum_{z \in B(x, \kappa \log |x|)} G(e, z|r)\le c\mathrm{e}^{v\kappa \log |x|} G(e,x|r) \mathrm{e}^{\alpha \kappa \log |x|}\le |x|^{\beta}G_r(e,x)$$
for some $\beta$ depending on $\alpha, \kappa$ and $v$.

Now assume that the branching random walk visits  $U(x)$ through paths outside $B(x, \kappa \log |x|)$. If the events $E_1$ and $E_3$ do not happen,  then  the event  $E_2$ happens: $\mathrm{BRW}(\Gamma,\nu,\mu)$ first visits $U(x)$  {at a point} $z\in U(x)\setminus U_{\epsilon}(x,\kappa \log |x|)$ with $|z|< K|x|$, without  entering $B(x, \kappa \log |x|)$.  
Therefore,
\begin{align*}
  & \mathbf{P}(\mathrm{BRW}(\Gamma,\nu,\mu)\text{ visits } \Omega(x, C)) \leq \mathbf{P} (\mathrm{BRW}(\Gamma,\nu,\mu)\text{ visits } U(x)) \\
  \leq& \mathbf{P}(\mathrm{BRW}(\Gamma,\nu,\mu)\text{ visits }B(x, \kappa \log |x|))+
  \mathbf{P}(E_1\cup E_2\cup E_3) \\
  \leq& C(1+|x|^{\beta})G(e,x|r) \qedhere
\end{align*} 
\end{proof}

\bibliographystyle{alpha}
\bibliography{Greengrowth}

\end{document}